\documentclass[hidelinks,onefignum,onetabnum]{siamart220329}
\usepackage{amsmath}
\usepackage{bm}
\usepackage[margin=1.3in]{geometry} % 将页边距设置为1英寸
\usepackage{lipsum}
\usepackage{amsfonts}
\usepackage{graphicx}
\usepackage{epstopdf}
\usepackage{algorithmic}
\usepackage{makecell}
\usepackage{epsfig}
\usepackage{color}
\usepackage{amssymb}
\usepackage{amsmath}
\usepackage{subfigure} 
\usepackage{amsmath}
\usepackage{hyperref}
\usepackage{multirow}
\usepackage{graphicx}
\usepackage{float}
\usepackage{tabularx} % 在导言区加载宏包
\allowdisplaybreaks

\ifpdf
  \DeclareGraphicsExtensions{.eps,.pdf,.png,.jpg}
\else
  \DeclareGraphicsExtensions{.eps}
\fi

\newsiamremark{rmk}{Remark}
\newsiamremark{hypothesis}{Hypothesis}
\crefname{hypothesis}{Hypothesis}{Hypotheses}
\newsiamthm{claim}{Claim}
\newsiamthm{lem}{Lemma}

\headers{Optimally Convergent Parallel Splitting Algorithm for MPET}{J. Zhao, H. Chen, M. Cai and S. Sun}

\title{An Optimally Convergent Parallel Splitting Algorithm for the Multiple-Network Poroelasticity Model
\thanks{Submitted to the editors.
}}

\author{Jijing Zhao\thanks{School of Mathematical Sciences, Tongji University, Shanghai, China 
		(\email{jijingzhao@tongji.edu.cn}).}
	\and Huangxin Chen\thanks{School of Mathematical Sciences and Fujian Provincial Key Laboratory on Mathematical Modeling and High Performance Scientific Computing, Xiamen University, Xiamen, Fujian, China (\email{chx@xmu.edu.cn}).}
	\and Mingchao Cai\thanks{Department of Mathematics, Morgan State University, Baltimore, MD 21251, USA (\email{mingchao.cai@morgan.edu}).}
	\and Shuyu Sun\thanks{School of Mathematical Sciences, Tongji University, Shanghai, China; China University of Petroleum (East China), Qingdao, Shandong, China 
		(\email{suns@tongji.edu.cn}). Corresponding author.} 
	}

\usepackage{amsopn}

\ifpdf
\hypersetup{
  pdftitle={An Optimally Convergent parallel splitting Algorithm},
  pdfauthor={J. Zhao, H. Chen, M. Cai, and S. Sun}
}
\fi

\begin{document}
\small
\maketitle

% REQUIRED
%\begin{minipage}{0.9\textwidth}
\begin{abstract}
This paper presents a novel parallel splitting algorithm for solving quasi-static multiple-network poroelasticity (MPET) equations. By introducing a total pressure variable, the MPET system can be reformulated into a coupled Stokes-parabolic system. To efficiently solve this system, we propose a parallel splitting approach. In the first time step, a monolithic solver is used to solve all variables simultaneously. For subsequent time steps, the system is split into a Stokes subproblem and a parabolic subproblem. These subproblems are then solved in parallel using a stabilization technique. 
This parallel splitting approach differs from sequential or iterative decoupling, significantly reducing computational time. The algorithm is proven to be unconditionally stable, optimally convergent, and robust across various parameter settings. These theoretical results are confirmed by numerical experiments. We also apply this parallel algorithm to simulate fluid-tissue interactions within the physiological environment of the human brain.
\end{abstract}
%\end{minipage}

% REQUIRED
\begin{keywords}
			 Multiple-network poroelasticity, Biot model, unconditionally energy stable, decoupled algorithm.
\end{keywords}

% REQUIRED
\begin{MSCcodes}
65N30, 65M12, 65M60
\end{MSCcodes}

\section{Introduction}
	\label{intro}
This paper addresses the quasi-static multiple-network poroelasticity (MPET) equations \cite{meptori,lee_mixed_2019}, a generalization of the Biot model \cite{BM55, BM72}, which has widespread applications in both geological and biological fields \cite{WU2022}. Let \(\Omega \subseteq \mathbb{R}^d\) with \(d = 2,3\) denote a bounded polygonal domain with boundary \(\partial \Omega\), and let \(\Omega_T = \Omega \times [0, T]\) represent the unsteady domain over the time interval \([0, T]\). Given a specified number of networks \(A \in \mathbb{N}\), our objective is to find the displacement \(\boldsymbol{u}\) and pressures \(p_j\) for \(j = 1, \ldots, A\) that satisfy the governing equations
%In this paper, we consider the quasi-static multi-network poroelasticity (MPET) equations, which in fact represent a more general form of the Biot model \cite{lee_mixed_2019}. Let $\Omega  \subseteq \mathbb{R}^d ,(d=2,3)$ be a bounded polygonal domain with boundary $\partial \Omega$. We set $\Omega_T=\Omega \times (0,T)$ for unsteady domain with terminal time $T$. For a given number of networks $A \in \mathbb{N}$ , find the displacement $\boldsymbol{u}$ and the network pressures $p_j$ for $j = 1, \cdots, A$ such that
\begin{equation} \label{govern1}
	\begin{aligned}
		-\nabla \cdot \boldsymbol{\sigma}(\boldsymbol{u}) + \sum_{j=1}^{A} \alpha_j \nabla p_j = \boldsymbol{f} & \quad \text{in } \Omega_T, \\
		\partial_t(c_j p_j + \alpha_j \nabla \cdot \boldsymbol{u}) + \nabla \cdot (-{\boldsymbol{K}_j}\nabla p_j ) + S_j= q_j& \quad \text{in } \Omega_T, \quad 1 \leqslant j \leqslant A,
	\end{aligned}
\end{equation}
where the isotropic stiffness tensor
\begin{equation*}
	\begin{aligned}
		\boldsymbol{\sigma}(\boldsymbol{u}) := 2 \mu {\varepsilon}(\boldsymbol{u}) + \lambda \nabla \cdot \boldsymbol{u} \boldsymbol{I} \quad \mbox{with} \quad
         {\varepsilon}(\boldsymbol{u}) := \frac{1}{2}(\nabla \boldsymbol{u} + (\nabla \boldsymbol{u})^T).
	\end{aligned}
\end{equation*}
Here, $\mu$ and $\lambda$ are the standard nonnegative Lam\'{e} parameters, and $\boldsymbol{I}$ represents the identity tensor. The transfer terms \( S_j \) represent fluid exchange from network \( j \) to other fluid networks, based on the pressure differentials between networks. Specifically, \( S_j \) is expressed as
\begin{equation}\label{def_S}
	S_j = S_j\left(p_1, \ldots, p_A\right) = \sum_{i=1}^A s_{j \leftarrow i} \left(p_j - p_i\right),
\end{equation}
where $s_{j \leftarrow i}$ denotes the nonnegative transfer coefficients for $i, j = 1, \ldots, A$. The transfer coefficients \( s_{j \leftarrow i} \) are symmetric \( s_{j \leftarrow i} = s_{i \leftarrow j} \), although they can take on arbitrary values. Additionally, each network \( j \) is associated with a storage coefficient \( c_j \geq 0 \), a Biot-Willis coefficient \( \alpha_j \in (0, 1] \), and a hydraulic conductivity tensor \( \boldsymbol{K}_j = \kappa_j \boldsymbol{I} \). The (effective) stress tensor is \( \boldsymbol{\sigma}(\boldsymbol{u}) \), while the total stress tensor is denoted by \( \hat{\boldsymbol{\sigma}} (\boldsymbol{u}, p) := \boldsymbol{\sigma}(\boldsymbol{u}) - \sum_{j=1}^{A} \alpha_j p_j \boldsymbol{I} \). When \( A = 1 \), the MPET reduces to the well-known quasi-static Biot's equations \cite{BM55, BM72, AZ84}. 
The case where \( A = 2 \) is the Barenblatt-Biot model \cite{a2biot}. 
Thus, the MPET equations can be seen as a generalization of Biot's equations. 
The numerical discretization of MPET faces similar challenges, inheriting the Poisson locking and pressure oscillation phenomena from single-network poroelasticity \cite{yi_study_2017}. 
Moreover, due to the increased number of parameters and unknowns in MPET, more efficient and stable numerical discretization methods warrant further exploration.
%As such, the MPET equations serve as a generalization of Biot's equations, and numerical discretizations face significant challenges due to the variety of parameters and unknowns.

%The MPET model is essentially a coupled system of elliptic-parabolic equations, derived from the principles of mass and momentum balance in a porous, linearly elastic medium that contains \( A \) segregated viscous fluid networks coexisting within the same domain \( \Omega \).
%This interaction between elasticity and fluid dynamics has significant applications in geophysics and biology, including modeling fluid flow in fractured geological porous media \cite{WU2022,al2022} and small-scale blood flow in human tissues \cite{lee_mixed_2019}. 
%Known as the multi-continuum poroelastic media flow model in the oil and gas industry, it is crucial for developing shale and tight reservoirs \cite{WU2022,al2022}.

 %The numerical research on Biot's model is extensive, including continuous and discontinuous finite element methods. The solution methods can be classified into fully coupled methods and decoupled methods.
 
The research on Biot's model is extensive and covers a variety of numerical solution techniques \cite{phi09,yi_study_2017}. %including continuous and discontinuous finite element methods. 
%such as continuous finite element methods, discontinuous Galerkin methods, and weak Galerkin methods.
These numerical methods can be classified into fully coupled algorithms and decoupled algorithms.
The fully coupled algorithm has the advantage of energy stability, but its computational cost is high, especially for large coupled systems like MPET. Moreover, the coupled algorithm is not practical for existing solvers.
To address this, researchers have explored decoupled algorithms, but stability issues remain a critical challenge. The decoupled methods for poroelasticity are sensitive to time step and material parameters \cite{   NMPDE,altmann,cai23jsc,FX}, which can affect stability and convergence. These decoupled methods are generally classified into iterative and sequential approaches. Iterative methods, such as drained, undrained, fixed-strain, and fixed-stress splits \cite{kim1, kim2, MW2013,gu23}, decouple the system at each time step and iterate until a solution is reached. {These methods, especially the fixed-stress splitting scheme, can be highly efficient when the optimal stabilization factor is used—this factor depends on the mechanical properties of the porous medium and the coupling coefficient \cite{Storvik2019}. Moreover, in \cite{BORREGALES20191466}, a partially parallel-in-time version of the fixed-stress method is proposed, treating time as an additional direction for parallelization. This approach can also be extended to the MPET framework.}
Sequential decoupled methods rely on numerical solutions from previous time steps to decouple terms, and such methods do not require iteration. 
For the two-field Biot model, sequential methods enforce constraints on parameters or time steps. Optimal convergence is achieved under weak coupling conditions, which are dependent on the parameters, and the optimal convergence order is proven by constructing a delay system \cite{semi-explicit_2021}.
%For the two-field Biot model, sequential methods enforce constraints on parameters or time steps, with convergence achieved under weak coupling conditions \cite{semi-explicit_2021, semi-explicit_2023}. 
The numerical schemes do not maintain energy stability if the weak-coupling conditions are not satisfied. To enhance stability, stabilizing terms are added for sequential methods, ensuring unconditional stability \cite{jscse,CMAse}.

Despite progress in decoupled techniques for the Biot model, their application to the quasi-static multiple-network poroelasticity has been limited. Lee \cite{lee2019} explored partitioned numerical methods for MPET equations using sequential subproblem-solving techniques. Preconditioners \cite{precon} and iterative methods \cite{Hongqingguo,altmann} have also been applied to these equations. However, all the aforementioned methods involve iterative or sequential decoupling. 
%\textcolor{red}{Parallel splitting algorithms for poroelasticity have also been developed \cite{NMPDE} and can be extended to the MPET framework.} However, the design of unconditionally energy-stable and optimally convergent parallel splitting schemes remains an open area for future research.
Efficient and stable parallel splitting methods \cite{zhao2025biot}, which enable the parallel computation of subproblems for the MPET model, remain an open area for further investigation.

This paper introduces a parallel splitting algorithm for the MPET model, offering an efficient alternative to traditional iterative and sequential decoupled methods. 
In the first time step, all variables are solved simultaneously. From the second time step onward, the system is split into two subsystems: a generalized Stokes problem and a stabilized parabolic problem, which are solved in parallel at each time level.
The proposed algorithm offers several advantages. First, it is unconditionally stable, with stability unaffected by time-step size or model parameters, enabling robust performance under strong coupling or relatively large time steps. Second, it supports parallel computation, significantly reducing computational time compared to fully coupled, iterative, or sequential methods, while allowing the use of existing solvers for the subsystems. Third, the optimal convergence results have been proven for both time and spatial discretizations, with robustness to all parameters. Notably, the method is locking-free, with error coefficients unaffected by large Lam\'{e} constants or vanishing fluid storage coefficients and permeability. %These features make the algorithm efficient, stable, and well-suited for solving the MPET system.

The paper is organized as follows. Section 2 introduces the mixed formulation of MPET equations and establishes an energy stability. Section 3 presents the unconditionally energy-stable parallel splitting scheme and provides an error analysis. Section 4 includes numerical examples to demonstrate the effectiveness of the parallel method, along with brain simulation experiments based on a four-pressure-network model. Finally, Section 5 concludes the paper.

	\section{Mixed formulation for MPET and its energy law}\label{equation}
To close the system \eqref{govern1}, suitable boundary and initial conditions must be prescribed in this paper.
	\begin{align}
		\hat{\boldsymbol{\sigma}}(\boldsymbol{u}, p) \boldsymbol{n}  &= \boldsymbol{f}_1 \quad \text{on } \Gamma_t := \partial\Omega_{t} \times (0, T),\notag \\
		\boldsymbol{u}&= 0 \quad \text{on } \Gamma_{\bm{u}} := \partial\Omega_{u} \times (0, T),\label{bou}\\
		\boldsymbol{K}_j \nabla p_j \cdot \boldsymbol{n}  &= g_j  \quad \text{on } \Gamma_N := \partial\Omega_{N} \times (0, T),\notag  \\
		p_j &= 0  \quad \text{on } \Gamma_D := \partial\Omega_{D} \times (0, T),\notag 
	\end{align}
where $\boldsymbol{n}$ is the unit outward normal to the boundary, $1 \leqslant j \leqslant A.$ $  \partial\Omega_{t} \cup  \partial\Omega_{u}= \partial\Omega$ and $  \partial\Omega_{N} \cup  \partial\Omega_{D}= \partial\Omega$ with $ |\Gamma_{\bm{u}}| > 0$, $ |\Gamma_D| > 0$.
The initial conditions are given by
\begin{equation}\label{ini}
	\boldsymbol{u}(x,0) = \boldsymbol{u}_0(x), \quad p_j(x,0) = p_{0,j}(x) \quad \text{in } \Omega .
\end{equation}	
In some engineering literature, the second Lam\'e constant $\mu$ is also called the shear modulus and denoted
by $G$, and $B=\lambda + \frac{2}{3}G$ is called the bulk modulus. \( \lambda \), \( \mu \) and $B$ are computed from the Young's modulus \( E \) and the Poisson ratio \( \nu \) by the following formulas
\begin{equation*}
	\lambda = \frac{E \nu}{(1 + \nu)(1 - 2\nu)}, \quad \mu = G = \frac{E}{2(1 + \nu)}, \quad  B = \frac{E}{3(1 - 2\nu)} .
\end{equation*}

For $A$ networks, we denote $\boldsymbol{\alpha}=(\alpha_1,\alpha_2,\cdots,\alpha_A)$, $\boldsymbol{p} = (p_1,p_2,\cdots,p_A )$, and $\boldsymbol{\alpha} \cdot \boldsymbol{p}=\sum_{j=1}^{A}\alpha_jp_j$. Following \cite{lee_mixed_2019, precon}, we introduce the total pressure 
$\xi = \boldsymbol{\alpha} \cdot \boldsymbol{p}-\lambda \nabla \cdot \boldsymbol{u}.$
Inserting the total pressure $\xi$ and its time-derivative into \eqref{govern1}, we obtain an augmented system of quasi-static multiple-network poroelasticity equations: for $t \in (0, T ]$, find the displacement vector field $\boldsymbol{u}$ and the pressure scalar fields $\xi$ and $\boldsymbol{p}$ such that
\begin{equation}
	\begin{aligned}\label{3f}
		-\nabla \cdot ( 2 \mu  \varepsilon (\boldsymbol{u})  -  \xi \boldsymbol{I} ) &= \boldsymbol{f} ,\\
		\nabla \cdot \boldsymbol{u}+\frac{1}{\lambda} \xi -\frac{1}{\lambda} \boldsymbol{\alpha} \cdot \boldsymbol{p}&=0,\\
		c_j  \partial_t p_j + \frac{\alpha_j}{\lambda} \left( \boldsymbol{\alpha} \cdot  \partial_t\boldsymbol{p}  - \partial_t \xi \right) + \nabla \cdot \left(-{\boldsymbol{K}_j}\nabla p_j \right) +S_j&= q_j ,\quad j=1,2,\cdots, A.
	\end{aligned}
\end{equation}
%where $\partial_t \boldsymbol{p} = (\partial_tp_1,\partial_tp_2,\cdots,\partial_tp_A )$.

After the reformulation, the boundary conditions \eqref{bou} and initial conditions \eqref{ini} with $\xi(x,0)=\boldsymbol{\alpha} \cdot \boldsymbol{p}(x,0)-\lambda \nabla \cdot \boldsymbol{u}(x,0)$ can still be applied. Equation \eqref{3f} represents the mixed formulation of the MPET equations, where the system is a coupling of a generalized Stokes system for \( (\boldsymbol{u}, \xi) \) and a parabolic system for \( \boldsymbol{p} \). Furthermore, when \( A = 1 \), this reduces to the well-known three-field Biot model \cite{siam2016}. As \( \alpha_j \to 0 \) or \( \lambda \to \infty \), the coupled equation decouples into two subsystems: a Stokes system and a parabolic system.

To study the weak form and energy analysis of the mixed formulation \eqref{3f}, we give the standard Sobolev spaces $W^{m,s}$ with integer $m,s$. We also use $H^m(\Omega)$ for $W^{m,2}(\Omega)$, and norm $\| \cdot \|_{H^m(\Omega)}$ for $\| \cdot \|_{W^{m,2}(\Omega)}$. $H^m_{0,\Gamma}(\Omega)$ is the subspace of $H^m(\Omega)$ with the vanishing trace on $\Gamma \subset \partial \Omega$.
The $L^2$ inner product is defined as $(v,w) = \int_{\Omega} vwdx$. For vectors $\boldsymbol{v}=(v_1,v_2,\cdots,v_A)$ and $\boldsymbol{w}=(w_1,w_2,\cdots,w_A)$, the $L^2$ inner product is denoted as $(\boldsymbol{v},\boldsymbol{w})=\int \sum_{i=1}^A v_i w_idx$ and the norm of the vector function is $\| \boldsymbol{v} \|_{L^2}  =(\int \sum_{i=1}^A v_i^2dx)^{1/2} $. 

We introduce the following functional spaces, $$\boldsymbol{V} =  (H^1_{0,\Gamma_{\bm{u}}}(\Omega))^d, \quad W = L^2(\Omega), \quad \boldsymbol{M}=M_1 \times M_2 \times \cdots \times M_A,$$ 
where $M_j =  H^1_{0,\Gamma_D}(\Omega)( 1\leqslant j \leqslant A).$
Assuming \( \left|\Gamma_{\bm{u}} \right| > 0 \), Korn's inequality holds in \( \boldsymbol{V} \). Specifically, there exists a constant \( C_k = C_k(\Omega, \Gamma_{\bm{u}}) > 0 \) such that
\begin{equation}\label{corn}
	\|\boldsymbol{u}\|_{H^1(\Omega)} \leq C_k \|\varepsilon(\boldsymbol{u})\|_{L^2(\Omega)}, \quad \forall \, \boldsymbol{u} \in \boldsymbol{V}.
\end{equation}
Additionally, the following inf-sup condition is satisfied: there exists a constant \( \beta_0(\Omega,\Gamma_{\bm{u}}) > 0 \), such that
\[
\sup_{\boldsymbol{u} \in \boldsymbol{V}} \frac{(\operatorname{div} \boldsymbol{u}, q)}{\|\boldsymbol{u}\|_{H^1(\Omega)}} \geq \beta_0 \|q\|_{L^2(\Omega)}, \quad \forall \, q \in L^2(\Omega).
\]

By multiplying test functions and integration by parts, we have the following variational formulation for \eqref{3f}: find $\boldsymbol{u}\in C^1([0, T];\boldsymbol{V}), \xi \in C^1([0, T];$ $W)$ and $\boldsymbol{p} \in C^1([0, T];\boldsymbol{M})$ such that
\begin{align}
	a_1(\boldsymbol{u}, \boldsymbol{v})- b \left(\boldsymbol{v}, \xi \right) &= (\boldsymbol{f}, \boldsymbol{v}) +  \langle \boldsymbol{f}_1, \boldsymbol{v}\rangle ,& \forall \boldsymbol{v} \in \boldsymbol{V},\label{weakforu1}\\
	b \left(\boldsymbol{u}, \phi \right) +a_2(\xi, \phi)  &=\left(\frac{1}{\lambda} \boldsymbol{\alpha} \cdot \boldsymbol{p}, \phi \right), & \forall \phi \in W,\label{weakforu2}\\
	a_3(\partial_t p_j, \psi_j) +   \left(\frac{\alpha_j}{\lambda} \boldsymbol{\alpha} \cdot \partial_t \boldsymbol{p}, \psi_j\right)- \left( \frac{\alpha_j}{\lambda}  \partial_t \xi ,\psi_j \right)  &+  d(p_j, \psi_j)\label{weakforu3} = (q_j, \psi_j) + \langle g_j, \psi_j \rangle ,& \forall \psi_j \in M_j, 
\end{align}
for $ 1 \leqslant j \leqslant A$. Here, the bilinear forms are
$$
\begin{aligned}
	&a_1(\boldsymbol{u}, \boldsymbol{v})=2 \mu \int_{\Omega} \varepsilon(\boldsymbol{u}): \varepsilon(\boldsymbol{v}), \quad & b(\boldsymbol{v}, \phi)=\int_{\Omega} \phi \nabla \cdot \boldsymbol{v}, \\
	&a_2(\xi, \phi)=\frac{1}{\lambda} \int_{\Omega} \xi \phi, \quad &a_3(p_j, \psi_j)=c_j \int_{\Omega} p_j \psi_j,   \\
	& d(p_j, \psi_j)=\int_{\Omega} \boldsymbol{K}_j \nabla p_j \cdot \nabla \psi_j + \int_{\Omega} S_j (\boldsymbol{p}) \psi_j.&
\end{aligned}
$$

The coupled system \eqref{weakforu1}-\eqref{weakforu3} involves complex coupling terms such as \( \frac{1}{\lambda} \left( \boldsymbol{\alpha} \cdot \boldsymbol{p}, \phi \right) \) and \( \left( \frac{\alpha_j}{\lambda} \partial_t \xi , \psi_j \right) \). Solving this coupled system is time- and memory-intensive, so an efficient splitting method that decouples the system into two subsystems is a natural solution. Before introducing the discrete scheme, we first analyze the energy stability of the continuous problem.

\begin{lemma}
	Every weak solution \( (\boldsymbol{u}, \xi, \boldsymbol{p}) \) of problems \eqref{weakforu1}-\eqref{weakforu3} satisfies the following energy law:
	\begin{align*}
		\frac{d}{dt} E(t)  + \frac{1}{2} \sum_{j=1}^A \sum_{i=1}^A \|  s_{j\leftarrow i}^{1/2} \left( p_j-p_i \right)& \|_{L^2(\Omega)}^2+\sum_{j=1}^A \| \kappa_j^{1/2} \nabla p_j\|_{L^2(\Omega)}^2 \\
	&	= \sum_{j=1}^A(q_j, p_j) + \sum_{j=1}^A \langle g_j, p_j \rangle -( \partial_t \boldsymbol{f}, \boldsymbol{u}) - \langle \partial_t \boldsymbol{f}_1, \boldsymbol{u} \rangle .
	\end{align*}
	for \( t \in [0, T] \), where
	\begin{align*}
		E(t) := &\frac{1}{2} \left[2 \mu \lVert \varepsilon(\boldsymbol{u}(t)) \rVert^2_{L^2(\Omega)} + \frac{1}{\lambda} \lVert \boldsymbol{\alpha} \cdot \boldsymbol{p}(t) -\xi(t) \rVert^2_{L^2(\Omega)} + \sum_{j=1}^Ac_j \lVert p_j(t) \rVert^2_{L^2(\Omega)} \right] \\
		&- (\boldsymbol{f}(t), \boldsymbol{u}(t)) - \langle \boldsymbol{f}_1(t), \boldsymbol{u}(t) \rangle .
	\end{align*}
	Moreover, there holds
	\begin{equation*}
		\begin{aligned}
			\| \xi \|_{L^2(\Omega)}  \leq C \left(  2\mu \|   \varepsilon ( \boldsymbol{u}) \|_{L^2(\Omega)} + \| \boldsymbol{f} \|_{L^2(\Omega)} + \| \boldsymbol{f}_1\|_{L^2(\Gamma_t)}  \right),
		\end{aligned}
	\end{equation*}
	where $C$ is a positive constant.
\end{lemma}
\begin{proof}
	Taking $\boldsymbol{v}= \partial_t \boldsymbol{u}$ in \eqref{weakforu1}, differentiating with respect to $t$ in \eqref{weakforu2} and setting $\phi=\xi$, and taking $\psi_j=p_j$ in \eqref{weakforu3}, we have
	\begin{align*}
		2\mu (\varepsilon(\boldsymbol{u}), \varepsilon( \partial_t \boldsymbol{u}))- \left( \xi,\nabla \cdot \partial_t \boldsymbol{u} \right) &= (\boldsymbol{f}, \partial_t \boldsymbol{u}) +  \langle \boldsymbol{f}_1, \partial_t \boldsymbol{u} \rangle ,\\
		\left(	\nabla \cdot \partial_t \boldsymbol{u}  , \xi \right) +\frac{1}{\lambda} \left( \partial_t \xi, \xi \right) -\frac{1}{\lambda}\left( \boldsymbol{\alpha} \cdot \partial_t \boldsymbol{p}, \xi \right) &=0, \\
		\left(c_j \partial_t p_j,p_j \right) +   \left(\frac{\alpha_j}{\lambda} \left( \boldsymbol{\alpha} \cdot \partial_t \boldsymbol{p}-\partial_t \xi \right),p_j \right)  +  (\boldsymbol{K}_j \nabla p_j , \nabla p_j) +\left(S_j,p_j\right)&= (q_j, p_j) + \langle g_j, {p_j} \rangle ,
	\end{align*}
	with $1 \leqslant j \leqslant A$. Summing the equations up, we have
	\begin{equation}
		\begin{aligned}\label{proflem1}
			&	 \left(2\mu \varepsilon (\boldsymbol{u}), \partial_t \varepsilon (\boldsymbol{u})\right)+
			\frac{1}{\lambda} \left( \partial_t \xi, \xi \right) -\frac{1}{\lambda}\left( \boldsymbol{\alpha} \cdot \partial_t \boldsymbol{p}, \xi \right)+ \sum_{j=1}^{A}   \left( c_j \partial_t p_j,p_j \right) + \frac{1}{\lambda}  \left( \boldsymbol{\alpha } \cdot \partial_t \boldsymbol{p}, \boldsymbol{\alpha }\cdot \boldsymbol{p} \right)  \\
			&-\frac{1}{\lambda}  \left( \partial_t \xi, \boldsymbol{\alpha} \cdot \boldsymbol{p}  \right) + \sum_{j=1}^A ( \boldsymbol{K}_j \nabla p_j , \nabla p_j) +\sum_{j=1}^{A} \left(S_j,p_j\right) - \frac{d}{dt}\left[ \left(\boldsymbol{f}, \boldsymbol{u} \right) +  \langle \boldsymbol{f}_1,  \boldsymbol{u} \rangle  \right]   \\
			= &\sum_{j=1}^A(q_j, p_j) + \sum_{j=1}^A \langle g_j, p_j \rangle -( \partial_t \boldsymbol{f}, \boldsymbol{u}) - \langle \partial_t \boldsymbol{f}_1, \boldsymbol{u} \rangle .
		\end{aligned}
	\end{equation}
	Noting the identity,$$ \left(    \boldsymbol{\alpha}  \cdot  \partial_t \boldsymbol{p} -\partial_t \xi ,\boldsymbol{\alpha }  \cdot  \boldsymbol{p} -\xi \right)= ( \boldsymbol{\alpha}  \cdot  \partial_t \boldsymbol{p},  \boldsymbol{\alpha }  \cdot  \boldsymbol{p} ) -  (\partial_t \xi , \boldsymbol{\alpha}  \cdot  \boldsymbol{p} ) -  ( \boldsymbol{\alpha } \cdot   \partial_t \boldsymbol{p},\xi)+\left( \partial_t \xi,\xi\right) .$$
	By the definition of $S_j$ in \eqref{def_S} and $s_{j\leftarrow i}=s_{i\leftarrow j}$, it follows that 
	$$  \sum_{j=1}^A\left(S_j,p_j\right)=\sum_{j=1}^A \sum_{i=1}^A \left( s_{j\leftarrow i} \left(p_j-p_i\right),p_j  \right)=\frac{1}{2} \sum_{j=1}^A \sum_{i=1}^A \|  s_{j\leftarrow i}^{1/2} \left( p_j-p_i \right) \|_{L^2(\Omega)}^2. $$
	Taking above two equations in \eqref{proflem1}, we have the desired result.
%	\begin{equation*}
%		\begin{aligned}
%			&	\frac{ d}{dt} \left[ \mu (\varepsilon(\boldsymbol{u}),  \varepsilon(\boldsymbol{u}))+ \sum_{j=1}^A  \frac{c_j}{2} \left(p_j,p_j \right) +  \frac{1}{2\lambda} \left(   \boldsymbol{\alpha} \cdot \boldsymbol{p} -\xi , \boldsymbol{\alpha} \cdot \boldsymbol{p}-\xi \right) - \left(\boldsymbol{f}, \boldsymbol{u} \right) -  \langle \boldsymbol{f}_1,  \boldsymbol{u} \rangle  \right] \\
%			&+ \frac{1}{2} \sum_{j=1}^A \sum_{i=1}^A \|  s_{j\leftarrow i}^{1/2} \left( p_j-p_i \right) \|_{L^2(\Omega)}^2+\sum_{j=1}^A \| \kappa_j^{1/2} \nabla p_j \|_{L^2(\Omega)}^2  \\
%			= &\sum_{j=1}^A(q_j, p_j) + \sum_{j=1}^A \langle g_j, p_j \rangle -( \partial_t \boldsymbol{f}, \boldsymbol{u}) - \langle \partial_t \boldsymbol{f}_1, \boldsymbol{u} \rangle .
%		\end{aligned}
%	\end{equation*}
	The bound for \( \xi \) is derived using the inf-sup condition and Korn’s inequality \eqref{corn}. Specifically, from \eqref{weakforu1}, the following inequality holds:
	\begin{equation*}
		\begin{aligned}
			\beta_0 \| \xi \|_{L^2(\Omega)} & \leq \sup\limits_{\boldsymbol{v} \in \boldsymbol{V}}  \frac{| \left( \nabla \cdot \boldsymbol{v},\xi  \right)|}{\| \boldsymbol{v} \|_{H^1(\Omega)}}\\
			& =\sup\limits_{\boldsymbol{v}\in  \boldsymbol{V}  }  \frac{\left|  2\mu ( \varepsilon(\boldsymbol{u}),  \varepsilon(\boldsymbol{v})) - (\boldsymbol{f}, \boldsymbol{v}) - \langle \boldsymbol{f}_1, \boldsymbol{v} \rangle \right|}{\| \boldsymbol{v}  \|_{H^1(\Omega)}}\\
			& \leq C_1\left(  2\mu \|  \varepsilon (\boldsymbol{u}) \|_{L^2(\Omega)} + \| \boldsymbol{f} \|_{L^2(\Omega)} + \| \boldsymbol{f}_1\|_{L^2(\Gamma_t)}  \right).
		\end{aligned}
	\end{equation*}
	The constant $\beta_0$ is from the inf-sup condition and $C_1$ is from {Cauchy–Schwarz} inequality and trace inequality. This completes the proof.
\end{proof}

\section{Fully discrete parallel splitting scheme and its optimally convergent analysis}\label{sch}
In this section, we will present the parallel splitting fully discrete scheme for the Stokes-parabolic system. This algorithm differs from existing iterative decoupled and sequential decoupled methods, as the decoupled subproblems can be computed in parallel.

Let \( \mathcal{T}_h \) represent a partition of the domain \( \Omega \) into triangular elements in \( \mathbb{R}^2 \) or tetrahedral elements in \( \mathbb{R}^3 \), where \( h \) denotes the maximum element diameter in the mesh. In this manuscript, Taylor-Hood elements are used for the pair \( (\boldsymbol{u}, \xi) \), and Lagrange finite elements are employed for \( \boldsymbol{p} \). 
The finite element spaces on \( \mathcal{T}_h \) are defined as follows:
$$
\begin{aligned}
	\boldsymbol{V}_h&:=\left\{\boldsymbol{v}_h \in \boldsymbol{V} \cap C^0(\bar{\Omega}) ;\left.\boldsymbol{v}_h\right|_E \in \boldsymbol{P}_k(E), \text{ } \forall E \in \mathcal{T}_h\right\}, \\
	W_h&:=\left\{\phi_h \in W \cap C^0(\bar{\Omega});\left.\phi_h\right|_E \in P_{k-1}(E), \text{ } \forall E \in \mathcal{T}_h\right\}, \\
	M_{j,h}&:=\left\{\psi_{j,h} \in M_j\cap C^0(\bar{\Omega});\left.\psi_{j,h}\right|_E \in P_l(E), \text{ } \forall E \in \mathcal{T}_h\right\}, j=1,\cdots,A,
\end{aligned}
$$
where $k \geq 2$ and $l \geq 1$ are integers. We set $\boldsymbol{M}_h= M_{1,h} \times \cdots \times M_{A,h}$. With this choice of stable Stokes element pair, the corresponding finite element spaces satisfy the following discrete inf-sup condition. Specifically, there exists a positive constant \( \tilde{\beta} \), independent of \( h \), such that
$$
\sup_{\boldsymbol{v}_h \in \boldsymbol{V}_h} \frac{b\left(\boldsymbol{v}_h, \phi_h\right)}{\left\|\boldsymbol{v}_h\right\|_{H^1(\Omega)}} \geq \tilde{\beta}\left\|\phi_h\right\|_{L^2(\Omega)}, \quad \forall \phi_h \in W_h .
$$
We define the discrete formulation for a function $\phi$ at time $t_{n+1}$ as $\phi^{n+1}$, where $0 \leq n \leq N$ and $n$ is integer. The time step size is denoted by $\Delta t = T / N$. Additionally, let \( C \) represent a generic positive constant that remains independent of mesh and time sizes.

With reference to these element spaces, we present the fully discrete parallel splitting scheme for the quasi-static MPET model.

\textbf{Initial step}: find $\boldsymbol{u}^1_h \in \boldsymbol{V}_h$, $\xi^1_h \in W_h$ and $\boldsymbol{p}^1_h\in\boldsymbol{M}_h$ such that
\begin{align}\label{in1}
	a_1(\boldsymbol{u}^1_h, \boldsymbol{v}_h)- b \left(\boldsymbol{v}_h, \xi^1_h \right) &= (\boldsymbol{f}, \boldsymbol{v}_h) +  \langle \boldsymbol{f}_1, \boldsymbol{v}_h\rangle ,& \forall \mathbf{v}_h \in \boldsymbol{V}_h,\\
	b \left(\boldsymbol{u}^1_h, \phi_h \right) +a_2(\xi^1_h, \phi_h) \label{in2} &=\left(\frac{1}{\lambda} \boldsymbol{\alpha} \! \cdot \! \boldsymbol{p}^1_h, \phi_h \right), & \forall \phi_h \in W_h,\\
		 a_3(\frac{p_{j,h}^1 \! -\! p_{j,h}^0}{\Delta t}, \psi_{j,h}) \!+ \!  \left(\frac{\alpha_j}{\lambda} \frac{ \boldsymbol{\alpha} \! \cdot \! (\boldsymbol{p}_h^1-  \boldsymbol{p}_h^0)}{\Delta t},   \psi_{j,h}\right)& \! -\! \left( \frac{\alpha_j}{\lambda}  \frac{\xi_h^1 \! -\! \xi_h^0}{\Delta t} ,\psi_{j,h} \! \right) \notag \\ +  d(p^1_{j,h}, \psi_{j,h}) 
		&= (q_j, \psi_{j,h}) + \langle g_j, \psi_{j,h} \rangle ,& \forall \psi_j \in M_{j,h}. \label{in3}
	\end{align}

\textbf{Subsequent steps}: for $n \geqslant 1$, given $\xi^{n-1}_h , \xi^{n}_h \in W_h$, and $\boldsymbol{p}^{n-1}_h, \boldsymbol{p}^{n}_h  \in \boldsymbol{M}_h$, the following two subsystems are solved in parallel:

\textbf{Subsystem 1:}
find $\boldsymbol{u}^{n+1}_h \in \boldsymbol{V}_h$, $\xi^{n+1}_h \in W_h$ such that
\begin{align}\label{dis1}
	a_1(\boldsymbol{u}^{n+1}_h, \boldsymbol{v}_h)- b \left(\boldsymbol{v}_h, \xi^{n+1}_h \right) &= (\boldsymbol{f}, \boldsymbol{v}_h) +  \langle \boldsymbol{f}_1, \boldsymbol{v}_h\rangle ,& \forall \mathbf{v}_h \in \boldsymbol{V}_h,\\
	b \left(\boldsymbol{u}^{n+1}_h, \phi_h \right) +a_2(\xi^{n+1}_h, \phi_h)  &=\left(\frac{1}{\lambda} \boldsymbol{\alpha} \! \cdot \! \boldsymbol{p}^{n}_h, \phi_h \right), & \forall \phi_h \in W_h.\label{dis2}
\end{align}

\textbf{Subsystem 2:}
find $\boldsymbol{p}^{n+1}_h\in\boldsymbol{M}_h$ such that
	\begin{align}
		&a_3(\frac{p_{j,h}^{n+1}-p_{j,h}^{n}}{\Delta t}, \psi_{j,h}) +   \left(\frac{\alpha_j}{\lambda}  \frac{\boldsymbol{\alpha} \! \cdot \! \left( \boldsymbol{p}_h^{n+1}-\boldsymbol{p}_h^{n}\right)}{\Delta t}, \psi_{j,h}\right)\notag \\
		&+ \left( L  \alpha_j  \frac{ \boldsymbol{\alpha} \! \cdot \! \left( \boldsymbol{p}_{h}^{n+1}-2  \boldsymbol{p}_{h}^{n}+ \boldsymbol{p}_{h}^{n-1} \right) }{\Delta t} ,\psi_{j,h} \right)  
		+  d(p^{n+1}_{j,h}, \psi_{j,h})
		\label{dis3}  \\
		=& \left( \frac{\alpha_j}{\lambda}  \frac{\xi_h^{n}-\xi_h^{n-1}}{\Delta t} ,\psi_{j,h} \right) + (q_j, \psi_{j,h}) + \langle g_j, \psi_{j,h} \rangle , \quad \forall \psi_j \in M_{j,h},\notag
	\end{align}
where $L=\dfrac{\mu}{\lambda^2}$ is the coefficient of the stabilizer, independent of the time step. 

\begin{rmk}
The coefficient of stabilizer is explicitly determined. In addition, the stabilizer is a term of \( \Delta t \), arising from that \( 2\phi^{n} - \phi^{n-1} \) is the second-order time extrapolation of \( \phi^{n+1} \), divided by \( \Delta t \). This is because the proposed stabilizing method is constructed based on the time derivative of the pressure rather than the pressure itself.
%Therefore, we use a second-order format for the numerator here, and a detailed theoretical analysis will be provided later.
\end{rmk}

\begin{rmk}
%The initial time step can also be handled by sequential decoupled methods.
Here, a coupled approach is used at the initial step for the convenience of error analysis. From the second time step onward, the algorithm decouples the problem into two independent subsystems: a Stokes system \eqref{dis1}-\eqref{dis2} and a parabolic system \eqref{dis3}, which are solved independently at each time level, allowing parallel computation. %Further decoupled of the subproblems is beyond the scope of this paper.
\end{rmk}

\subsection{Interpolation operators}  
We construct the following interpolation operators:
\begin{equation*}
    \Pi_h^{{V}}: \boldsymbol{V} \rightarrow \boldsymbol{V}_h, \quad \Pi_h^{W}: {W} \rightarrow {W}_h,\quad \Pi_h^{M_j}: M_j \rightarrow M_{j, h}, \quad j=1, \cdots, A.
\end{equation*}
For any $\left(\boldsymbol{u}, \xi \right) \in \boldsymbol{V} \times W$, we introduce the interpolant $\left(\Pi_h^{{V}} \boldsymbol{u}, \Pi_h^{W} \xi \right) \in \boldsymbol{V}_h \times W_{h}$, which is uniquely determined as the Stokes equations:
\begin{equation}
	\begin{aligned}\label{eqs}
		\left( 2 \mu \varepsilon \! \left(\Pi_h^{{V}} \boldsymbol{u} \right), \varepsilon(\boldsymbol{v}_h)\right)-\left( \Pi_h^{W} \xi, \nabla \cdot \boldsymbol{v}_h \right) & = \left( 2 \mu \varepsilon(\boldsymbol{u}), \varepsilon(\boldsymbol{v}_h) \right) -\left( \xi, \nabla \cdot \boldsymbol{v}_h \right), & & \forall \boldsymbol{v}_h \in \boldsymbol{V}_h, \\
		\left(\nabla \cdot \Pi_h^V \boldsymbol{u}, \phi_h \right) & =\left(\nabla \cdot \boldsymbol{u}, \phi_h \right), & & \forall \phi_h \in W_h.
	\end{aligned}
\end{equation}

Next, we denote \(N\)-tuple of interpolants \cite{lee2019} as \(\Pi_h^{M} \boldsymbol{p} = \left(\Pi_h^{M_1} p_1, \ldots, \Pi_h^{M_A} p_A\right)\). The interpolation \(\Pi_h^M \boldsymbol{p}\) is defined as the solution to the corresponding elliptic system
\begin{equation}
	\begin{aligned}\label{eqfore}
		\sum_{j=1}^A\left[\left( \boldsymbol{K}_j \nabla \Pi_h^{M_j} p_j, \nabla \psi_{j} \right)+\left( S_j\left(\Pi_h^{M} \boldsymbol{p}\right), \psi_{j} \right) \right]=\sum_{j=1}^A\left[\left( \boldsymbol{K}_j \nabla p_j, \nabla \psi_{j} \right)+\left( S_j(\boldsymbol{p}), \psi_{j} \right) \right],
	\end{aligned}
\end{equation}
for $\psi_{j} \in M_{j, h}$. The well-posedness of this problem can be established based on the properties of $S_j$.
%\begin{equation*}
%\sum_{j=1}^A\left(S_j(\boldsymbol{p}),p_j\right)=\sum_{j=1}^A \sum_{i=1}^A \left( s_{j\leftarrow i} \left(p_j-p_i\right),p_j  \right)=\frac{1}{2} \sum_{j=1}^A \sum_{i=1}^A \left\|  s_{j\leftarrow i}^{1/2} \left( p_j-p_i \right) \right\|_{L^2(\Omega)}^2  \geq 0. 
%\end{equation*}

%; i.e., for any $p_j \in M_j$, we define its interpolant $\Pi_h^{M_j} p_j \in M_{j, h}$ as the unique solution of
%\begin{equation}\label{eqfore}
%\left( \boldsymbol{K}_j \nabla \Pi_h^{M_j} p_j, \psi_h \right)=\left( \boldsymbol{K}_j \nabla p_j, \nabla \psi_h \right) \quad \forall \psi_h \in M_{j, h}
%\end{equation}

For the selected finite element spaces and projection operators, the following error estimate holds for the Stokes-type interpolant defined in \eqref{eqs}. If $\boldsymbol{u} \in \boldsymbol{H}_{0, \Gamma_{\bm{u}}}^{k+1}(\Omega)$ and $\xi \in H^k(\Omega)$, then
\begin{equation*}
\left\|\boldsymbol{u}-\Pi_h^V \boldsymbol{u} \right\|_{H^1(\Omega)}+\left\| \xi - \Pi_h^{W} \xi \right\|_{L^2(\Omega)} \leqslant C h^k\left(\|\boldsymbol{u}\|_{H^{k+1}(\Omega)}+\left\| \xi \right\|_{H^k(\Omega)}\right).
\end{equation*}
Furthermore, the following error estimate holds for the elliptic interpolants defined by \eqref{eqfore}. Assuming that \( p_j \) for \( \boldsymbol{p} = (p_1, p_2, \cdots, p_A) \) are sufficiently regular, then
\begin{equation*}
	\left\|\boldsymbol{p}-\Pi_h^{M} \boldsymbol{p}\right\|_{H^1(\Omega)} \leqslant h^l\left\|\boldsymbol{p}\right\|_{H^{l+1}(\Omega)}, \quad 
	\left\|\boldsymbol{p}-\Pi_h^{M} \boldsymbol{p} \right\|_{L^2(\Omega)} \leqslant Ch^{l+1}\left\|\boldsymbol{p}\right\|_{H^{l+1}(\Omega)}.
\end{equation*}

%We also recall the following estimates, obtained by using Taylor expansions:
%\begin{align*}
%	& \left\|\partial_t \boldsymbol{u}^{n+1}-\frac{\boldsymbol{u}^{n+1}-\boldsymbol{u}^n}{\Delta t}\right\|_{L^2(\Omega)}^2 \leqslant C \Delta t \int_{t_n}^{t_{n+1}}\left\|\partial_{t t} \boldsymbol{u} \right\|_{L^2(\Omega)}^2, \\
%	& \left\| \boldsymbol{u}^{n+1}- \boldsymbol{u}^n\right\|_{L^2(\Omega)}^2 \leqslant C \Delta t \int_{t_n}^{t_{n+1}}\left\|\partial_t 
%	\boldsymbol{u} \right\|_{L^2(\Omega)}^2, \\
%	& \left\|\partial_t \boldsymbol{p}^{n+1}-\frac{\boldsymbol{p}^{n+1}-\boldsymbol{p}^n}{\Delta t}\right\|_{L^2(\Omega)}^2 \leqslant C \Delta t \int_{t_n}^{t_{n+1}}\left\|\partial_{t t} \boldsymbol{p}\right\|_{L^2(\Omega)}^2, \\
%	& \left\|\boldsymbol{p}^{n+1}-\boldsymbol{p}^n\right\|_{L^2(\Omega)}^2 \leqslant C \Delta t \int_{t_n}^{t_{n+1}}\left\|\partial_t \boldsymbol{p} \right\|_{L^2(\Omega)}^2.
%\end{align*}

In the next subsection, we show optimal error estimates of the fully discrete parallel scheme.
\subsection{Error estimate}
\textbf{\textbf{Assumption 1}.}\label{assump1} Assume that $\boldsymbol{u} \in L^{\infty}\left(0, T ; \boldsymbol{H}_{0, \Gamma_{\bm{u}}}^{k+1}(\Omega)\right),$\\  $\partial_t \boldsymbol{u} \in 
 L^2\left(0, T ; \boldsymbol{H}_{0, \Gamma_{\bm{u}}}^{k+1}(\Omega)\right),$ 
 $\partial_{t t} \boldsymbol{u} \in L^2\left(0, T ; \boldsymbol{H}_{0, \Gamma_{\bm{u}}}^1(\Omega)\right), \xi \in L^{\infty}\left(0, T ; H^k(\Omega)\right),$ $\partial_t \xi \in L^2\left(0, T ; H^k(\Omega)\right),$ $\partial_{t t} \xi \in L^2  \left(0, T ;  L^2(\Omega)\right), \boldsymbol{p} \in L^{\infty}\left(0, T ; H_{0, \Gamma_D}^{l+1}(\Omega)\right),$ $\partial_t \boldsymbol{p} \in L^2\left(0, T ; H_{0, \Gamma_D}^{l+1}(\Omega)\right), \partial_{t t} \boldsymbol{p} \in L^2\left(0, T ; L^2(\Omega)\right)$.

Setting 
\begin{equation*}
		\begin{aligned}
		&e_{\boldsymbol{u}}^{n+1}=\boldsymbol{u}^{n+1}-\boldsymbol{u}_h^{n+1}, \quad e^{h,n+1}_{\boldsymbol{u}}=\Pi_h^V \boldsymbol{u}  -\boldsymbol{u}_h^{n+1},\\
		&e_\xi^{n+1}=\xi^{n+1}-\xi_h^{n+1},\quad e^{h,n+1}_\xi=\Pi_h^W \xi - \xi_h^{n+1},\\	&e_{p_j}^{n+1}=p_j^{n+1}-p_{j,h}^{n+1},\quad e^{h,n+1}_{p_j}= \Pi_h^{M_j} p_j-p_{j,h}^{n+1},
	\end{aligned}
\end{equation*}
where $j=1,\cdots,A$ with vector $\boldsymbol{e}^{n+1}_{p}=(e^{n+1}_{p_1},\cdots,e^{n+1}_{p_A})  $. We also need to set 
\begin{equation*}
D_{\boldsymbol{u}}^{n+1} =  e^{h,n+1}_{\boldsymbol{u}}-e^{h,n}_{\boldsymbol{u}}, \quad  D_{{\xi}}^{n+1} =  e^{h,n+1}_\xi-e^{h,n}_\xi,\quad  D_{p_j}^{n+1} =  e^{h,n+1}_{p_j}-e^{h,n}_{p_j},
\end{equation*}  
and set  $\boldsymbol{e}^{h,n+1}_{p}=(e^{h,n+1}_{p_1},\cdots,e^{h,n+1}_{p_A})  $, $  \boldsymbol{D}^{n+1}_{p}=(D^{n+1}_{p_1},\cdots,D^{n+1}_{p_A})  $ with
\begin{align*}
	\boldsymbol{\alpha}\! \cdot \! \boldsymbol{e}^{h,n+1}_{p} =\sum_{j=1}^A \alpha_j e^{h,n+1}_{p_j},\quad \boldsymbol{\alpha} \! \cdot \! \boldsymbol{D}^{n+1}_{p} =\sum_{j=1}^A \alpha_j D^{n+1}_{p_j}.
\end{align*}   

\begin{theorem}\label{thm1}
	Under \textbf{Assumption 1}, let \((\boldsymbol{u}, \xi, \boldsymbol{p})\) denote the solutions of equations \eqref{weakforu1}-\eqref{weakforu3}. For \( 1 \leqslant n \leqslant N \), let \((\boldsymbol{u}_h^{n+1}, \xi_h^{n+1}, \boldsymbol{p}_h^{n+1})\) represent the solutions of the discrete schemes \eqref{dis1}-\eqref{dis3}. If \( L \geqslant \frac{16 \mu  }{\lambda^2 \tilde{\beta}^2 } \), where $ \tilde{\beta}$ from the inf-sup condition, we obtain the following estimate:
	\begin{equation}\label{th1eq}
		\begin{aligned}
			& \sum_{n=1}^N\left[ \mu\left\|\varepsilon \! \left(D_{\boldsymbol{u}}^{n+1}\right)\right\|_{L^2(\Omega)}^2+ \sum_{j=1}^A \frac{c_j}{2}\left\| D_{p_j}^{n+1} \right\|^2_{L^2(\Omega)}+\frac{1}{4 \lambda}\left\| \boldsymbol{\alpha} \! \cdot \! \boldsymbol{D}_p^{n+1}-D_{\xi}^{n+1}\right\|_{L^2(\Omega)}^2\right] \\
			&  +\Delta t \sum_{j=1}^A   \frac{ \kappa_j }{2}  \left\|  \nabla e_{p_j}^{h,N+1} \right\|^2_{L^2(\Omega)} +\Delta t  \sum_{i,j=1}^A  \frac{s_{j\leftarrow i}}{4}    \left\|  e^{h,N+1}_{p_j}-e_{p_i}^{h,N+1}  \right\|_{L^2(\Omega)}^2       \\
			\leqslant & {C}\left[(\Delta t)^3 \int_0^T\left(\left\|\partial_{t t} \boldsymbol{u}\right\|_{H^1(\Omega)}^2+\left\|\partial_{t t} \xi\right\|_{L^2(\Omega)}^2+\left\|\partial_{t t} \boldsymbol{p}\right\|_{L^2(\Omega)}^2\right) d s\right. \\
			& +h^{2 k} \Delta t \int_0^T\left(\left\|\partial_t \boldsymbol{u}\right\|_{H^{k+1}(\Omega)}^2+\left\|\partial_t \xi\right\|_{H^k(\Omega)}^2\right) d s  \left.+h^{2 l+2} \Delta t \int_0^T\left\|\partial_t \boldsymbol{p}\right\|_{H^{l+1}(\Omega)}^2 d s\right].
		\end{aligned}
	\end{equation}
\end{theorem}

\begin{proof}
	We begin by deriving the error equation by subtracting equations \eqref{dis1}-\eqref{dis3} from \eqref{weakforu1}-\eqref{weakforu3}, respectively, then we have
	\begin{equation*}
		\begin{aligned}
			&a_1(e_{\boldsymbol{u}}^{n+1},\boldsymbol{v}_h)-b(\boldsymbol{v}_h,e_\xi^{n+1})=0,\\
			&b(e_{\boldsymbol{u}}^{n+1}, \phi_h )+a_2(e_\xi^{n+1},\phi_h)=\left( \frac{1}{\lambda} \boldsymbol{\alpha} \! \cdot \! \left( \boldsymbol{p}^{n+1} - \boldsymbol{p}_h^n \right) ,\phi_h \right),
		\end{aligned}
	\end{equation*}
	and
	\begin{equation}
		\begin{aligned}\label{errP}
			&a_3 \! \left( \! \partial_t p_j^{n+1} \! -\! \frac{p_{j,h}^{n+1}-p_{j,h}^n}{\Delta t} ,\psi_{j,h} \! \right) \! +\! \left( \! \frac{\alpha_j}{\lambda} \! \left( \!\boldsymbol{\alpha} \! \cdot \! \partial_t \boldsymbol{p}^{n+1} \! -\! \frac{\boldsymbol{\alpha} \! \cdot \! \boldsymbol{p}_h^{n+1} \! - \! \boldsymbol{\alpha} \! \cdot \! \boldsymbol{p}_h^{n}}{\Delta t}\right), \psi_{j,h} \! \right)\\
			&+\left(\! L \alpha_j \boldsymbol{\alpha} \! \cdot \! \left( \! \partial_t \boldsymbol{p}^{n+1} \! -\! \frac{ \boldsymbol{p}_{h}^{n+1} \! -\! \boldsymbol{p}_{h}^n}{\Delta t} \! \right) ,\psi_{j,h} \! \right) \!-\! \left( \! L  \alpha_j \boldsymbol{\alpha} \! \cdot \! \left( \! \partial_t \boldsymbol{p}^{n+1} \! -\! \frac{\boldsymbol{p}_{h}^n \! -\! \boldsymbol{p}_{h}^{n-1}}{\Delta t} \right),\psi_{j,h} \!  \right)\\
			 &- \left( \! \frac{\alpha_j}{\lambda} \!
			\left( \! \partial_t \xi^{n+1} \! -\! \frac{\xi_h^{n}\! -\! \xi_h^{n-1}}{\Delta t}\!  \right),\psi_{j,h} \! \right)+d(e_{p_j}^{n+1},\psi_{j,h}) =0 , \quad j=1,2,\cdots, A.
		\end{aligned}
	\end{equation}
	Using the Stokes projection operator \eqref{eqs}, we have
	\begin{align}
		&a_1(e_{\boldsymbol{u}}^{h,n+1},  {\boldsymbol{v}_h} )-b({\boldsymbol{v}_h},e_\xi^{h,n+1})=0,\label{errorforu1}\\
		&b(e_{\boldsymbol{u}}^{h,n+1}, \phi_h )+a_2(e_\xi^{n+1},\phi_h)=\left( \frac{1}{\lambda} \boldsymbol{\alpha} \! \cdot \! \left( \boldsymbol{p}^{n+1} - \boldsymbol{p}^n+ \boldsymbol{e}_p^n \right) ,\phi_h \right).\label{twoofe}
	\end{align}
	{Making} the difference between the $(n+1)$-th case and the $n$-th case of \eqref{twoofe}, we have
	\begin{equation}
		\begin{aligned}\label{errorforb}
&b(D_{\boldsymbol{u}}^{n+1},\phi_h)+a_2(D_\xi^{n+1},\phi_h)\!+\!a_2(\xi^{n+1} \! -\! \xi^n,\phi_h)-a_2(\Pi_h^W\xi^{n+1}-\Pi_h^W \xi^{n},\phi_h)\\
			=&\left( \frac{1}{\lambda} \boldsymbol{\alpha} \! \cdot \! \left(  \boldsymbol{p}^{n+1}-\boldsymbol{p}^n \! \right) ,\phi_h \! \right)+ \left( \! \frac{1}{\lambda} \boldsymbol{\alpha} \! \cdot \! \boldsymbol{D}_p^n,\phi_h \! \right) \!- \! \left( \! \frac{1}{\lambda} \boldsymbol{\alpha} \! \cdot \! \left( \Pi_h^M \boldsymbol{p}^n \! -\! \Pi_h^M \boldsymbol{p}^{n-1} \right),\phi_h \! \right).
		\end{aligned}
	\end{equation}
	Noting the following equations hold for the continuous problem in \eqref{weakforu2},
	\begin{equation}
		\begin{aligned}\label{exactforb}
			&b(\boldsymbol{u}^{n+1}-\boldsymbol{u}^n,\phi_h)+a_2(\xi^{n+1}-\xi^n,\phi_h)=\left( \frac{1}{\lambda} \boldsymbol{\alpha} \! \cdot \! \left( \boldsymbol{p}^{n+1} - \boldsymbol{p}^n \right) ,\phi_h \right),\\
			&b(\Delta t \partial_t \boldsymbol{u}^{n+1},\phi_h)+a_2(\Delta t \partial_t \xi^{n+1},\phi_h)=\left( \frac{\Delta t}{\lambda} \boldsymbol{\alpha} \! \cdot \! \partial_t  \boldsymbol{p}^{n+1}  ,\phi_h \right).
		\end{aligned}
	\end{equation}
	By substituting the two equations from \eqref{exactforb} into \eqref{errorforb}, we obtain
		\begin{align}
			\label{errorforb2} &b(D_{\boldsymbol{u}}^{n+1},\phi_h)\!+\!a_2(D_\xi^{n+1},\phi_h)\!-\! \left( \frac{1}{\lambda} \boldsymbol{\alpha} \! \cdot \! \boldsymbol{D}_p^n,\phi_h \right)\\
			=&	\left( \frac{1}{\lambda} \boldsymbol{\alpha} \! \cdot \! \left(  \boldsymbol{p}^{n+1} \! -\! \boldsymbol{p}^n \right) ,\phi_h \! \right)\! - \! \left( \frac{1}{\lambda} \boldsymbol{\alpha} \! \cdot \! \left( \! \Pi_h^M \boldsymbol{p}^n \! -\!  \Pi_h^M \boldsymbol{p}^{n-1} \! \right),\phi_h \! \right) \notag \\
			&-a_2(\xi^{n+1}\! -\! \xi^n,\phi_h)+a_2\! \left(\Pi_h^W\xi^{n+1} \! -\! \Pi_h^W \xi^{n},\phi_h \! \right) \notag \\
			=&b\left(\boldsymbol{u}^{n+1} \! -\! \boldsymbol{u}^n,\phi_h\right)- \left( \frac{1}{\lambda} \boldsymbol{\alpha} \! \cdot \! \left( \!  \Pi_h^M \boldsymbol{p}^n \! -\! \Pi_h^M \boldsymbol{p}^{n-1} \! \right),\phi_h \right)+a_2( \Pi_h^W\xi^{n+1} \! -\! \Pi_h^W \xi^{n},\phi_h) \notag \\
			=&b\left(\boldsymbol{u}^{n+1} \! -\! \boldsymbol{u}^n \! -\! \Delta t \partial_t \boldsymbol{u}^{n+1},\phi_h\right)- \left( \frac{1}{\lambda} \boldsymbol{\alpha} \! \cdot \! \left( \Pi_h^M \boldsymbol{p}^n \! -\! \Pi_h^M \boldsymbol{p}^{n-1}  \! -\! \Delta t \partial_t  \boldsymbol{p}^{n+1} \! \right) ,\phi_h \! \right) \notag \\
			&+a_2(\Pi_h^W\xi^{n+1} \! -\! \Pi_h^W \xi^{n} \!-\!\Delta t \partial_t \xi^{n+1},\phi_h). \notag
		\end{align}
	For the error equations for the parabolic system, from \eqref{errP} and elliptic projection operator \eqref{eqfore}, we have
	\begin{align}
		& \label{errorforp} a_3\left(  \frac{e_{p_j}^{h,n+1}-e_{p_j}^{h,n}}{\Delta t} ,\psi_{j,h} \right)+ \left(   \frac{\alpha_j}{\lambda}  \frac{\boldsymbol{\alpha} \! \cdot \! \left( \boldsymbol{e}_p^{h,n+1}-\boldsymbol{e}_p^{h,n} \right)}{\Delta t}  ,\psi_{j,h} \right)  \\
		& +\left( L  \alpha_j   \frac{  \boldsymbol{\alpha} \! \cdot \! \left( \boldsymbol{e}_{p}^{h,n+1}-\boldsymbol{e}_{p}^{h,n}\right) }{\Delta t} ,\psi_{j,h} \right) - \left( L \alpha_j   \frac{ \boldsymbol{\alpha} \! \cdot \! \left( \boldsymbol{e}_{p}^{h,n}-\boldsymbol{e}_{p}^{h,n-1}\right)}{\Delta t} ,\psi_{j,h}  \right)  \notag  \\
		&- \left(  \frac{\alpha_j}{\lambda} 
		\frac{e_\xi^{h,n}-e_\xi^{h,n-1}}{\Delta t} ,\psi_{j,h} \right)+d(e_{p_j}^{h,n+1},\psi_{j,h}) \notag \\
		=& -a_3\left( \partial_t p_j^{n+1}- \frac{\Pi_h^{M_j}p_j^{n+1}-\Pi_h^{M_j}p_j^n}{\Delta t} ,\psi_{j,h} \right) \notag  \\
		& -\left(\frac{\alpha_j}{\lambda} \left( \boldsymbol{\alpha} \! \cdot \! \partial_t \boldsymbol{p}^{n+1}- \frac{\boldsymbol{\alpha} \! \cdot \! \left( \Pi_h^M \boldsymbol{p}^{n+1}- \Pi_h^M \boldsymbol{p}^{n}\right) }{\Delta t}\right), \psi_{j,h}\right) \notag \\
		& -\left(L \alpha_j \boldsymbol{\alpha} \! \cdot \!  \left( \partial_t \boldsymbol{p}^{n+1}-\frac{    \Pi_h^{M}  \boldsymbol{p}^{n+1}-\Pi_h^M \boldsymbol{p}^n}{\Delta t} \right),\psi_{j,h} \right)  \notag \\
		&+ \left( L \alpha_j \boldsymbol{\alpha} \! \cdot \! \left(  \partial_t \boldsymbol{p}^{n+1}-\frac{   \Pi_h^{M}  \boldsymbol{p}^{n}-\Pi_h^M \boldsymbol{p}^{n-1}}{\Delta t} \right),\psi_{j,h}  \right)  \notag \\
		&+ \left(  \frac{\alpha_j}{\lambda} 
		\left( \partial_t \xi^{n+1}-\frac{\Pi_h^W \xi^{n}-\Pi_h^W \xi^{n-1}}{\Delta t} \right),\psi_{j,h} \right). \notag
		\end{align}
	Making the difference between the $(n+1)$-th case and the $n$-th case of \eqref{errorforu1} and setting $\boldsymbol{v}_h = D_{\boldsymbol{u}}^{n+1}$, then, taking $\phi_h = D_\xi^{n+1}$ in \eqref{errorforb2}, we obtain the following two equations, respectively:
	\begin{equation}\label{errorequationtotal}
		\begin{aligned}
			&a_1(D_{\boldsymbol{u}}^{n+1},D_{\boldsymbol{u}}^{n+1})-b(D_\mathbf{u}^{n+1},D_\xi^{n+1})=0,
		\end{aligned}
	\end{equation}
	and
		\begin{align}		&b(D_{\boldsymbol{u}}^{n+1},D_\xi^{n+1})+a_2(D_\xi^{n+1},D_\xi^{n+1})-\left( \frac{1}{\lambda} \boldsymbol{\alpha} \! \cdot \! \boldsymbol{D}_p^n,D_\xi^{n+1} \right) \label{errorequationtotall} \\
			=&  b\left(\boldsymbol{u}^{n+1}-\boldsymbol{u}^n-\Delta t \partial_t \boldsymbol{u}^{n+1},D_\xi^{n+1} \right) \notag \\
			&- \left( \frac{1}{\lambda} \boldsymbol{\alpha} \! \cdot \! \left( \Pi_h^M \boldsymbol{p}^n- \Pi_h^M \boldsymbol{p}^{n-1} - \Delta t \partial_t  \boldsymbol{p}^{n+1} \right) ,D_\xi^{n+1} \right) \notag \\
			&+a_2\left( \Pi_h^W\xi^{n+1}-\Pi_h^W \xi^{n}-\Delta t \partial_t \xi^{n+1},D_\xi^{n+1} \right). \notag
		\end{align}
	
	%Noting that $$ \sum_{j=1}^A  \left(    L  \alpha_j \left( D_{p_j}^{n+1}  -D_{p_j}^n  \right),D_{p_j}^{n+1} \right)  \geq L \alpha_{min}  \sum_{j=1}^A  \left(   D_{p_j}^{n+1}  -D_{p_j}^n  ,D_{p_j}^{n+1} \right)=L \alpha_{min}    \left(  \boldsymbol{D}_{p}^{n+1}  -\boldsymbol{D}_{p}^n  ,\boldsymbol{D}_{p}^{n+1} \right) $$
	Letting $\psi_{j,h} = \Delta t D_{p_j}^{n+1}$ in \eqref{errorforp} and summing over $j = 1, \cdots, A$, we obtain
		\begin{align} 
			& \sum_{j=1}^A a_3\left(  {D}_{p_j}^{n+1} , {D}_{p_j}^{n+1}  \right)+\left( L \boldsymbol{\alpha} \! \cdot \!  \left(  \boldsymbol{D}_p^{n+1} -  \boldsymbol{D}_p^{n} \right) , \boldsymbol{\alpha} \! \cdot \! \boldsymbol{D}_p^{n+1} \right)  \notag \\
			&+\left( \frac{1}{\lambda} \boldsymbol{\alpha} \! \cdot \! \boldsymbol{D}_p^{n+1},\boldsymbol{\alpha} \! \cdot \! \boldsymbol{D}_p^{n+1}   \right)- \left(  \frac{1}{\lambda} 
			D_\xi^n ,\boldsymbol{\alpha} \! \cdot \! \boldsymbol{D}_p^{n+1} \right)+\Delta t  \sum_{j=1}^A d\left({e}_{p_j}^{h, n+1}, {D}_{p_j}^{n+1}\right)\label{errorfore}  \\
			=& -\sum_{j=1}^Aa_3\left( \Delta t \partial_t p_j^{n+1}- \left(\Pi_h^{M_j}p_j^{n+1}-\Pi_h^{M_j}p_j^n \right),D_{p_j}^{n+1} \right) 	\notag \\
			&-\left(\frac{1}{\lambda} \left( \Delta t \boldsymbol{\alpha} \! \cdot \! \partial_t \boldsymbol{p}^{n+1}- \boldsymbol{\alpha} \! \cdot\! \left( \Pi_h^M \boldsymbol{p}^{n+1}- \Pi_h^M \boldsymbol{p}^{n}\right) \right),  \boldsymbol{\alpha} \! \cdot \! \boldsymbol{D}_p^{n+1} \right) \notag \\
			&- \left(L \boldsymbol{\alpha} \! \cdot \!  \left(\Delta t \partial_t \boldsymbol{p}^{n+1}-   \Pi_h^{M}  \boldsymbol{p}^{n+1}+\Pi_h^M \boldsymbol{p}^n \right), \boldsymbol{\alpha} \! \cdot \! \boldsymbol{D}_{p}^{n+1} \right) \notag \\
			&+ \left( L \boldsymbol{\alpha} \! \cdot \! \left( \Delta t \partial_t \boldsymbol{p}^{n+1}-   \Pi_h^{M}  \boldsymbol{p}^{n}+ \Pi_h^M \boldsymbol{p}^{n-1} \right), \boldsymbol{\alpha} \! \cdot \! \boldsymbol{D}_{p}^{n+1}  \right) \notag \\
			&+ \left(  \frac{1}{\lambda} 
			\left( \Delta t \partial_t \xi^{n+1}-\Pi_h^W \xi^{n}+\Pi_h^W \xi^{n-1}  \right), \boldsymbol{\alpha} \! \cdot \! \boldsymbol{D}_p^{n+1} \right).\notag %+\sum_{i=1}^A \Delta t \left( s_{j\leftarrow i} \left(e_{p_j}^{I,n+1}-e_{p_i}^{I,n+1}  \right), D_{p_j}^{n+1}  \right).
		\end{align}
	Noting that the following identity holds,
	\begin{align*}
	\frac{1}{2 \lambda} \! \left\|D_{\xi}^{n+1}\right\|_{L^2(\Omega)}^2 \!+\! \frac{1}{2 \lambda}\! \left\| \boldsymbol{\alpha} \! \cdot \! \boldsymbol{D}_p^{n+1}\right\|_{L^2(\Omega)}^2 \!-\! \left( \! \frac{1}{\lambda} \boldsymbol{\alpha} \! \cdot \! \boldsymbol{D}_p^{n+1}, D_{\xi}^{n+1} \! \right) \!=\! \frac{1}{2 \lambda} \! \left\| \boldsymbol{\alpha} \! \cdot \! \boldsymbol{D}_p^{n+1} \! -\! D_{\xi}^{n+1}\right\|_{L^2(\Omega)}^2.
	\end{align*}
	By summing the indices \( n \) from \( 1 \) to \( N \) for equations \eqref{errorequationtotal}, \eqref{errorequationtotall}, and \eqref{errorfore}, and applying the above equation to the left-hand side, we obtain the following equation:
		\begin{align}
			& \sum_{n=1}^N\left[2 \mu\left\|\varepsilon \! \left(D_\mathbf{u}^{n+1}\right)\right\|_{L^2(\Omega)}^2+\frac{1}{2 \lambda}\left\|D_{\xi}^{n+1}\right\|_{L^2(\Omega)}^2+ \sum_{j=1}^A {c_j}\left\| D_{p_j}^{n+1} \right\|^2_{L^2(\Omega)} \right.\notag \\
			& \left. + \frac{1}{2 \lambda}\left\| \boldsymbol{\alpha} \! \cdot \! \boldsymbol{D}_p^{n+1}\right\|_{L^2(\Omega)}^2+\frac{1}{2 \lambda}\left\| \boldsymbol{\alpha }  \! \cdot \!  \boldsymbol{D}_p^{n+1}-D_{\xi}^{n+1}\right\|_{L^2(\Omega)}^2\right]\label{errorall}  \\
			&+\frac{L}{2} \left( \left\| \boldsymbol{\alpha} \! \cdot \! \boldsymbol{D}_p^{N+1}  \right\|_{L^2(\Omega)}^2 -\left\| \boldsymbol{\alpha} \! \cdot \! \boldsymbol{D}_p^{1}  \right\|_{L^2(\Omega)}^2  +\sum_{n=1}^N  \left\| \boldsymbol{\alpha} \! \cdot \! \left( \boldsymbol{D}_p^{n+1}   - \boldsymbol{D}_p^{n} \right) \right\|_{L^2(\Omega)}^2      \right) \notag \\
			& +\Delta t \sum_{n=1}^N \sum_{j=1}^A d\left({e}_{p_j}^{h, n+1}, {D}_{p_j}^{n+1}\right)=\sum_{i=1}^{10} T_i, \notag
		\end{align}
	where
	\begin{align*} 
		& T_1=\sum_{n=1}^N b\left(\boldsymbol{u}^{n+1}-\boldsymbol{u}^n-\Delta t \partial_t \boldsymbol{u}^{n+1}, D_{\xi}^{n+1}\right), \\ 
		& T_2=\sum_{n=1}^N a_2\left( \Pi_h^W\xi^{n+1}-\Pi_h^W \xi^{n}-\Delta t \partial_t \xi^{n+1},D_\xi^{n+1} \right), \\ 
		& T_3=\sum_{n=1}^N -\left( \frac{1}{\lambda} \boldsymbol{\alpha} \! \cdot \! \left( \Pi_h^M \boldsymbol{p}^n- \Pi_h^M \boldsymbol{p}^{n-1} - \Delta t \partial_t  \boldsymbol{p}^{n+1} \right) ,D_\xi^{n+1} \right), \\
		& T_4=\sum_{n=1}^N \sum_{j=1}^A -a_3\left( \Delta t \partial_t p_j^{n+1}- \left(\Pi_h^{M_j}p_j^{n+1}-\Pi_h^{M_j}p_j^n \right),D_{p_j}^{n+1} \right), \\ 
		&T_5=\sum_{n=1}^N-\left(\frac{1}{\lambda} \left( \Delta t \boldsymbol{\alpha} \! \cdot \! \partial_t \boldsymbol{p}^{n+1}- \boldsymbol{\alpha} \! \cdot \! \left( \Pi_h^M \boldsymbol{p}^{n+1}- \Pi_h^M \boldsymbol{p}^{n}\right) \right),  \boldsymbol{\alpha} \! \cdot \! \boldsymbol{D}_p^{n+1} \right),\\
		& T_6=\sum_{n=1}^N \left(  \frac{1}{\lambda} 
		\left( \Delta t \partial_t \xi^{n+1}-\Pi_h^W \xi^{n}+\Pi_h^W \xi^{n-1}  \right), \boldsymbol{\alpha} \! \cdot \! \boldsymbol{D}_p^{n+1} \right), \\ 
		& T_7=\sum_{n=1}^N  - \left(L \boldsymbol{\alpha} \!\cdot \!  \left(\Delta t \partial_t \boldsymbol{p}^{n+1}-   \Pi_h^{M}  \boldsymbol{p}^{n+1}+\Pi_h^M \boldsymbol{p}^n \right), \boldsymbol{\alpha} \! \cdot \! \boldsymbol{D}_{p}^{n+1} \right),\\
		& T_8=\sum_{n=1}^N \left( L \boldsymbol{\alpha} \! \cdot \! \left( \Delta t \partial_t \boldsymbol{p}^{n+1}-   \Pi_h^{M}  \boldsymbol{p}^{n}+ \Pi_h^M \boldsymbol{p}^{n-1} \right), \boldsymbol{\alpha} \! \cdot \! \boldsymbol{D}_{p}^{n+1}  \right),\\
		& T_9=\sum_{n=1}^N -\left( \frac{1}{\lambda} \boldsymbol{\alpha} \! \cdot \! \left( \boldsymbol{D}_p^{n+1} -\boldsymbol{D}_p^n \right), D_{\xi}^{n+1}\right),\\
		&T_{10}=\sum_{n=1}^N\left(  \frac{1}{\lambda} 
		D_\xi^n ,\boldsymbol{\alpha} \! \cdot \! \boldsymbol{D}_p^{n+1} \right).
	\end{align*}
	
	For the final term on the left-hand side of equation \eqref{errorall}, we can represent it using the definitions of \( d(\cdot,\cdot) \) and \( S_j(\cdot) \) in \eqref{def_S}, yielding
	\begin{align*}
		&	\Delta t \sum_{n=1}^N \sum_{j=1}^A d\left({e}_{p_j}^{h, n+1}, {D}_{p_j}^{n+1}\right)\\
		=& \Delta t \sum_{n=1}^N \sum_{j=1}^A \! \left[ \! \int_{\Omega}  \kappa_j \nabla e_{p_j}^{h,n+1}  \nabla D_{p_j}^{n+1}  \! + \! \sum_{i=1}^A \! \int_{\Omega} \!  s_{j\leftarrow i} (e^{h,n+1}_{p_j}-e_{p_i}^{h,n+1}) (e^{h,n+1}_{p_j}-e_{p_j}^{h,n}) \! \right]\\
		=&\Delta t \sum_{n=1}^N \sum_{j=1}^A \! \left[\sum_{i=1}^A \int_{\Omega}  s_{j\leftarrow i} (e^{h,n+1}_{p_j}-e_{p_i}^{h,n+1}) \left(   \frac{1}{2} \left( e^{h,n+1}_{p_j} -e^{h,n+1}_{p_i}  \right)+\frac{1}{2} \! \left( \! e^{h,n+1}_{p_j} +e^{h,n+1}_{p_i}  \! \right) \right. \right. \\
		&\left. \left. -  \frac{1}{2} \left( e_{p_j}^{h,n}-e_{p_i}^{h,n}     \right)   -\frac{1}{2} \left( e_{p_j}^{h,n}+e_{p_i}^{h,n}     \right)   \right) \right. + \left.   \int_{\Omega}  \kappa_j \nabla e_{p_j}^{h,n+1}  \nabla D_{p_j}^{n+1}     \right]\\
		=&\Delta t \sum_{n=1}^N \sum_{i,j=1}^A  \int_{\Omega}  \frac{1}{2} s_{j\leftarrow i} (e^{h,n+1}_{p_j}-e_{p_i}^{h,n+1}) \left(   \left( e^{h,n+1}_{p_j} -e^{h,n+1}_{p_i}  \right) -  \left( e_{p_j}^{h,n}-e_{p_i}^{h,n}     \right)     \right) \\
		&+ \Delta t \sum_{n=1}^N \sum_{i,j=1}^A  \int_{\Omega} \frac{1}{2} s_{j\leftarrow i} (e^{h,n+1}_{p_j}-e_{p_i}^{h,n+1}) \left(  \left( e^{h,n+1}_{p_j} +e^{h,n+1}_{p_i}  \right)    - \left( e_{p_j}^{h,n}+e_{p_i}^{h,n}     \right)   \right) \\
		&+  \Delta t \sum_{n=1}^N \sum_{j=1}^A  \! \frac{ \kappa_j }{2} \! \left( \! \left\|  \nabla e_{p_j}^{h,n+1} \right\|^2_{L^2(\Omega)} -\left\|  \nabla e_{p_j}^{h,n} \right\|^2_{L^2(\Omega)} + \left\|  \nabla D_{p_j}^{n+1} \right\|^2_{L^2(\Omega)}    		\right)\\
		=&\Delta t \! \sum_{n=1}^N \sum_{i,j=1}^A  \! \frac{s_{j\leftarrow i}}{4} \! \left( \!  \left\|  e^{h,n+1}_{p_j}-e_{p_i}^{h,n+1}  \right\|_{L^2(\Omega)}^2  -\left\|  e^{h,n}_{p_j}-e_{p_i}^{h,n}  \right\|_{L^2(\Omega)}^2    +\left\|  D^{n+1}_{p_j}-D_{p_i}^{n+1}  \right\|_{L^2(\Omega)}^2         \! \right)  \\
		&+  \Delta t \sum_{n=1}^N \sum_{j=1}^A   \frac{ \kappa_j }{2} \left( \left\|  \nabla e_{p_j}^{h,n+1} \right\|^2_{L^2(\Omega)} -\left\|  \nabla e_{p_j}^{h,n} \right\|^2_{L^2(\Omega)} + \left\|  \nabla D_{p_j}^{n+1} \right\|^2_{L^2(\Omega)}    		\right)      \\
		=&\Delta t  \sum_{i,j=1}^A \!  \frac{s_{j\leftarrow i}}{4}  \left( \!  \left\|  e^{h,N+1}_{p_j}-e_{p_i}^{h,N+1}  \right\|_{L^2(\Omega)}^2  -\left\|  e^{h,1}_{p_j}-e_{p_i}^{h,1}  \right\|_{L^2(\Omega)}^2    +\sum_{n=1}^N \left\|  D^{n+1}_{p_j}-D_{p_i}^{n+1}  \right\|_{L^2(\Omega)}^2        \!  \right)  \\
		&+  \Delta t \sum_{j=1}^A   \frac{ \kappa_j }{2} \left( \left\|  \nabla e_{p_j}^{h,N+1} \right\|^2_{L^2(\Omega)} -\left\|  \nabla e_{p_j}^{h,1} \right\|^2_{L^2(\Omega)} + \sum_{n=1}^N \left\|  \nabla D_{p_j}^{n+1} \right\|^2_{L^2(\Omega)}    		\right)   .
	\end{align*}
	
	Thus, for the left-hand side of \eqref{errorall}, we have
        \begin{align}
			& \sum_{n=1}^N\left[2 \mu\left\|\varepsilon \! \left(D_{\boldsymbol{u}}^{n+1}\right)\right\|_{L^2(\Omega)}^2+\frac{1}{2 \lambda}\left\|D_{\xi}^{n+1}\right\|_{L^2(\Omega)}^2+ \sum_{j=1}^A {c_j}\left\| D_{p_j}^{n+1} \right\|^2_{L^2(\Omega)}\right.\notag \\
			& \left. + \frac{1}{2 \lambda}\left\| \boldsymbol{\alpha}\! \cdot \! \boldsymbol{D}_p^{n+1}\right\|_{L^2(\Omega)}^2  
		    +\frac{1}{2 \lambda}\left\| \boldsymbol{\alpha }  \! \cdot \! \boldsymbol{D}_p^{n+1}-D_{\xi}^{n+1}\right\|_{L^2(\Omega)}^2\right] \label{errforalllast}  \\
			&+\frac{L}{2} \left( \left\| \boldsymbol{\alpha} \! \cdot \! \boldsymbol{D}_p^{N+1}  \right\|_{L^2(\Omega)}^2 \! -\! \left\| \boldsymbol{\alpha} \! \cdot \! \boldsymbol{D}_p^{1}  \right\|_{L^2(\Omega)}^2  \! +\! \sum_{n=1}^N  \left\| \boldsymbol{\alpha} \! \cdot \! \left( \boldsymbol{D}_p^{n+1} \!  -\! \boldsymbol{D}_p^{n} \right) \right\|_{L^2(\Omega)}^2      \right) \notag \\
			&+\Delta t \sum_{j=1}^A   \frac{ \kappa_j }{2} \left( \left\|  \nabla e_{p_j}^{h,N+1} \right\|^2_{L^2(\Omega)} -\left\|  \nabla e_{p_j}^{h,1} \right\|^2_{L^2(\Omega)}    		\right) \notag \\
			& +\Delta t  \sum_{i,j=1}^A  \frac{s_{j\leftarrow i}}{4}  \left(   \left\|  e^{h,N+1}_{p_j}-e_{p_i}^{h,N+1}  \right\|_{L^2(\Omega)}^2  -\left\|  e^{h,1}_{p_j}-e_{p_i}^{h,1}  \right\|_{L^2(\Omega)}^2         \right)  \leqslant \notag \text{L.H.S.}  
		\end{align}
	
	Next, we bound the terms $T_i$ for $i=1,2, \ldots, 10$. Recalling the definition of \( b(\cdot, \cdot) \) and $a_2(\cdot,\cdot)$, we can employ the Cauchy-Schwarz inequality, Young's inequality, Stokes projection operator and Taylor expansion to derive the following estimate for \( T_1 \) and $T_2$, for any \( \epsilon_1 > 0 \).
	\begin{align*}
		T_1 +T_2
        %& =\sum_{n=1}^N \int_{\Omega} D_{\xi}^{n+1} \nabla \cdot \left(\boldsymbol{u}^{n+1}-\boldsymbol{u}^n-\Delta t \partial_t \boldsymbol{u}^{n+1}\right) \\
		 \leq& \frac{\epsilon_1}{2} \sum_{n=1}^N\left\|D_{\xi}^{n+1}\right\|_{L^2(\Omega)}^2 \! +\! \frac{C}{\epsilon_1} \! \sum_{n=1}^N\left\| \nabla  \! \cdot \! \left(\boldsymbol{u}^{n+1} \! -\! \boldsymbol{u}^n \! -\! \Delta t \partial_t \boldsymbol{u}^{n+1}\right)\right\|_{L^2(\Omega)}^2\\
		&+\! \frac{C}{\epsilon_1 \lambda^2} \sum_{n=1}^N\left( \! \left\|\xi^{n+1}\!-\! \xi^n \!-\! \Delta t \partial_t \xi^{n+1}\right\|_{L^2(\Omega)}^2 \!+ \! \left\|\Pi_h^W \!\left(\xi^{n+1}\!- \! \xi^n\right) \! -\! \left(\xi^{n+1} \! -\! \xi^n\right)\right\|_{L^2(\Omega)}^2\!\right)\\
         \leq & \frac{\epsilon_1}{2} \sum_{n=1}^N\left\|D_{\xi}^{n+1}\right\|_{L^2(\Omega)}^2+\frac{C}{\epsilon_1}(\Delta t)^3 \int_0^T\left\|\partial_{t t} \boldsymbol{u}\right\|_{H^1(\Omega)}^2 d s\\
        &+\frac{C}{\epsilon_1 \lambda^2}\left[(\Delta t)^3 \int_0^T \! \!\! \left\|\partial_{t t} \xi\right\|_{L^2(\Omega)}^2 d s \! + \! h^{2 k} \Delta t \int_0^T \! \! \! \left(\left\|\partial_t \boldsymbol{u}\right\|_{H^{k+1}(\Omega)}^2 \!+\! \left\|\partial_t \xi \right\|_{H^k(\Omega)}^2 \! \right) d s\right].
	\end{align*}
	%For the same $\epsilon_1$, we estimate $T_2$ similarly, using the Stokes projection operator and Taylor expansion.
	%\begin{align*}
	%	T_2
        %= & \sum_{n=1}^N \frac{1}{\lambda} \int_{\Omega} D_{\xi}^{n+1}\left(\Pi_h^W \xi^{n+1}-\Pi_h^W \xi^n-\Delta t \partial_t \xi^{n+1}\right) \\
		%\leq & \frac{\epsilon_1}{4} \sum_{n=1}^N\left\|D_{\xi}^{n+1}\right\|_{L^2}^2+\frac{C}{\epsilon_1 \lambda^2} \sum_{n=1}^N\left\|\Pi_h^W \xi^{n+1}-\Pi_h^W \xi^n-\Delta t \partial_t \xi^{n+1}\right\|_{L^2(\Omega)}^2 \\
	%	\leq & \frac{\epsilon_1}{4} \sum_{n=1}^N\left\|D_{\xi}^{n+1}\right\|_{L^2}^2 \! +\! \frac{C}{\epsilon_1 \lambda^2} \sum_{n=1}^N\left(\left\|\xi^{n+1}\!-\! \xi^n \!-\! \Delta t \partial_t \xi^{n+1}\right\|_{L^2(\Omega)}^2 \!+ \! \left\|\Pi_h^W \left(\xi^{n+1}\!- \! \xi^n\right)-\left(\xi^{n+1} \! -\! \xi^n\right)\right\|_{L^2(\Omega)}^2\right).\\
     %   \leq & \frac{\epsilon_1}{4} \sum_{n=1}^N\left\|D_{\xi}^{n+1}\right\|_{L^2(\Omega)}^2+\frac{C}{\epsilon_1 \lambda^2}\left[(\Delta t)^3 \int_0^T\left\|\partial_{t t} \xi\right\|_{L^2(\Omega)}^2 d s +h^{2 k} \Delta t \int_0^T\left(\left\|\partial_t \boldsymbol{u}\right\|_{H^{k+1}(\Omega)}^2+\left\|\partial_t \xi\right\|_{H^k(\Omega)}^2\right) d s\right].
%	\end{align*}
We reformulate \( T_3 \) and utilizing the properties of the elliptic operator, we obtain:
	\begin{align*}
		T_3
		= & \sum_{n=1}^N \frac{1}{\lambda} \int_{\Omega} D_{\xi}^{n+1} \boldsymbol{\alpha}\! \cdot \! \left(\Delta t \partial_t \boldsymbol{p}^{n+1}-\Delta t \partial_t \boldsymbol{p}^n\right) \\
		&+\sum_{n=1}^N \frac{1}{\lambda} \int_{\Omega} D_{\xi}^{n+1}\boldsymbol{\alpha} \! \cdot \! \left(\Delta t \partial_t \boldsymbol{p}^n-\Pi_h^M \boldsymbol{p}^n+\Pi_h^M \boldsymbol{p}^{n-1}\right) \\
		\leq & \frac{\epsilon_1}{4} \sum_{n=1}^N\left\|D_{\xi}^{n+1}\right\|_{L^2(\Omega)}^2+\frac{C }{\epsilon_1 \lambda^2} \sum_{n=1}^N\left(\left\|\Delta t  \boldsymbol{\alpha} \! \cdot \! \left( \partial_t \boldsymbol{p}^{n+1}- \partial_t \boldsymbol{p}^n \right) \right\|_{L^2(\Omega)}^2\right. \\
		& \left.+\left\|  \boldsymbol{\alpha}\! \cdot \! \left( \Delta t \partial_t \boldsymbol{p}^n -\boldsymbol{p}^n\!+\! \boldsymbol{p}^{n-1} \right) \right\|_{L^2(\Omega)}^2 \! +\! \left\|  \boldsymbol{\alpha} \! \cdot \! \left( \boldsymbol{p}^n\! -\! \boldsymbol{p}^{n-1} \! -\! \Pi_h^M \left(\boldsymbol{p}^n-\boldsymbol{p}^{n-1}\right)\right) \right\|_{L^2(\Omega)}^2\right)\\
		\leq& \frac{\epsilon_1}{4} \sum_{n=1}^N\left\|D_{\xi}^{n+1}\right\|_{L^2(\Omega)}^2\!+ \! \frac{C \alpha_m^2}{\epsilon_1 \lambda^2}\left[(\Delta t)^3 \! \int_0^T \! \! \! \left\|\partial_{t t} \boldsymbol{p}\right\|_{L^2(\Omega)}^2 d s \! +\! h^{2 l+2} \Delta t \int_0^T \! \! \! \left\|\partial_t \boldsymbol{p} \right\|_{H^{l\!+\!1}(\Omega)}^2 d s \! \right],
	\end{align*}
	where $\alpha_m= \text{max}(\alpha_1,\cdots,\alpha_A)$.
	Following the definition of \( a_3(\cdot, \cdot) \), we derive the estimate for the term \( T_4 \), valid for any \( \epsilon_2 > 0 \),
		\begin{align*}
			T_4
			%=&\sum_{n=1}^N \sum_{j=1}^A c_j \int_{\Omega} D_{p_j}^{n+1}\left(\Pi_h^M p_j^{n+1}- \Pi_h^M p_j^n-\Delta t \partial_t p_j^{n+1}\right)\\
			\leq & \sum_{n=1}^N \sum_{j=1}^A  \frac{c_j}{2} \left\|D_{p_j}^{n+1}\right\|_{L^2(\Omega)}^2 \!+ \sum_{n=1}^N \frac{c_m}{2}  \left\|  \Pi_h^M \boldsymbol{p}^{n+1}- \Pi_h^M \boldsymbol{p}^n-\Delta t \partial_t \boldsymbol{p}^{n+1} \right\|^2_{L^2(\Omega)} \\
			\leq &\! \sum_{n=1}^N \! \sum_{j=1}^A \!  \frac{c_j}{2} \! \left\|D_{p_j}^{n+1}\right\|_{L^2(\Omega)}^2\! \! +\! \frac{C c_m }{2} \! \left[ \! (\Delta t)^3 \! \int_0^T \! \! \! \left\|\partial_{t \!t} \boldsymbol{p}\right\|_{L^2(\Omega)}^2 d s \! +\! h^{2 l+2} \Delta t \int_0^T \! \! \! \left\|\partial_t \boldsymbol{p}\right\|_{H^{\! l+ \! 1}(\Omega\!)}^2\! d s \! \right],
		\end{align*}
	where $c_m=\text{max}(c_1,\cdots,c_A)$. Similarly, we obtain
		\begin{align*}
			T_5
			%&=\sum_{n=1}^A \frac{1}{\lambda} \int_{\Omega} \left( \boldsymbol{\alpha} \cdot \boldsymbol{D}_p^{n+1} \right) \left( \boldsymbol{\alpha} \cdot \left( \Pi_h^M \boldsymbol{p}^{n+1}- \Pi_h^M \boldsymbol{p}^{n}-   \Delta t  \partial_t \boldsymbol{p}^{n+1} \right) \right)\\
			&\leq \! \frac{\epsilon_2 }{8}\! \sum_{n=1}^N \! \left\| \boldsymbol{\alpha}\!  \cdot \! \boldsymbol{D}_p^{n+1}\right\|_{L^2(\Omega)}^2\!+\! C\!  \frac{\alpha_m^2}{\epsilon_2 \lambda^2}\left[(\Delta t)^3 \! \int_0^T \! \! \! \left\|\partial_{t t} \boldsymbol{p} \right\|_{L^2(\Omega)}^2 d s \!+ \! h^{2  l+2} \! \Delta t \! \int_0^T \! \! \left\|\partial_t \boldsymbol{p} \right\|_{H^{l+1} \! (\Omega \! )}^2 d s \! \right]\!.
		\end{align*}
	Using a similar approach, we can bound \( T_6 \) as follows,
	\begin{align*}
		T_6
		%= & \sum_{n=1}^N \frac{1}{\lambda}  \left( \int_{\Omega} \left(\Delta t \partial_t \xi^{n+1}-\Delta t \partial_t \xi^{n}\right) \boldsymbol{\alpha } \cdot \boldsymbol{D}_p^{n+1} \right. \\ 
		%& \left.+\int_{\Omega} \left(\Delta t \partial_t \xi^{n}- \Pi_h^W \xi^{n}+\Pi_h^W \xi^{n-1}\right) \boldsymbol{\alpha} \cdot \boldsymbol{D}_p^{n+1} \right) \\
		\leq & \frac{\epsilon_2 }{8} \sum_{n=1}^N\left\| \boldsymbol{\alpha} \! \cdot \! \boldsymbol{D}_p^{n+1}\right\|_{L^2(\Omega)}^2+\frac{C}{\epsilon_2 \lambda^2}\left[(\Delta t)^3 \int_0^T\left\|\partial_{t t} \xi\right\|_{L^2(\Omega)}^2 d s \right. \\
		&\left. +h^{2 k} \Delta t \int_0^T\left(\left\|\partial_t \boldsymbol{u}\right\|_{H^{k+1}(\Omega)}^2+\left\|\partial_t \xi\right\|_{H^k(\Omega)}^2\right) d s\right].
	\end{align*}
	Also, we bound $T_7$ and $T_8$ as follows,
	\begin{align*}
		T_7
		\leq & \sum_{n=1}^N \left( \frac{\epsilon_2 }{8} \left\| \boldsymbol{\alpha} \! \cdot \! \boldsymbol{D}_{p}^{n+1}\right\|_{L^2(\Omega)}^2 + \frac{CL^2  }{\epsilon_2 }\left\|  \boldsymbol{\alpha} \! \cdot \! \left(  \Pi_h^{M}  \boldsymbol{p}^{n+1}-\Pi_h^M \boldsymbol{p}^n -\Delta t \partial_t \boldsymbol{p}^{n+1} \right) \right\|_{L^2(\Omega)}^2 \right)\\
		\leq & \sum_{n=1}^N \left( \frac{\epsilon_2 }{8} \left\| \boldsymbol{\alpha} \! \cdot \! \boldsymbol{D}_{p}^{n+1}\right\|_{L^2(\Omega)}^2 + \frac{CL^2 \alpha_m^2 }{\epsilon_2 }\left\|  \Pi_h^{M}  \boldsymbol{p}^{n+1}-\Pi_h^M \boldsymbol{p}^n -\Delta t \partial_t \boldsymbol{p}^{n+1} \right\|_{L^2(\Omega)}^2 \right)\\
		\leq & \frac{\epsilon_2 }{8} \sum_{n=1}^N\left\| \boldsymbol{\alpha} \! \cdot \! \boldsymbol{D}_p^{n+1}\right\|_{L^2(\Omega)}^2 +C\frac{L^2 \alpha_m^2 }{\epsilon_2  }\left[(\Delta t)^3 \int_0^T\left\|\partial_{t t} \boldsymbol{p} \right\|_{L^2(\Omega)}^2 d s \right. \\
		& \left. +h^{2 l+2} \Delta t \int_0^T\left\|\partial_t \boldsymbol{p}\right\|_{H^{l+1}(\Omega)}^2 d s\right].\\
		T_8
		%=&\sum_{n=1}^N \left( L \boldsymbol{\alpha} \cdot \left( \Delta t \partial_t \boldsymbol{p}^{n+1}-   \Pi_h^{M}  \boldsymbol{p}^{n}+ \Pi_h^M \boldsymbol{p}^{n-1} \right), \boldsymbol{\alpha} \cdot \boldsymbol{D}_{p}^{n+1}  \right)\\
		%=&\sum_{n=1}^N  L\left( \boldsymbol{\alpha} \cdot   \left( \Delta t \partial_t \boldsymbol{p}^{n+1}-\Delta t \partial_t \boldsymbol{p}^{n} \right),  \boldsymbol{\alpha} \cdot \boldsymbol{D}_{p}^{n+1} \right) \\
		%&+ L\left( \boldsymbol{\alpha} \cdot \Delta t \partial_t \boldsymbol{p}^{n}- \boldsymbol{\alpha} \cdot \left({\Pi_h^M \boldsymbol{p}^{n}-\Pi_h^M \boldsymbol{p}^{n-1}}\right), \boldsymbol{\alpha} \cdot \boldsymbol{D}_{p}^{n+1} \right)\\
		\leq & \frac{\epsilon_2 }{8} \sum_{n=1}^N\left\|\boldsymbol{\alpha}\! \cdot \! \boldsymbol{D}_p^{n+1}\right\|_{L^2(\Omega)}^2 \! +\! C\frac{L^2 \! \alpha_m^2 }{\epsilon_2}\! \left[ \! (\Delta t)^3\! \int_0^T\! \! \left\|\partial_{t t} \boldsymbol{p}\right\|_{L^2(\Omega)}^2 d s \! +\! h^{2 l+2} \!  \Delta t \! \int_0^T \!  \! \left\| \! \partial_t  \boldsymbol{p}\right\|_{H^{l+1}}^2 d s \! \right].
	\end{align*}
Next, we estimate \( T_9 \) and \( T_{10} \) as follows,	\begin{align*}
		T_9
		%=&\sum_{n=1}^N -\left( \frac{1}{\lambda} \boldsymbol{\alpha} \cdot \left( \boldsymbol{D}_p^{n+1} -\boldsymbol{D}_p^n \right), D_{\xi}^{n+1}\right)\\
		\leq &	\frac{\epsilon_1}{4} \sum_{n=1}^N\left\|D_{\xi}^{n+1}\right\|_{L^2(\Omega)}^2+\frac{1}{\lambda^2 \epsilon_1} \sum_{n=1}^N \| \boldsymbol{\alpha} \! \cdot \! \left( \boldsymbol{D}_p^{n+1}-\boldsymbol{D}_p^n \right) \|_{L^2(\Omega)}^2,\\
		T_{10}
		%&=\sum_{n=1}^N\left(  \frac{1}{\lambda} D_\xi^n ,\boldsymbol{\alpha} \cdot \boldsymbol{D}_p^{n+1} \right)\\
		\leq &	 \sum_{n=1}^N  \frac{1}{2\lambda}\left\|D_{\xi}^{n}\right\|_{L^2(\Omega)}^2+ \sum_{n=1}^N  \frac{1}{2\lambda} \| \boldsymbol{\alpha} \! \cdot \! \boldsymbol{D}_p^{n+1} \|_{L^2(\Omega)}^2.
	\end{align*}
	Using the bounds of $T_1,\cdots, T_{10}$ in \eqref{errforalllast}, we have
	\begin{align*}
		& \sum_{n=1}^N\left[2 \mu\left\|\varepsilon \! \left(D_{\boldsymbol{u}}^{n+1}\right)\right\|_{L^2(\Omega)}^2+\frac{1}{2 \lambda}\left\|D_{\xi}^{n+1}\right\|_{L^2(\Omega)}^2+ \sum_{j=1}^A {c_j}\left\| D_{p_j}^{n+1} \right\|^2_{L^2(\Omega)}\right. \\
		&\left. +   \frac{1}{2 \lambda}\left\| \boldsymbol{\alpha} \! \cdot \! \boldsymbol{D}_p^{n+1}\right\|_{L^2(\Omega)}^2+\frac{1}{2 \lambda}\left\| \boldsymbol{\alpha } \! \cdot \!  \boldsymbol{D}_p^{n+1}-D_{\xi}^{n+1}\right\|_{L^2(\Omega)}^2\right]\\
		&+\frac{L}{2} \left( \! \left\| \boldsymbol{\alpha} \! \cdot \! \boldsymbol{D}_p^{N+1}  \right\|_{L^2(\Omega)}^2 -\left\| \boldsymbol{\alpha} \! \cdot \! \boldsymbol{D}_p^{1}  \right\|_{L^2(\Omega)}^2  +\sum_{n=1}^N  \left\| \boldsymbol{\alpha} \! \cdot \! \left( \boldsymbol{D}_p^{n+1}   - \boldsymbol{D}_p^{n} \right) \right\|_{L^2(\Omega)}^2     \! \right) \\
		& +\Delta t \sum_{j=1}^A   \frac{ \kappa_j }{2} \left( \left\|  \nabla e_{p_j}^{h,N+1} \right\|^2_{L^2(\Omega)} -\left\|  \nabla e_{p_j}^{h,1} \right\|^2_{L^2(\Omega)}    		\right)\\
		& +\Delta t  \sum_{i,j=1}^A  \frac{s_{j\leftarrow i}}{4}  \left(   \left\|  e^{h,N+1}_{p_j}-e_{p_i}^{h,N+1}  \right\|_{L^2(\Omega)}^2  -\left\|  e^{h,1}_{p_j}-e_{p_i}^{h,1}  \right\|_{L^2(\Omega)}^2         \right)  \\
		\leq&{\epsilon_1} \sum_{n=1}^N\left\|D_{\xi}^{n+1}\right\|_{L^2(\Omega)}^2+ \frac{\epsilon_2 }{2} \sum_{n=1}^N\left\|\boldsymbol{\alpha} \! \cdot \! \boldsymbol{D}_p^{n+1}\right\|_{L^2(\Omega)}^2+\sum_{n=1}^N \sum_{j=1}^A  \frac{c_j}{2} \left\|D_{p_j}^{n+1}\right\|_{L^2(\Omega)}^2\\
		&+\frac{1}{\lambda^2 \epsilon_1} \sum_{n=1}^N \|   \boldsymbol{\alpha} \! \cdot \! \left( \boldsymbol{D}_p^{n+1} \! -\!  \boldsymbol{D}_p^n  \right) \|_{L^2(\Omega)}^2 \! +\! \sum_{n=1}^N  \frac{1}{2\lambda}\left\|D_{\xi}^{n}\right\|_{L^2(\Omega)}^2+ \sum_{n=1}^N  \frac{1}{2\lambda} \| \boldsymbol{\alpha} \! \cdot \! \boldsymbol{D}_p^{n+1} \|_{L^2(\Omega)}^2\\
		&	+ {C}\left[(\Delta t)^3 \int_0^T\left(\left\|\partial_{t t} \boldsymbol{u}\right\|_{H^1(\Omega)}^2+\left\|\partial_{t t} \xi\right\|_{L^2(\Omega)}^2+\left\|\partial_{t t} \boldsymbol{p} \right\|_{L^2(\Omega)}^2\right) d s\right. \\
		& +h^{2 k} \Delta t \int_0^T\left(\left\|\partial_t \boldsymbol{u}\right\|_{H^{k+1}(\Omega)}^2+\left\|\partial_t \xi\right\|_{H^k(\Omega)}^2\right) d s  \left.+h^{2 l+2} \Delta t \int_0^T\left\|\partial_t \boldsymbol{p}\right\|_{H^{l+1}(\Omega)}^2 d s\right].
	\end{align*}
	Applying the inf-sup condition yields
	\begin{equation}\label{bound1}
		\begin{aligned}
		\tilde{\beta}\left\|D_{\xi}^{n+1}\right\|_{L^2(\Omega)} & \leq \sup _{\boldsymbol{v} \in \boldsymbol{V}_h} \frac{b\left(\boldsymbol{v}, D_{\xi}^{n+1}\right)}{\left\|\boldsymbol{v}\right\|_{H^1(\Omega)}}=\sup _{\boldsymbol{v} \in \boldsymbol{V}_h} \frac{a_1\left(D_{\boldsymbol{u}}^{n+1}, \boldsymbol{v}\right)}{\left\|\boldsymbol{v}\right\|_{H^1(\Omega)}} \leq 2 \mu \left\|\varepsilon \! \left(D_{\boldsymbol{u}}^{n+1}\right)\right\|_{L^2(\Omega)},
		\end{aligned}
	\end{equation}
	then we have
	\begin{equation}\label{boud2}
		\begin{aligned}
			\frac{1}{2}\left\| \boldsymbol{\alpha} \! \cdot \! \boldsymbol{D}_p^{n+1}\right\|_{L^2(\Omega)}^2 & \leq\left\| \boldsymbol{\alpha} \! \cdot \!  \boldsymbol{D}_p^{n+1} \! -\! D_{\xi}^{n+1}\right\|_{L^2(\Omega)}^2+\left\|D_{\xi}^{n+1}\right\|_{L^2(\Omega)}^2 \\
			& \leq\left\| \boldsymbol{\alpha} \! \cdot \! \boldsymbol{D}_p^{n+1}-D_{\xi}^{n+1}\right\|_{L^2(\Omega)}^2+ \left( \frac{2 \mu }{\tilde{\beta}}\right)^2 \left\|\varepsilon \! \left(D_{\boldsymbol{u}}^{n+1}\right)\right\|_{L^2(\Omega)}^2.
		\end{aligned}
	\end{equation}
	Thus, we set the coefficient $\epsilon_1=\frac{\tilde{\beta}^2}{8\mu }$, such that
	$$ {\epsilon_1} \left\|D_{\xi}^{n+1}\right\|_{L^2(\Omega)}^2 \leq \frac{\mu}{2}\left\|\varepsilon\left(D_{\boldsymbol{u}}^{n+1}\right)\right\|^2_{L^2(\Omega)}.$$
	Determine coefficient $\epsilon_2=\min \left\{\frac{1}{4 \lambda}, \frac{\tilde{\beta}^2}{8 \mu}\right\}$, such that
	$$ \frac{\epsilon_2  }{2} \left\| \boldsymbol{\alpha} \! \cdot \! \boldsymbol{D}_p^{n+1}\right\|_{L^2(\Omega)}^2 \leq \frac{\mu}{2}\left\|\varepsilon \! \left(D_{\boldsymbol{u}}^{n+1}\right)\right\|_{L^2(\Omega)}^2+\frac{1}{4 \lambda}\left\| \boldsymbol{\alpha} \! \cdot \! \boldsymbol{D}_p^{n+1} \! -\! D_{\xi}^{n+1}\right\|_{L^2(\Omega)}^2.$$
	Thus, we derive that
		\begin{align}
			& \sum_{n=1}^N \! \left[ \! \mu \! \left\|\varepsilon \! \left(D_\mathbf{u}^{n+1}\right)\right\|_{L^2(\Omega)}^2 \! +\! \sum_{j=1}^A \frac{c_j}{2}\left\| D_{p_j}^{n+1} \right\|^2_{L^2(\Omega)} \! +\! \frac{1}{4 \lambda}\left\| \boldsymbol{\alpha} \!  \cdot \! \boldsymbol{D}_p^{n+1} \! -\! D_{\xi}^{n+1}\right\|_{L^2(\Omega)}^2\right] \notag \\
			&+\frac{L}{2} \! \left\| \boldsymbol{\alpha}  \! \cdot  \!  \boldsymbol{D}_p^{N+1}   \! \right\|_{L^2(\Omega)}^2 \! + \!
			\left(\frac{L}{2} \! - \! \frac{1}{\lambda^2 \epsilon_1}\right)\sum_{n=1}^N  \left\| \boldsymbol{\alpha} \! \cdot \! \left(  \boldsymbol{D}_p^{n+1} \!  -  \! \boldsymbol{D}_p^{n}  \! \right) \right\|_{L^2(\Omega)}^2 \! +\! \frac{1}{2\lambda}\left\|D_{\xi}^{N+1}\right\|_{L^2(\Omega)}^2 \label{lastinth1}   \\
			&+\Delta t \sum_{j=1}^A   \frac{ \kappa_j }{2}  \left\|  \nabla e_{p_j}^{h,N+1} \right\|^2_{L^2(\Omega)} +\Delta t  \sum_{i,j=1}^A  \frac{s_{j\leftarrow i}}{4}    \left\|  e^{h,N+1}_{p_j}-e_{p_i}^{h,N+1}  \right\|_{L^2(\Omega)}^2  \notag   \\
			\leq & {C}\left[(\Delta t)^3 \int_0^T\left(\left\|\partial_{t t} \boldsymbol{u}\right\|_{H^1(\Omega)}^2+\left\|\partial_{t t} \xi\right\|_{L^2(\Omega)}^2+\left\|\partial_{t t} \boldsymbol{p}\right\|_{L^2(\Omega)}^2\right) d s\right. \notag \\
			& +h^{2 k} \Delta t \int_0^T\left(\left\|\partial_t \boldsymbol{u}\right\|_{H^{k+1}(\Omega)}^2+\left\|\partial_t \xi\right\|_{H^k(\Omega)}^2\right) d s  \left.+h^{2 l+2} \Delta t \int_0^T\left\|\partial_t \boldsymbol{p}\right\|_{H^{l+1}(\Omega)}^2 d s\right]\notag \\
			&  +\frac{L}{2}  \left\| \boldsymbol{\alpha } \! \cdot \! \boldsymbol{D}_p^{1}  \right\|_{L^2(\Omega)}^2+\frac{1}{2\lambda}\left\|D_{\xi}^{1}\right\|_{L^2(\Omega)}^2+ \Delta t \sum_{j=1}^A   \frac{ \kappa_j }{2} \left\|  \nabla e_{p_j}^{h,1} \right\|^2_{L^2(\Omega)}    \notag	\\
			&+\Delta t  \sum_{i,j=1}^A  \frac{s_{j\leftarrow i}}{4} \left\|  e^{h,1}_{p_j}-e_{p_i}^{h,1}  \right\|_{L^2(\Omega)}^2     . \notag
		\end{align}
	
	Next, we estimate the error contributed by the coupled solution at the initial time step for \eqref{in1}-\eqref{in3}. By employing a similar proof technique as before and the fact that $\boldsymbol{u}_h^0=\Pi_h^V \boldsymbol{u}^0, \xi_h^0=\Pi_h^W \xi^0, \boldsymbol{p}_h^0=\Pi_h^M \boldsymbol{p}^0$, it's not hard to obtain
	\begin{equation}\label{last2inth1}
		\begin{aligned}
			&  \mu \left\|\varepsilon \! \left(D_{\boldsymbol{u}}^1\right)\right\|_{L^2(\Omega)}^2+     \frac{1}{2\lambda} \left\|  \boldsymbol{\alpha} \! \cdot \!  \boldsymbol{D}_p^1-D_\xi^1     \right\|^2_{L^2(\Omega)} \\
			& + \Delta t \sum_{j=1}^A   \frac{ \kappa_j }{2} \left\|  \nabla e_{p_j}^{h,1} \right\|^2_{L^2(\Omega)}  +\Delta t  \sum_{i,j=1}^A  \frac{s_{j\leftarrow i}}{4} \left\|  e^{h,1}_{p_j}-e_{p_i}^{h,1}  \right\|_{L^2(\Omega)}^2    \\
			\leq		&  C \left(  (\Delta t)^3 \int_0^{t_1}\left(\left\|\partial_{t t} \boldsymbol{u}\right\|_{H^1(\Omega)}^2+\left\|\partial_{t t} \xi\right\|_{L^2(\Omega)}^2+\left\|\partial_{t t} p\right\|_{L^2(\Omega)}^2\right) d s \right. \\
			& \left. +\Delta t h^{2 k} \int_0^{t_1}\left(\left\|\partial_t \boldsymbol{u}\right\|_{H^{k+1}(\Omega)}^2+\left\|\partial_t \xi\right\|_{H^k(\Omega)}^2\right) d s+\Delta t h^{2 l+2} \int_0^{t_1}\left\|\partial_t p\right\|_{H^{l+1}(\Omega)}^2 d s \right).
		\end{aligned}
	\end{equation}
By combining \eqref{bound1}-\eqref{boud2} and utilizing the estimate in \eqref{last2inth1} to account for the error contribution from the initial step (at \( t = t_1 \)) of the solution in \eqref{lastinth1}, we arrive at \eqref{th1eq}, thereby completing the proof.
\end{proof}

\begin{rmk}
From the proof process, it is evident that in every step where errors are controlled using scaled inequalities, the Lam\'{e} constant consistently appears in the denominator, while \( c_m = \max(c_1, \dots, c_A) \) and \( \alpha_m = \max(\alpha_1, \dots, \alpha_A) \) appear in the numerator. Specifically, we derive 
$$
 C = C_{\lambda, c_m, \alpha_m} \leqslant \tilde{C} \left( 1 + c_m + \frac{1+\alpha_m^2}{\lambda} + \frac{1+\alpha_m^2}{\lambda^2} + \frac{1+\alpha_m^2}{\lambda^3} \right).
$$
Thus, as \( \lambda \to \infty \), \( c_m \to 0 \), and \( \alpha_m \to 0 \), the constant \(C\) in the error estimates does not grow. This demonstrates that the scheme is robust with respect to these parameters. A similar conclusion also holds in Theorem \ref{mainre}.
\end{rmk}

\begin{theorem}\label{mainre}
	Under the same assumptions in \cref{thm1}, let $(\boldsymbol{u},\xi,\boldsymbol{p})$ be solutions of equations \eqref{weakforu1}-\eqref{weakforu3}. For $ 1 \leqslant n \leqslant N $, $(\boldsymbol{u}_h^{n+1},\xi_h^{n+1},\boldsymbol{p}_h^{n+1})$ be solutions of discrete schemes \eqref{dis1}-\eqref{dis3},
	we have 
		\begin{align}
			&{\mu}  \left\| \varepsilon( {{e}_{\boldsymbol{u}}^{h,N+1} }) \right\|_{L^2(\Omega)}^2   + \sum_{j=1}^A {c_j}  \left\| e_{p_j}^{h,N+1} \right\|_{L^2(\Omega)}^2 +\frac{1}{\lambda}     \left\| \boldsymbol{\alpha} \! \cdot  \! \boldsymbol{e}_p^{h,N+1} \right\|_{L^2(\Omega)}^2 + \left\| e_{\xi}^{h,N+1}\right\|_{L^2(\Omega)}^2 \notag  \\
            &+  \Delta t  \sum_{n=1}^N \sum_{j=1}^A \kappa_j \left\|  \nabla e_{p_j}^{h,n+1} \right\|_{L^2(\Omega)}^2 + \frac{\Delta t}{2} \sum_{n=1}^N \sum_{i,j=1}^A \|  s_{j\leftarrow i}^{1/2} \left( e^{h,n+1}_{p_j}-e^{h,n+1}_{p_i} \right) \|_{L^2(\Omega)}^2 \label{th3eq}  \\ 
			\leqslant 
			& {C}\left[(\Delta t)^2 \int_0^T\left(\left\|\partial_{t t} \boldsymbol{u}\right\|_{H^1(\Omega)}^2+\left\|\partial_{t t} \xi\right\|_{L^2(\Omega)}^2+\left\|\partial_{t t} \boldsymbol{p} \right\|_{L^2(\Omega)}^2\right) d s\right. \notag \\
			& +h^{2 k} \int_0^T\left(\left\|\partial_t \boldsymbol{u}\right\|_{H^{k+1}(\Omega)}^2+\left\|\partial_t \xi\right\|_{H^k(\Omega)}^2\right) d s  \left.+h^{2 l+2}  \int_0^T\left\|\partial_t \boldsymbol{p} \right\|_{H^{l+1}(\Omega)}^2 d s\right]. \notag
		\end{align}
\end{theorem}
\begin{proof}
	Selecting $\boldsymbol{v}=D_{\boldsymbol{u}}^{n+1}$ in \eqref{errorforu1}, $\phi=e_\xi^{h,n+1}$ in \eqref{errorforb2}, we obtain
		\begin{align}
&a_1(e_{\boldsymbol{u}}^{h,n+1},D_{\boldsymbol{u}}^{n+1})-b(D_{\boldsymbol{u}}^{n+1},e_\xi^{h,n+1})=0,\notag \\
			& b(D_{\boldsymbol{u}}^{n+1},e_\xi^{h,n+1})+a_2(D_\xi^{n+1},e_\xi^{h,n+1})-\left( \frac{1}{\lambda} \boldsymbol{\alpha} \! \cdot \! \boldsymbol{D}_p^n,e_\xi^{h,n+1} \right)\label{errorequationtotal2} \\
			& = b\left(\boldsymbol{u}^{n+1}  -  \boldsymbol{u}^n  -\Delta t \partial_t \boldsymbol{u}^{n+1},e_\xi^{h,n+1} \right) \notag \\
			& \quad - \left( \frac{1}{\lambda} \boldsymbol{\alpha} \! \cdot \! \left( \Pi_h^M \boldsymbol{p}^n- \Pi_h^M \boldsymbol{p}^{n-1} - \Delta t \partial_t  \boldsymbol{p}^{n+1} \right) ,e_\xi^{h,n+1} \right) \notag \\
			& \quad +a_2\left(\Pi_h^W\xi^{n+1}-\Pi_h^W \xi^{n}-\Delta t \partial_t \xi^{n+1},e_\xi^{h,n+1}\right). \notag
		\end{align}
	Setting $\psi=e_{p_j}^{h,n+1}\Delta t$ in \eqref{errorforp}, for $j=1,\cdots, A$, we have
		\begin{align*}
			&a_3\left(  {D_{p_j}^{h,n+1}} ,e_{p_j}^{h,n+1} \right)+ \left(   \frac{\alpha_j}{\lambda}  {\boldsymbol{\alpha} \! \cdot \!  \boldsymbol{D}_p^{n+1}}  ,e_{p_j}^{h,n+1} \right)  - \left(  \frac{\alpha_j}{\lambda} 
			D_\xi^{n} ,e_{p_j}^{h,n+1} \right)\\
			&+\left( L  \alpha_j     \boldsymbol{\alpha} \! \cdot \! \left( \boldsymbol{D}_p^{n+1} -\boldsymbol{D}_p^n \right) ,e_{p_j}^{h,n+1}  \right)+\Delta t d(e_{p_j}^{h,n+1},e_{p_j}^{h,n+1}  ) \\
			=& -a_3\left( \Delta t \partial_t p_j^{n+1}- \left( \Pi_h^{M_j}p_j^{n+1}-\Pi_h^{M_j}p_j^n\right) ,e_{p_j}^{h,n+1}   \right)\\
			&-\left(\frac{\alpha_j}{\lambda} \left( \Delta t \boldsymbol{\alpha} \! \cdot \! \partial_t \boldsymbol{p}^{n+1}-  \boldsymbol{\alpha} \! \cdot \! \left( \Pi_h^M \boldsymbol{p}^{n+1}- \Pi_h^M \boldsymbol{p}^{n}\right) \right), e_{p_j}^{h,n+1}   \right)\\
			& -\left(L \alpha_j \boldsymbol{\alpha} \! \cdot \!  \left( \Delta t \partial_t \boldsymbol{p}^{n+1}-     \Pi_h^{M}  \boldsymbol{p}^{n+1}+\Pi_h^M \boldsymbol{p}^n \right),e_{p_j}^{h,n+1}   \right)\\
			&+ \left( L \alpha_j \boldsymbol{\alpha} \! \cdot  \! \left(  \Delta t \partial_t \boldsymbol{p}^{n+1}-    \Pi_h^{M}  \boldsymbol{p}^{n}+ \Pi_h^M \boldsymbol{p}^{n-1} \right),e_{p_j}^{h,n+1}    \right) \\
			&+ \left(  \frac{\alpha_j}{\lambda} 
			\left( \Delta t \partial_t \xi^{n+1}- \Pi_h^W \xi^{n}+\Pi_h^W \xi^{n-1} \right),e_{p_j}^{h,n+1}  \right).
			%+\sum_{i=1}^A \left( s_{j\leftarrow i} \left(p^{n+1}_j-p^{n+1}_i -\Pi_h^M \left(  p^{n+1}_j-p^{n+1}_i \right)  \right), \psi_{j,h}  \right).
		\end{align*}
	Summing over $j=1,\cdots,A$, we obtain
	\begin{align}
		& \sum_{j=1}^A a_3\left(  {D_{p_j}^{h,n+1}} ,e_{p_j}^{h,n+1} \right)+ \left(   \frac{1}{\lambda}  {\boldsymbol{\alpha} \! \cdot \!  \boldsymbol{D}_p^{n+1}}  , \boldsymbol{\alpha} \! \cdot \! \boldsymbol{e}_{p}^{h,n+1} \right) - \left(  \frac{1}{\lambda} 
		D_\xi^{n} , \boldsymbol{\alpha}\! \cdot  \! \boldsymbol{e}_{p}^{h,n+1} \right) \notag \\
		&+\left( L       \boldsymbol{\alpha} \! \cdot \! \left( \boldsymbol{D}_p^{n+1} -\boldsymbol{D}_p^n \right) , \boldsymbol{\alpha} \! \cdot \! \boldsymbol{e}_{p}^{h,n+1}  \right) +\Delta t \sum_{j=1}^A d(e_{p_j}^{h,n+1},e_{p_j}^{h,n+1}  ) \label{errforp} \\
		=& - \sum_{j=1}^A a_3\left( \Delta t \partial_t p_j^{n+1}- \left( \Pi_h^{M_j}p_j^{n+1}-\Pi_h^{M_j}p_j^n\right) ,e_{p_j}^{h,n+1}   \right) \notag \\
		& -\left(\frac{1}{\lambda} \left( \Delta t \boldsymbol{\alpha} \! \cdot \!  \partial_t \boldsymbol{p}^{n+1}-  \boldsymbol{\alpha} \! \cdot \! \left( \Pi_h^M \boldsymbol{p}^{n+1}- \Pi_h^M \boldsymbol{p}^{n}\right) \right),\boldsymbol{\alpha} \! \cdot \! \boldsymbol{e}_{p}^{h,n+1}   \right)\notag \\
		& -\left(L  \boldsymbol{\alpha} \! \cdot \! \left( \Delta t \partial_t \boldsymbol{p}^{n+1}-     \Pi_h^{M}  \boldsymbol{p}^{n+1}+\Pi_h^M \boldsymbol{p}^n \right), \boldsymbol{\alpha } \! \cdot \! \boldsymbol{e}_{p}^{h,n+1}   \right) \notag \\
		&+ \left( L  \boldsymbol{\alpha} \! \cdot \! \left(  \Delta t \partial_t \boldsymbol{p}^{n+1}-    \Pi_h^{M}  \boldsymbol{p}^{n}+ \Pi_h^M \boldsymbol{p}^{n-1} \right), \boldsymbol{\alpha} \! \cdot \! \boldsymbol{e}_{p}^{h,n+1}    \right) \notag \\
		&+ \left(  \frac{1}{\lambda} 
		\left( \Delta t \partial_t \xi^{n+1}- \Pi_h^W \xi^{n}+\Pi_h^W \xi^{n-1} \right), \boldsymbol{\alpha} \! \cdot \! \boldsymbol{e}_{p}^{h,n+1}  \right).\notag
		%+\sum_{i=1}^A \left( s_{j\leftarrow i} \left(p^{n+1}_j-p^{n+1}_i -\Pi_h^M \left(  p^{n+1}_j-p^{n+1}_i \right)  \right), \psi_{j,h}  \right).
	\end{align}
	Summing equations \eqref{errorequationtotal2} and \eqref{errforp} up and then summing over the index $n$ from 1 to $N$, we have
		\begin{align}
			&\sum_{n=1}^N\left[a_1\left(e_{\boldsymbol{u}}^{h, n+1}, D_{\boldsymbol{u}}^{n+1}\right)+a_2\left(D_{\xi}^{n+1}, e_{\xi}^{h, n+1}\right)-\left( \frac{1}{\lambda} \boldsymbol{\alpha}\! \cdot \! \boldsymbol{D}_p^{n+1},e_\xi^{h,n+1} \right) \right.\notag \\
			& \left. + \sum_{j=1}^A a_3\left(  {D_{p_j}^{h,n+1}} ,e_{p_j}^{h,n+1} \right)  + \left(   \frac{1}{\lambda}  {\boldsymbol{\alpha} \! \cdot \!  \boldsymbol{D}_p^{n+1}}  , \boldsymbol{\alpha} \! \cdot \! \boldsymbol{e}_{p}^{h,n+1} \right) +\left( L       \boldsymbol{\alpha} \! \cdot  \! \boldsymbol{D}_p^{n+1}  , \boldsymbol{\alpha} \! \cdot \! \boldsymbol{e}_{p}^{h,n+1}  \right)\right. \label{errorall3}  \\
			&\left. - \left(  \frac{1}{\lambda} 
			D_\xi^{n+1} , \boldsymbol{\alpha} \! \cdot \! \boldsymbol{e}_{p}^{h,n+1} \right)+\Delta t \sum_{j=1}^A d\left(e_{p_j}^{h,n+1},e_{p_j}^{h,n+1}  \right)  \right]=\sum_{i=1}^{11}E_i. \notag
		\end{align}
	where
	\begin{align*} 
		& E_1=\sum_{n=1}^N b\left(\boldsymbol{u}^{n+1}-\boldsymbol{u}^n-\Delta t \partial_t \boldsymbol{u}^{n+1}, e_{\xi}^{h,n+1}\right), \\ 
		& E_2=\sum_{n=1}^N - \left( \frac{1}{\lambda} \boldsymbol{\alpha} \! \cdot \! \left( \Pi_h^M \boldsymbol{p}^n- \Pi_h^M \boldsymbol{p}^{n-1} - \Delta t \partial_t  \boldsymbol{p}^{n+1} \right) ,e_\xi^{h,n+1} \right),\\
		& E_3=\sum_{n=1}^N a_2\left(\Pi_h^W\xi^{n+1}-\Pi_h^W \xi^{n}-\Delta t \partial_t \xi^{n+1},e_\xi^{h,n+1}\right), \\ 
		& E_4=\sum_{n=1}^N  \sum_{j=1}^A a_3\left(  \Pi_h^{M_j}p_j^{n+1}-\Pi_h^{M_j}p_j^n-\Delta t \partial_t p_j^{n+1} ,e_{p_j}^{h,n+1}   \right), \\ 
		& E_5=\sum_{n=1}^N \frac{1}{\lambda}  \left(   \boldsymbol{\alpha} \! \cdot \! \left( \Pi_h^M \boldsymbol{p}^{n+1}- \Pi_h^M \boldsymbol{p}^{n}\right) -\Delta t \boldsymbol{\alpha} \! \cdot \!  \partial_t \boldsymbol{p}^{n+1},\boldsymbol{\alpha} \! \cdot \! \boldsymbol{e}_{p}^{h,n+1}   \right), \\ 
		& E_6=\sum_{n=1}^N \left(  \frac{1}{\lambda} 
		\left( \Delta t \partial_t \xi^{n+1}- \Pi_h^W \xi^{n}+\Pi_h^W \xi^{n-1} \right), \boldsymbol{\alpha} \! \cdot \! \boldsymbol{e}_{p}^{h,n+1}  \right),\\
		& E_7=\sum_{n=1}^N  \left(L  \boldsymbol{\alpha} \!\cdot \!  \left(      \Pi_h^{M}  \boldsymbol{p}^{n+1}-\Pi_h^M \boldsymbol{p}^n -\Delta t \partial_t \boldsymbol{p}^{n+1} \right), \boldsymbol{\alpha } \! \cdot \! \boldsymbol{e}_{p}^{h,n+1}   \right),\\
		& E_8=\sum_{n=1}^N  \left( L  \boldsymbol{\alpha} \! \cdot \! \left(  \Delta t \partial_t \boldsymbol{p}^{n+1}-    \Pi_h^{M}  \boldsymbol{p}^{n}+ \Pi_h^M \boldsymbol{p}^{n-1} \right), \boldsymbol{\alpha} \! \cdot \! \boldsymbol{e}_{p}^{h,n+1}    \right),\\
		&E_9=\sum_{n=1}^N\left( \frac{1}{\lambda} \boldsymbol{\alpha} \! \cdot \! \left( \boldsymbol{D}_p^n -\boldsymbol{D}_p^{n+1} \right)  ,e_\xi^{h,n+1} \right),\\
		& E_{10}=\sum_{n=1}^N \left(  \frac{1}{\lambda} 
		\left(	D_\xi^{n} -D_\xi^{n+1} \right), \boldsymbol{\alpha} \! \cdot \! \boldsymbol{e}_{p}^{h,n+1} \right),\\
		& E_{11}=\sum_{n=1}^N L \left(  \boldsymbol{\alpha} \! \cdot \! \boldsymbol{D}_p^{n}, \boldsymbol{\alpha} \! \cdot \! \boldsymbol{e}_{p}^{h,n+1}\right).\\
	\end{align*}
	Noting the following identity,
		\begin{align*}
			&	\frac{1}{\lambda} \left(  \boldsymbol{\alpha} \! \cdot \! \boldsymbol{D}_p^{n+1}\! -\!D_\xi^{n+1}, \boldsymbol{\alpha} \! \cdot \! \boldsymbol{e}_p^{h,n+1} \! -\! e_\xi^{h,n+1}  \right)\\
			=&\frac{1}{\lambda} \left( \! \boldsymbol{\alpha} \! \cdot \! \boldsymbol{D}_p^{n+1}, \boldsymbol{\alpha} \! \cdot \! \boldsymbol{e}_p^{n+1} \right)-\frac{1}{\lambda} \left( \boldsymbol{\alpha} \! \cdot \! \boldsymbol{D}_p^{n+1},e_\xi^{h,n+1}\right)-\frac{1}{\lambda} \left( D_\xi^{n+1}, \boldsymbol{\alpha} \! \cdot \! \boldsymbol{e}_p^{h,n+1}\right)+\frac{1}{\lambda} \left( D_\xi^{n+1},e_\xi^{h,n+1} \right).
		\end{align*}
	Using the definitions of $a_2(\cdot,\cdot)$ and $a_3(\cdot,\cdot)$, we can rewrite the left-hand side of \eqref{errorall3} as follows,
		\begin{align}
			& \sum_{n=1}^N \! \left[ \! \mu \! \left( \!\left\|\varepsilon(e_{\boldsymbol{u}}^{h,n+1})\right\|_{L^2(\Omega)}^2 \!- \! \left\|\varepsilon (e_{\boldsymbol{u}}^{h,n})\right\|_{L^2(\Omega)}^2     \!  \right) \! +\! \sum_{j=1}^A \! \frac{c_j}{2} \! \left( \left\| e_{p_j}^{h,n+1} \right\|_{L^2(\Omega)}^2-\left\| e_{p_j}^{h,n} \right\|_{L^2(\Omega)}^2 \! \right) \right. \notag \\
			& \left. +\frac{1}{2\lambda}   \left(  \left\| \boldsymbol{\alpha} \! \cdot \! \boldsymbol{e}_p^{h,n+1} \! -\! e_{\xi}^{h,n+1}\right\|_{L^2(\Omega)}^2  \!  - \! \left\|  \boldsymbol{\alpha} \! \cdot \! \boldsymbol{e}_p^{h,n} \! -\! e_{\xi}^{h,n}\right\|_{L^2(\Omega)}^2      \right) \right. \label{errorall4} \\
			& \left.	 +\frac{L}{2}   \left( \left\| \boldsymbol{\alpha} \! \cdot \! \boldsymbol{e}_p^{h,n+1}\right\|_{L^2(\Omega)}^2 \! -\! \left\|  \boldsymbol{\alpha} \! \cdot \! \boldsymbol{e}_p^{h,n}\right\|_{L^2(\Omega)}^2    \right) \! + \! \Delta t \sum_{j=1}^A d\left(e_{p_j}^{h,n+1},e_{p_j}^{h,n+1}  \right) \right]  \leq \sum_{i=1}^{11} E_i.\notag
		\end{align}
	We proceed by estimating the terms \( E_i \) for \( i = 1, 2, \ldots, 11 \). Using the definition of \( b(\cdot, \cdot) \), combined with the Cauchy-Schwarz and Young's inequalities, as well as the properties of the projection operators, we derive the following bounds for \( E_1 \) to \( E_3 \) for any \( \epsilon_1 > 0 \).
	\begin{align*}
		E_1 
		& \leq \frac{\epsilon_1}{3}\Delta t \sum_{n=1}^N\left\|e_{\xi}^{h,n+1}\right\|_{L^2(\Omega)}^2+\frac{C}{\epsilon_1}(\Delta t)^2 \int_0^T\left\|\partial_{t t} \boldsymbol{u}\right\|_{H^1(\Omega)}^2 d s.
	\end{align*}
	For $E_2$, we can obtain
	\begin{align*}
		E_2
		%=& \sum_{n=1}^N  \left( \frac{1}{\lambda} \boldsymbol{\alpha} \cdot \left( \Delta t \partial_t  \boldsymbol{p}^{n+1}-\Pi_h^M \boldsymbol{p}^n+ \Pi_h^M \boldsymbol{p}^{n-1}  \right) ,e_\xi^{h,n+1} \right)\\
		= & \sum_{n=1}^N \frac{1}{\lambda} \int_{\Omega} e_{\xi}^{h,n+1}  \boldsymbol{\alpha} \cdot \left(\Delta t \partial_t \boldsymbol{p}^{n+1}-\Delta t \partial_t \boldsymbol{p}^n\right)\\
		&+\sum_{n=1}^N \frac{1}{\lambda} \int_{\Omega} e_{\xi}^{h,n+1} \boldsymbol{\alpha} \cdot \left(\Delta t \partial_t \boldsymbol{p}^n-\Pi_h^M \boldsymbol{p}^n+ \Pi_h^M \boldsymbol{p}^{n-1}\right) \\
		\leq & \frac{\epsilon_1}{3}\Delta t \sum_{n=1}^N\left\|e_{\xi}^{h,n+1}\right\|_{L^2(\Omega)}^2+\frac{C \alpha_m^2}{\epsilon_1 \lambda^2 \Delta t} \sum_{n=1}^N\left(\left\|\Delta t \partial_t \boldsymbol{p}^{n+1}-\Delta t \partial_t \boldsymbol{p}^n\right\|_{L^2(\Omega)}^2\right. \\
		& \left.+\left\|\boldsymbol{p}^n-\boldsymbol{p}^{n-1}-\Delta t \partial_t \boldsymbol{p}^n\right\|_{L^2(\Omega)}^2+\left\| \Pi_h^M \left(\boldsymbol{p}^n-\boldsymbol{p}^{n-1}\right)-\left(\boldsymbol{p}^n-\boldsymbol{p}^{n-1}\right)\right\|_{L^2(\Omega)}^2\right)\\
		\leq& \frac{\epsilon_1}{3}\Delta t  \! \sum_{n=1}^N\left\|e_{\xi}^{h,n+1}\right\|_{L^2(\Omega)}^2 \! +\! \frac{C \alpha_m^2}{\epsilon_1 \lambda^2}\left[(\Delta t)^2 \! \! \int_0^T \! \! \left\|\partial_{t t} \boldsymbol{p}\right\|_{L^2(\Omega)}^2 d s  +h^{2 l+2} \! \! \int_0^T \! \! \left\|\partial_t \boldsymbol{p} \right\|_{H^{l+1}(\Omega)}^2 d s \! \right].
	\end{align*}
	And for $E_3$, 
	\begin{align*} E_3 \leq & \frac{\epsilon_1}{3} \Delta t \sum_{n=1}^N\left\|e_{\xi}^{h,n+1}\right\|_{L^2(\Omega)}^2+\frac{C}{\epsilon_1 \lambda^2}\left[(\Delta t)^2 \int_0^T \! \! \left\|\partial_{t t} \xi\right\|_{L^2(\Omega)}^2 d s \right.\\
		&\left.+h^{2 k}  \int_0^T \! \! \left(\left\|\partial_t \boldsymbol{u}\right\|_{H^{k+1}(\Omega)}^2+\left\|\partial_t \xi\right\|_{H^k(\Omega)}^2\right) d s\right].
	\end{align*}
	Similarly, by using the definition of \(a_3(\cdot, \cdot)\), we have
	\begin{align*}
		E_4
		%=& \sum_{n=1}^N  \sum_{j=1}^A a_3\left(  \Pi_h^{M_j}p_j^{n+1}-\Pi_h^{M_j}p_j^n-\Delta t \partial_t p_j^{n+1} ,e_{p_j}^{h,n+1}   \right)\\
		%= & \sum_{n=1}^N \sum_{j=1}^A c_j \int_{\Omega} e_{p_j}^{h,n+1}\left(\Pi_h^{M_j} p_j^{n+1}- \Pi_h^{M_j} p_j^n-\Delta t \partial_t p_j^{n+1}\right)  \\
		\leq &\sum_{n=1}^N \! \Delta t \! \left( \sum_{j=1}^A\frac{c_j}{2}  \left\|e_p^{h,n+1}\right\|_{L^2(\Omega)}^2 \! \right) \! +\! \frac{Cc_m}{2}\! \left[ \! (\Delta t)^2 \! \int_0^T \! \! \left\|\partial_{t t} \boldsymbol{p} \right\|_{L^2}^2 d s \! + \! h^{2 l+2} \!\int_0^T \! \! \left\|\partial_t \boldsymbol{p}\right\|_{H^{l+1}}^2 d s \! \right].
	\end{align*}
	For any \( \epsilon_2 > 0 \), we obtain the following bound for the term \( E_5 \),
	\begin{align*}
		E_5=&\sum_{n=1}^N \frac{1}{\lambda}  \left(   \boldsymbol{\alpha} \! \cdot \! \left( \Pi_h^M \boldsymbol{p}^{n+1}- \Pi_h^M \boldsymbol{p}^{n}\right) -\Delta t \boldsymbol{\alpha} \! \cdot \! \partial_t \boldsymbol{p}^{n+1},\boldsymbol{\alpha} \! \cdot \! \boldsymbol{e}_{p}^{h,n+1}   \right)\\
		\leq &\frac{\epsilon_2 }{10} \Delta t \sum_{n=1}^N\left\| \boldsymbol{\alpha}\! \cdot \! \boldsymbol{e}_p^{h,n+1}\right\|_{L^2(\Omega)}^2  +C \frac{\alpha_m^2}{\epsilon_2 \lambda^2} \left[(\Delta t)^2 \int_0^T \! \! \left\|\partial_{t t}  \boldsymbol{p} \right\|_{L^2(\Omega)}^2 d s \right. \\
		 & \left. +h^{2 l+2} \int_0^T \! \! \left\|\partial_t \boldsymbol{p} \right\|_{H^{l+1}(\Omega)}^2 d s\right].
	\end{align*}
	Using the same \( \epsilon_2 \), we can bound \( E_6 \):
	\begin{align*}
		E_6
		%=&\sum_{n=1}^N \left(  \frac{1}{\lambda} \left( \Delta t \partial_t \xi^{n+1}- \Pi_h^W \xi^{n}+\Pi_h^W \xi^{n-1} \right), \boldsymbol{\alpha} \cdot \boldsymbol{e}_{p}^{h,n+1}  \right)\\
		%= & \sum_{n=1}^N \frac{1}{\lambda} \int_{\Omega} \boldsymbol{\alpha} \cdot \boldsymbol{e}_p^{h,n+1}\left(\Delta t \partial_t \xi^{n+1}-  \Pi_h^W  \xi^{n}+  \Pi_h^W  \xi^{n-1}\right) \\
		\leq & \frac{\epsilon_2 }{10} \Delta t \sum_{n=1}^N\left\| \boldsymbol{\alpha} \! \cdot \! \boldsymbol{e}_p^{h,n+1}\right\|_{L^2(\Omega)}^2+\frac{C}{\epsilon_2 \lambda^2}\left[(\Delta t)^2 \int_0^T\left\|\partial_{t t} \xi\right\|_{L^2(\Omega)}^2 d s \right. \\
		&\left. +h^{2 k}  \int_0^T\left(\left\|\partial_t \boldsymbol{u}\right\|_{H^{k+1}(\Omega)}^2+\left\|\partial_t \xi\right\|_{H^k(\Omega)}^2\right) d s\right].
	\end{align*}
	$E_7$ and $E_8$ can be bounded as follows:
	\begin{align*}
		E_7
		=&\sum_{n=1}^N  \left(L  \boldsymbol{\alpha} \! \cdot \!  \left(      \Pi_h^{M}  \boldsymbol{p}^{n+1}-\Pi_h^M \boldsymbol{p}^n -\Delta t \partial_t \boldsymbol{p}^{n+1} \right), \boldsymbol{\alpha } \! \cdot \! \boldsymbol{e}_{p}^{h,n+1}   \right)\\
		\leq &C\frac{L^2 \alpha_m^2}{\epsilon_2 }\left[(\Delta t)^2 \int_0^T\left\|\partial_{t t} \boldsymbol{p} \right\|_{L^2(\Omega)}^2 d s+h^{2 l+2}  \int_0^T\left\|\partial_t  \boldsymbol{p} \right\|_{H^{l+1}(\Omega)}^2 d s\right]\\
		&+ \frac{\epsilon_2 }{10} \Delta t \sum_{n=1}^N\left\| \boldsymbol{\alpha} \! \cdot  \! \boldsymbol{e}_p^{h,n+1}\right\|_{L^2(\Omega)}^2 .\\
		E_8
		%=&\sum_{n=1}^N  \left( L  \boldsymbol{\alpha} \cdot \left(  \Delta t \partial_t \boldsymbol{p}^{n+1}-    \Pi_h^{M}  \boldsymbol{p}^{n}+ \Pi_h^M \boldsymbol{p}^{n-1} \right), \boldsymbol{\alpha} \cdot \boldsymbol{e}_{p}^{h,n+1}    \right)\\
		=&\sum_{n=1}^N  L\left( \boldsymbol{\alpha} \! \cdot \! \left( \Delta t \partial_t \boldsymbol{p}^{n+1}-\Delta t \partial_t \boldsymbol{p}^{n}\right), \boldsymbol{\alpha} \! \cdot \! \boldsymbol{e}_p^{h,n+1} \right) \\
		&+ L\left( \boldsymbol{\alpha} \! \cdot \! \Delta t \partial_t \boldsymbol{p}^{n}- \boldsymbol{\alpha} \cdot \left({\Pi_h^M \boldsymbol{p}^{n}- \Pi_h^M \boldsymbol{p}^{n-1}}\right), \boldsymbol{\alpha} \! \cdot  \! \boldsymbol{e}_p^{h,n+1} \right)\\
		\leq 
		&C\frac{L^2 \alpha_m^2}{\epsilon_2 }\left[(\Delta t)^2 \int_0^T\left\|\partial_{t t} \boldsymbol{p} \right\|_{L^2(\Omega)}^2 d s+h^{2 l+2} \int_0^T\left\|\partial_t \boldsymbol{p} \right\|_{H^{l+1}(\Omega)}^2 d s\right]\\
		&+ \frac{\epsilon_2 }{10}\Delta t \sum_{n=1}^N\left\| \boldsymbol{\alpha} \! \cdot \! \boldsymbol{e}_p^{h,n+1}\right\|_{L^2(\Omega)}^2 .
	\end{align*}
	We then estimate the term \(E_9\) for any \(\epsilon_3 > 0\),
	\begin{align*}
		E_9=&\sum_{n=1}^N\left( \frac{1}{\lambda} \boldsymbol{\alpha} \! \cdot\! \left( \boldsymbol{D}_p^n -\boldsymbol{D}_p^{n+1} \right)  ,e_\xi^{h,n+1} \right)\\
		=&\sum_{n=1}^N \left( -\frac{1}{\lambda}\left( \boldsymbol{\alpha} \!
		\cdot \! \boldsymbol{D}_p^{n+1}, e_{\xi}^{h,n+1}\right)+ \frac{1}{\lambda}\left( \boldsymbol{\alpha} \! \cdot \! \boldsymbol{D}_p^{n}, e_{\xi}^{h,n}\right) +\frac{1}{\lambda}\left(\boldsymbol{\alpha} 
		\! \cdot \! \boldsymbol{D}_p^{n}, D_{\xi}^{n+1}\right) \right)\\
		=& -\frac{1}{\lambda}\left( \boldsymbol{\alpha} \!
		\cdot \! \boldsymbol{D}_p^{N+1}, e_{\xi}^{h,N+1}\right)+ \frac{1}{\lambda}\left( \boldsymbol{\alpha} \! \cdot \! \boldsymbol{D}_p^{1}, e_{\xi}^{h,1}\right) +\sum_{n=1}^N \frac{1}{\lambda}\left(\boldsymbol{\alpha} 
		\! \cdot \! \boldsymbol{D}_p^{n}, D_{\xi}^{n+1}\right)  \\
		\leq &\frac{\epsilon_3}{2} \left\|  e_{\xi}^{h,N+1}\right\|_{L^2(\Omega)}^2+\frac{1}{2\epsilon_3 \lambda^2} \left\| \boldsymbol{\alpha} \! \cdot \! \boldsymbol{D}_{p}^{N+1}\right\|_{L^2(\Omega)}^2+\frac{1}{2\lambda} \left\| \boldsymbol{\alpha} \! \cdot \!   \boldsymbol{D}_p^1  \right\|_{L^2(\Omega)}^2+\frac{1}{2\lambda} \left\|   e_\xi^{h,1}  \right\|_{L^2(\Omega)}^2\\
		&+\sum_{n=1}^N \left( \frac{1}{2\lambda}  \left\| \boldsymbol{\alpha} \! \cdot \! \boldsymbol{D}_p^{n} \right\|_{L^2(\Omega)}^2+\frac{1}{2\lambda}  \left\|  D_\xi^{n+1} \right\|_{L^2(\Omega)}^2   \right).
	\end{align*}
	Similarly, we bound $E_{10}$ with any $\epsilon_4>0$,
	\begin{align*}
		E_{10}
        %=&\sum_{n=1}^N \left(  \frac{1}{\lambda} \left(	D_\xi^{n} -D_\xi^{n+1} \right), \boldsymbol{\alpha} \cdot \boldsymbol{e}_{p}^{h,n+1} \right)\\
		%=&\sum_{n=1}^N \left( -\frac{1}{\lambda}\left(D_\xi^{n+1}, \boldsymbol{\alpha} \cdot \boldsymbol{e}_{p}^{h,n+1}\right)+ \frac{1}{\lambda}\left(D_\xi^{n}, \boldsymbol{\alpha} \cdot \boldsymbol{e}_{p}^{h,n}\right) +\frac{1}{\lambda}\left(D_\xi^{n}, \boldsymbol{\alpha} \cdot \boldsymbol{D}_{p}^{n+1}\right) \right)\\
		%=& -\frac{1}{\lambda}\left(D_\xi^{N+1}, \boldsymbol{\alpha} \cdot \boldsymbol{e}_{p}^{h,N+1}\right)+ \frac{1}{\lambda}\left(D_\xi^{1}, \boldsymbol{\alpha} \cdot \boldsymbol{e}_{p}^{h,1}\right) +\sum_{n=1}^N\frac{1}{\lambda}\left(D_\xi^{n}, \boldsymbol{\alpha} \cdot \boldsymbol{D}_{p}^{n+1}\right) \\
		\leq &\frac{\epsilon_4 }{2} \left\|  \boldsymbol{\alpha} \! \cdot \! \boldsymbol{e}_{p}^{h,N+1}\right\|_{L^2(\Omega)}^2+\frac{1}{2\epsilon_4 \lambda^2} \left\|  D_{\xi}^{N+1}\right\|_{L^2(\Omega)}^2+\frac{1}{2\lambda} \left\|   D_\xi^1  \right\|_{L^2(\Omega)}^2+\frac{1}{2\lambda} \left\|  \boldsymbol{\alpha} \! \cdot \! \boldsymbol{e}_p^{h,1}  \right\|_{L^2(\Omega)}^2\\
		&+\sum_{n=1}^N \left( \frac{1}{2\lambda}  \left\|  \boldsymbol{\alpha} 
		\! \cdot \! \boldsymbol{D}_p^{n+1} \right\|_{L^2(\Omega)}^2+\frac{1}{2\lambda}  \left\|  D_\xi^{n} \right\|_{L^2(\Omega)}^2   \right),
	\end{align*}
	and for the last one $E_{11}$, we consider the same $\epsilon_2$,
	\begin{align*}
		E_{11}
        %=&\sum_{n=1}^N L \left(  \boldsymbol{\alpha} \cdot \boldsymbol{D}_p^{n}, \boldsymbol{\alpha} \cdot \boldsymbol{e}_{p}^{h,n+1}\right)\\
		\leq & \sum_{n=1}^N \frac{CL^2}{\epsilon_2  \Delta t} \left\| \boldsymbol{\alpha} \! \cdot \! \boldsymbol{D}_p^n \right\|_{L^2(\Omega)}^2+\frac{\epsilon_2  }{10}\Delta t \sum_{n=1}^N  \left\|  \boldsymbol{\alpha} \! \cdot \! \boldsymbol{e}_p^{h,n+1} \right\|_{L^2(\Omega)}^2.
	\end{align*}
	
	By using the bounds of $E_1$ to $E_{11}$ in \eqref{errorall4}, we have
	\begin{align} 
		& \mu \! \left( \! \left\|\varepsilon(e_{\boldsymbol{u}}^{h,N+1})\right\|_{L^2(\Omega)}^2 \! -\! \left\|\varepsilon (e_{\boldsymbol{u}}^{h,1})\right\|_{L^2(\Omega)}^2       \right)\! +\! \sum_{j=1}^A\!  \frac{c_j}{2} \!\left(\! \left\|  e_{p_j}^{h,N+1} \right\|_{L^2(\Omega)}^2 \! -\! \left\| e_{p_j}^{h,1} \right\|_{L^2(\Omega)}^2 \! \right) \! \notag \\ 
		& \! + \! \frac{L}{2} \!  \left( \left\| \boldsymbol{\alpha} \! \cdot \! \boldsymbol{e}_p^{h,N+1}\right\|_{L^2(\Omega)}^2 \! -\! \left\|  \boldsymbol{\alpha} \! \cdot \! \boldsymbol{e}_p^{h,1}\right\|_{L^2(\Omega)}^2    \right) \!+\! \frac{1}{2\lambda}   \left(  \left\| \boldsymbol{\alpha} \! \cdot \! \boldsymbol{e}_p^{h,N+1} \!  - \! e_{\xi}^{h,N+1}\right\|_{L^2(\Omega)}^2\right. \label{errorall5} \\
		& \left.  - \left\|  \boldsymbol{\alpha}\! \cdot \! \boldsymbol{e}_p^{h,1}-e_{\xi}^{h,1}\right\|_{L^2(\Omega)}^2   \!   \right)  + \Delta t\sum_{n=1}^N \sum_{j=1}^A d\left(e_{p_j}^{h,n+1},e_{p_j}^{h,n+1}  \right) \notag \\
		\leq & {\epsilon_1}\Delta t \sum_{n=1}^N\left\|e_{\xi}^{h,n+1}\right\|_{L^2(\Omega)}^2+\sum_{n=1}^N \Delta t  \left( \! \sum_{j=1}^A \! \frac{c_j}{2}  \left\|e_p^{h,n+1}\right\|_{L^2(\Omega)}^2 \! \right) \!+\! \frac{\epsilon_2  }{2}\Delta t \sum_{n=1}^N \! \left\|  \boldsymbol{\alpha} \! \cdot \! \boldsymbol{e}_p^{h,n+1} \! \right\|_{L^2(\Omega)}^2  \notag \\
		&+ \frac{\epsilon_3}{2}\! \left\|  e_{\xi}^{h,N+1}\right\|_{L^2(\Omega)}^2\!+\! \frac{1}{2\epsilon_3 \lambda^2} \left\| \boldsymbol{\alpha} \! \cdot  \! \boldsymbol{D}_{p}^{N+1}\right\|_{L^2(\Omega)}^2 \! +\! \frac{1}{2\lambda} \left\| \boldsymbol{\alpha} \!\cdot \!   \boldsymbol{D}_p^1  \right\|_{L^2(\Omega)}^2 \!+\! \frac{1}{2\lambda} \left\|   e_\xi^{h,1}  \right\|_{L^2(\Omega)}^2 \notag \\
		&+	\! \frac{\epsilon_4 }{2} \left\|  \boldsymbol{\alpha} \! \cdot \! \boldsymbol{e}_{p}^{h,N+1}\right\|_{L^2(\Omega)}^2\!+\! \frac{1}{2\epsilon_4 \lambda^2} \left\|  D_{\xi}^{N+1}\right\|_{L^2(\Omega)}^2\!+\! \frac{1}{2\lambda} \left\|   D_\xi^1  \right\|_{L^2(\Omega)}^2\!+\! \frac{1}{2\lambda} \left\|  \boldsymbol{\alpha} \! \cdot \! \boldsymbol{e}_p^{h,1} \! \right\|_{L^2(\Omega)}^2 \notag \\
		&+\sum_{n=1}^N \left( \frac{1}{\lambda}  \left\|  \boldsymbol{\alpha} 
		\! \cdot \! \boldsymbol{D}_p^{n+1} \right\|_{L^2(\Omega)}^2+\frac{1}{\lambda}  \left\|  D_\xi^{n+1} \right\|_{L^2(\Omega)}^2   \right) +\sum_{n=1}^N \frac{CL^2}{\epsilon_2  \Delta t} \left\| \boldsymbol{\alpha} \! \cdot \! \boldsymbol{D}_p^n \right\|_{L^2(\Omega)}^2 \notag \\
		&+ {C}\left[(\Delta t)^2 \int_0^T\left(\left\|\partial_{t t} \boldsymbol{u}\right\|_{H^1(\Omega)}^2+\left\|\partial_{t t} \xi\right\|_{L^2(\Omega)}^2+\left\|\partial_{t t} \boldsymbol{p} \right\|_{L^2(\Omega)}^2\right) d s\right. \notag \\
		& +h^{2 k}  \int_0^T\left(\left\|\partial_t \boldsymbol{u}\right\|_{H^{k+1}(\Omega)}^2+\left\|\partial_t \xi\right\|_{H^k(\Omega)}^2\right) d s  \left.+h^{2 l+2}  \int_0^T\left\|\partial_t  \boldsymbol{p} \right\|_{H^{l+1}(\Omega)}^2 d s\right]. \notag
	\end{align}
	Similarly, applying the inf-sup condition yields
	\begin{equation}\label{boundxi}
	\tilde{\beta} \! \left\|e_{\xi}^{h,n+1}\right\|_{L^2(\Omega)} \leq \sup _{\boldsymbol{v} \in {\boldsymbol{V}}_h} \frac{b\! \left(\! \boldsymbol{v}, e_{\xi}^{h,n+1} \! \right)}{\left\|\boldsymbol{v}\right\|_{H^1(\Omega)}}\!= \! \sup _{\boldsymbol{v} \in \boldsymbol{V}_h} \frac{a_1 \! \left(e_{\boldsymbol{u}}^{h,n+1}, \boldsymbol{v}\right)}{\left\|\boldsymbol{v}\right\|_{H^1(\Omega)}} \leq 2 \mu \left\|\varepsilon \! \left(\! e_{\boldsymbol{u}}^{h,n+1}\! \right)\right\|_{L^2(\Omega)},
	\end{equation}
	which directly implies
	\begin{equation}\label{boundp}
		\begin{aligned}
			\frac{1}{2}\left\| \boldsymbol{\alpha} \! \cdot \! \boldsymbol{e}_p^{h,n+1}\right\|_{L^2(\Omega)}^2 & \! \leq \! \left\|  \boldsymbol{\alpha} \! \cdot \! \boldsymbol{e}_p^{h,n+1} \!- \! {e}_{\xi}^{h,n+1}\right\|_{L^2(\Omega)}^2+\left\|e_{\xi}^{h,n+1}\right\|_{L^2(\Omega)}^2 \\
			& \leq\left\| \boldsymbol{\alpha} \! \cdot \! \boldsymbol{e}_p^{h,n+1}\!- \! e_{\xi}^{h,n+1}\right\|_{L^2(\Omega)}^2+ \left( \frac{2 \mu }{\tilde{\beta}}\right)^2 \left\|\varepsilon \! \left(e_{\boldsymbol{u}}^{h,n+1}\right)\right\|_{L^2(\Omega)}^2.
		\end{aligned}
	\end{equation}
	Thus, we determine coefficient $\epsilon_1=\frac{\tilde{\beta}^2}{16\mu }$, such that
	$$ {\epsilon_1} \left\|e_{\xi}^{h,n+1}\right\|_{L^2(\Omega)}^2 \leq \frac{\mu}{4}\left\|\varepsilon \! \left(e_{\boldsymbol{u}}^{h,n+1}\right)\right\|^2_{L^2(\Omega)}.$$
	Determine coefficient $\epsilon_2=\min \left\{\frac{1}{4 \lambda}, \frac{\tilde{\beta}^2}{16  \mu}\right\}$, such that
	$$ \frac{\epsilon_2 }{2} \left\| \boldsymbol{\alpha} \! \cdot \! \boldsymbol{e}_p^{h,n+1}\right\|_{L^2(\Omega)}^2 \leq \frac{\mu}{4}\left\|\varepsilon \! \left(e_u^{h,n+1}\right)\right\|_{L^2(\Omega)}^2+\frac{1}{4 \lambda}\left\| \boldsymbol{\alpha} \! \cdot \! \boldsymbol{e}_p^{h,n+1}-e_{\xi}^{h,n+1}\right\|_{L^2(\Omega)}^2.$$
	Taking $\epsilon_3=\frac{\tilde{\beta}^2}{16\mu }$, such that
	$$ \frac{\epsilon_3}{2} \left\|e_{\xi}^{h,N+1}\right\|_{L^2(\Omega)}^2 \leq \frac{\mu}{4}\left\|\varepsilon \! \left(e_{\boldsymbol{u}}^{h,N+1}\right)\right\|^2_{L^2(\Omega)}.$$
	Similarly, taking $\epsilon_4=\min \left\{\frac{1}{4 \lambda}, \frac{\tilde{\beta}^2}{16  \mu}\right\}$
	such that
	$$\frac{\epsilon_4 }{2} \left\| \boldsymbol{\alpha}\! \cdot \! \boldsymbol{e}_{p}^{h,N+1}\right\|_{L^2(\Omega)}^2\! \leq \frac{\mu}{4} \! \left\|\varepsilon \! \left(e_u^{h,N+1}\right)\right\|_{L^2(\Omega)}^2+\frac{1}{4 \lambda}\left\| \boldsymbol{\alpha} \! \cdot \! \boldsymbol{e}_p^{h,N+1}-e_{\xi}^{h,N+1}\right\|_{L^2(\Omega)}^2.$$
	Also, terms $\| \boldsymbol{\alpha} \! \cdot \! \boldsymbol{D}_p^{n} \|_{L^2(\Omega)}^2$ and $\| D_\xi^n \|_{L^2(\Omega)}^2$ can be bounded by \eqref{th1eq}. 
	Regarding the error contributed by the first time step, we obtain the following from the a priori analysis for the coupled scheme at the initial time:
		\begin{align*}
			& \mu \left\|\varepsilon\left(e_{\boldsymbol{u}}^{h, 1}\right)\right\|_{L^2(\Omega)}^2+ \frac{1}{2\lambda} \left\|   \boldsymbol{\alpha} \! \cdot \! \boldsymbol{e}_p^{h,1}-e_\xi^{h,1} \right\|_{L^2(\Omega)}^2+\sum_{j=1}^A \frac{c_j}{2}  \left\| e_{p_j}^{h,1} \right\|_{L^2(\Omega)}^2\\
			&+ \Delta t \sum_{j=1}^A   { \kappa_j } \left\|  \nabla e_{p_j}^{h,1} \right\|^2_{L^2(\Omega)}    	+\Delta t  \sum_{i,j=1}^A  \frac{s_{j\leftarrow i}}{2} \left\|  e^{h,1}_{p_j}-e_{p_i}^{h,1}  \right\|_{L^2(\Omega)}^2   \\
			\leq	& C \left( (\Delta t)^2 \int_0^{t_1}\left(\left\|\partial_{t t} \boldsymbol{u}\right\|_{H^1(\Omega)}^2+\left\|\partial_{t t} \xi\right\|_{L^2(\Omega)}^2+\left\|\partial_{t t} \boldsymbol{p} \right\|_{L^2(\Omega)}^2\right) d s \right. \\
			& \left. +h^{2 k} \int_0^{t_1}\left(\left\|\partial_t \boldsymbol{u}\right\|_{H^{k+1}(\Omega)}^2+\left\|\partial_t \xi\right\|_{H^k(\Omega)}^2\right) d s+h^{2 l+2} \int_0^{t_1}\left\|\partial_t  \boldsymbol{p} \right\|_{H^{l+1}(\Omega)}^2 d s \right).
		\end{align*}
	Thus, by \eqref{errorall5}, we have
		\begin{align}
			& \frac{\mu}{2}  \left\|\varepsilon(e_{\boldsymbol{u}}^{h,N+1})\right\|_{L^2(\Omega)}^2   + \sum_{j=1}^A \frac{c_j}{2}  \left\| e_{p_j}^{h,N+1} \right\|_{L^2(\Omega)}^2 +\frac{L}{2}    \left\| \boldsymbol{\alpha} \! \cdot \! \boldsymbol{e}_p^{h,N+1}\right\|_{L^2(\Omega)}^2 \notag  \\
			& +\frac{1}{4\lambda}     \left\| \boldsymbol{\alpha} \! \cdot \! \boldsymbol{e}_p^{h,N+1} \! - \! e_{\xi}^{h,N+1}\right\|_{L^2(\Omega)}^2  \!+\!  \Delta t  \sum_{n=1}^N \sum_{j=1}^A \kappa_j \left\|  \nabla e_{p_j}^{h,n+1} \right\|_{L^2(\Omega)}^2 \label{errorall6} \\
			& +\frac{\Delta t}{2} \sum_{n=1}^N \sum_{i,j=1}^A \|  s_{j\leftarrow i}^{1/2} \left( e^{h,n+1}_{p_j}-e^{h,n+1}_{p_i} \right) \|_{L^2(\Omega)}^2 \notag \\
			\leq & \Delta t\sum_{n=1}^N \! \left( \! \frac{\mu}{2}\left\|\varepsilon \! \left(e_u^{h,n+1}\right)\right\|_{L^2(\Omega)}^2\! + \! \frac{1}{4 \lambda}\left\| \boldsymbol{\alpha} \! \cdot \! \boldsymbol{e}_p^{h,n+1} \!-\! e_{\xi}^{h,n+1}\right\|_{L^2(\Omega)}^2 \! +\! \sum_{j=1}^A \frac{c_j}{2}  \! \left\| e_{p_j}^{h,N+1} \right\|_{L^2(\Omega)}^2   \! \right) \notag \\
			&+ {C}\left[(\Delta t)^2 \int_0^T\left(\left\|\partial_{t t} \boldsymbol{u}\right\|_{H^1(\Omega)}^2+\left\|\partial_{t t} \xi\right\|_{L^2(\Omega)}^2+\left\|\partial_{t t} \boldsymbol{p} \right\|_{L^2(\Omega)}^2\right) d s\right.\notag \\
			& +h^{2 k} \int_0^T\left(\left\|\partial_t \boldsymbol{u}\right\|_{H^{k+1}(\Omega)}^2+\left\|\partial_t \xi\right\|_{H^k(\Omega)}^2\right) d s  \left.+h^{2 l+2}  \int_0^T\left\|\partial_t \boldsymbol{p} \right\|_{H^{l+1}(\Omega)}^2 d s\right].\notag
		\end{align}
	Finally, we apply the discrete Gr\"{o}nwall’s inequality in \eqref{errorall6} to obtain
		\begin{align*}
			&\frac{\mu}{2}  \left\|\varepsilon(e_{\boldsymbol{u}}^{h,N+1})\right\|_{L^2(\Omega)}^2   + \sum_{j=1}^A \frac{c_j}{2}  \left\| e_{p_j}^{h,N+1} \right\|_{L^2(\Omega)}^2 +\frac{L}{2}    \left\| \boldsymbol{\alpha} \! \cdot \! \boldsymbol{e}_p^{h,N+1}\right\|_{L^2(\Omega)}^2\\
			& +\frac{1}{4\lambda}     \left\| \boldsymbol{\alpha}  \! \cdot \! \boldsymbol{e}_p^{h,N+1}-e_{\xi}^{h,N+1}\right\|_{L^2(\Omega)}^2    +  \Delta t  \sum_{n=1}^N \sum_{j=1}^A \kappa_j \left\|  \nabla e_{p_j}^{h,n+1} \right\|_{L^2(\Omega)}^2 \\
			&+\frac{\Delta t}{2} \sum_{n=1}^N \sum_{i,j=1}^A \|  s_{j\leftarrow i}^{1/2} \left( e^{h,n+1}_{p_j}-e^{h,n+1}_{p_i} \right) \|_{L^2(\Omega)}^2\\ 
			\leq 
			& {C}\left[(\Delta t)^2 \int_0^T\left(\left\|\partial_{t t} \boldsymbol{u}\right\|_{H^1(\Omega)}^2+\left\|\partial_{t t} \xi\right\|_{L^2(\Omega)}^2+\left\|\partial_{t t} \boldsymbol{p} \right\|_{L^2(\Omega)}^2\right) d s\right. \\
			& +h^{2 k} \int_0^T\left(\left\|\partial_t \boldsymbol{u}\right\|_{H^{k+1}(\Omega)}^2+\left\|\partial_t \xi\right\|_{H^k(\Omega)}^2\right) d s  \left.+h^{2 l+2}  \int_0^T\left\|\partial_t \boldsymbol{p} \right\|_{H^{l+1}(\Omega)}^2 d s\right].
		\end{align*}
	Thus, we can obtain the desired result \eqref{th3eq} by \eqref{boundp} and \eqref{boundxi}. This completes the proof.
\end{proof}

\section{Numerical experiments}\label{experience}
This section is divided into two parts to evaluate the effectiveness of the proposed algorithms: convergence tests and a benchmark study. In the convergence tests, we examine the convergence order of the stabilized parallel scheme under various parameter settings. In particular, we verify the locking-free property of the scheme as \(\lambda \to \infty\) and analyze its convergence behavior when the fluid storage coefficient \(c_j \to 0\) and the diffusion coefficient \(\kappa_j \to 0\).
In the benchmark study, we test the proposed algorithm in a brain simulation involving a four-pressure network \cite{lee_mixed_2019}. All algorithms are implemented using FreeFEM++ \cite{freefem}.

\subsection{Tests for convergence}
 Let the computational domain be $\Omega=[0,1]^2$. We choose the body force $\boldsymbol{f}$ and the volumetric source/sink term $q_j$ in \eqref{govern1} so that the exact solution is as follows,
\begin{equation*}
	\begin{aligned}
		&	\boldsymbol{u}(x,y,t)= \left(
		\begin{array}{lr}
			e^{-t}\left(\sin\left(2\pi y\right) \left(-1+\cos(2\pi x)\right)+\dfrac{1}{\mu+\lambda}\sin(\pi x)\sin(\pi y)\right)  \\
			e^{-t}\left(\sin\left(2\pi x\right) \left(1-\cos(2\pi y)\right)+\dfrac{1}{\mu+\lambda}\sin(\pi x)\sin(\pi y)\right)  
		\end{array}
		\right), \\
		&	p_j(x,y,t)=e^{-jt}\sin(\pi x) \sin(\pi y), \quad j=1,2.
	\end{aligned}
\end{equation*}
The following parameters are adopted for testing convergence rates with respect to time step refinement,
$$E = 1, \quad \nu = 0.4, \quad c_j = 10^{-7}, \quad \alpha_j = 1.0, \quad \kappa_j = 10^{-7}, \quad s_{j \leftarrow i} = 0.1,$$ 
where \(i, j = 1, 2\). 

We fix the mesh size \(h = \frac{1}{64}\), \(k = 3\), and \(l = 2\) for spatial discretization and refine the time step size \(\Delta t\). We calculate the convergence rates of the \(L^2\) and \(H^1\) semi-norms for displacement and pressure at time \(T = 0.5\). The errors and convergence rates in time are summarized in Table \ref{table1}. The results show that the \(L^2\) errors for \(\boldsymbol{u}\) and \(\nabla \boldsymbol{u}\), as well as the \(L^2\) errors for \(p\) and \(\nabla p\), \(\xi\) and \(\nabla \xi\), all achieve an order of approximately 1 as \(\Delta t\) is refined, aligning with the theoretical analysis.

\begin{table}[ht]
	\centering
	\small
	\caption{Time refinement convergence rates for Example 1.}
	\label{table1}
    	\resizebox{0.9\textwidth}{!}{%
	\begin{tabular}{ccccccc}
		\hline
		$\Delta t$ &$\|\boldsymbol{u}-\boldsymbol{u}_h \|_{L^2(\Omega)}$  & rate &   $\| \xi-\xi_h \|_{L^2(\Omega)}$ & rate &   $\| \boldsymbol{p}-\boldsymbol{p}_h \|_{L^2(\Omega)}$ & rate   \\
		\hline
		1/8 &       5.029e-03   &         &   1.621e-02       &     &   6.787e-03   &                 \\
		1/16  &2.208e-03 &  1.19&   7.182e-03 &  1.17 &  3.891e-03  & 0.80   \\
		1/32 &1.069e-03  & 1.05 &  3.485e-03  & 1.04 &  1.976e-03  & 0.98  \\
		1/64  &5.280e-04  & 1.02  & 1.723e-03  & 1.02  & 9.910e-04   &1.00  \\	
		\hline
		$\Delta t$ &$\|\nabla \left( \boldsymbol{u}-\boldsymbol{u}_h \right) \|_{L^2(\Omega)}$  & rate &   $\| \nabla ( \xi-  \xi_h) \|_{L^2(\Omega)}$ & rate &   $\| \nabla ( \boldsymbol{p}- \boldsymbol{p}_h) \|_{L^2(\Omega)}$ & rate   \\
		\hline
		1/8 &       2.393e-02   &         &   1.218e-01       &     &   2.840e-02   &                 \\
		1/16  &1.051e-02 &  1.19&   5.324e-02 &  1.19 &  1.524e-02  & 0.90   \\
		1/32 &5.088e-03  & 1.05 &  2.576e-02  & 1.05 &  7.666e-03  & 0.99  \\
		1/64  &2.514e-03  & 1.02  & 1.273e-02  & 1.02  & 3.832e-03   &1.00  \\	
		\hline
	\end{tabular}
    }
\end{table}

\begin{table}[ht]
	\centering
	\small
	\caption{Numerical results of first case \eqref{case1} with $k=2,l=1$.}
	\label{commytable1}
	\resizebox{1\textwidth}{!}{%
	\begin{tabular}{cccccccc}
		\hline
		$h$ & $\Delta t$ &$\|\boldsymbol{u}-\boldsymbol{u}_h \|_{L^2(\Omega)}$  & rate &   $\| \boldsymbol{p}-\boldsymbol{p}_h \|_{L^2(\Omega)}$ & rate & $\| \xi-\xi_h \|_{L^2(\Omega)}$  & rate \\
		\hline
		1/4 &$1/8$ &        3.543e-02 &    &     3.510e-02    &   &      7.649e-02 &\\
		1/8 & $1/32$ & 4.192e-03  &    3.08     &  1.056e-02    &    1.73     &   1.638e-02  &    2.22       \\
		1/16 & $1/128$ & 5.634e-04    &   2.90    &   2.649e-03    &    1.99    &    3.922e-03  &     2.06  \\
		1/32 &  $1/512$ & 1.016e-04  &     2.47    &   6.603e-04    &    2.00    &    9.725e-04&    2.01         \\  
		\hline
		$h$ & $\Delta t$ &$\| \nabla \left( \boldsymbol{u}-\boldsymbol{u}_h \right) \|_{L^2(\Omega)}$  & rate &   $\| \nabla \left( \boldsymbol{p}-\boldsymbol{p}_h \right) \|_{L^2(\Omega)}$ & rate & $\| \nabla \left( \xi-\xi_h \right) \|_{L^2(\Omega)}$  & rate \\
		\hline
		1/4 &$1/8$ &        8.701e-01 &    &    5.037e-01   &   &      2.031e+00 &\\
		1/8 & $1/32$ & 2.394e-01  &    1.86     & 2.606e-01    &   0.95     &   9.648e-01 &    1.07       \\
		1/16 & $1/128$ & 6.164e-02    &   1.96    &  1.318e-01    &    0.98   &    4.762e-01 &     1.02  \\
		1/32 &  $1/512$ &1.554e-02 &     1.99   &   6.607e-02   &    1.00    &    2.373e-01 &    1.00         \\  
		\hline
	\end{tabular}
}
\end{table}

\begin{table}[ht]
	\centering
	\small
	\caption{Numerical results of first case \eqref{case1} with $k=3,l=2$.}
	\label{commytable3}
	\resizebox{1\textwidth}{!}{%
	\begin{tabular}{cccccccc}
		\hline
		$h$ & $\Delta t$ &$\|\boldsymbol{u}-\boldsymbol{u}_h \|_{L^2(\Omega)}$  & rate &   $\| \boldsymbol{p}-\boldsymbol{p}_h \|_{L^2(\Omega)}$ & rate & $\| \xi-\xi_h \|_{L^2(\Omega)}$  & rate \\
		\hline
		1/4 &$1/8$ &       7.445e-03 &    &     8.339e-03   &   &      1.488e-02 &\\
		1/8 & $1/64$ & 4.680e-04 &    3.99     &   1.404e-03    &    2.57     &   1.462e-03  &    3.35       \\
		1/16 & $1/512$ & 3.910e-05   &   3.58    &   1.716e-04   &    3.03    &    1.805e-04 &     3.02  \\
		1/32 &  $1/4096$ & 4.192e-06  &     3.22    &   2.136e-05    &    3.01    &    2.255e-05&    3.00         \\  
		\hline
		$h$ & $\Delta t$ &$\| \nabla \left( \boldsymbol{u}-\boldsymbol{u}_h \right) \|_{L^2(\Omega)}$  & rate &   $\| \nabla \left( \boldsymbol{p}-\boldsymbol{p}_h \right) \|_{L^2(\Omega)}$ & rate & $\| \nabla \left( \xi-\xi_h \right) \|_{L^2(\Omega)}$  & rate \\
		\hline
		1/4 &$1/8$ &        1.333e-01 &    &    8.340e-02   &   &      3.777e-01&\\
		1/8 & $1/64$ & 1.625e-02  &    3.04     & 2.047e-02   &   2.03     &   8.938e-02 &    2.08       \\
		1/16 & $1/512$ & 1.991e-03   &   3.03    &  5.090e-03  &    2.01   &    2.081e-02 &     2.10  \\
		1/32 &  $1/4096$ &2.467e-04 &     3.01   &   1.275e-03  &    2.00    &    5.001e-03&    2.06        \\  
		\hline
	\end{tabular}
}
\end{table}

\begin{table}[ht]
	\centering
	\small
	\caption{Numerical results of second case \eqref{case2} with $k=2,l=1$.}
	\label{commytable5}
	\resizebox{1\textwidth}{!}{%
	\begin{tabular}{cccccccc}
		\hline
		$h$ & $\Delta t$ &$\|\boldsymbol{u}-\boldsymbol{u}_h \|_{L^2(\Omega)}$  & rate &   $\| \boldsymbol{p}-\boldsymbol{p}_h \|_{L^2(\Omega)}$ & rate & $\| \xi-\xi_h \|_{L^2(\Omega)}$  & rate \\
		\hline
		1/4 &$1/8$ &       3.491e-02 &    &     3.714e-01    &   &      1.405e-01 &\\
		1/8 & $1/32$ &  4.085e-03  &    3.10     &  1.122e-01   &    1.73     &   2.623e-02  &    2.42       \\
		1/16 & $1/128$ & 4.736e-04    &   3.11    &   2.996e-02    &    1.90    &    6.027e-03   &     2.12  \\
		1/32 &  $1/512$ & 5.739e-05  &     3.04    &    7.681e-03   &   1.96    &    1.480e-03&    2.03         \\  
		\hline
		$h$ & $\Delta t$ &$\| \nabla \left( \boldsymbol{u}-\boldsymbol{u}_h \right) \|_{L^2(\Omega)}$  & rate &   $\| \nabla \left( \boldsymbol{p}-\boldsymbol{p}_h \right) \|_{L^2(\Omega)}$ & rate & $\| \nabla \left( \xi-\xi_h \right) \|_{L^2(\Omega)}$  & rate \\
		\hline
		1/4 &$1/8$ &       8.721e-01&    &    2.227e+00   &   &  3.532e+00   &\\
		1/8 & $1/32$ & 2.381e-01&    1.87     &  6.725e-01 &   1.73    &    1.585e+00 &    1.16       \\
		1/16 & $1/128$ & 6.120e-02   &   1.96    &2.099e-01&     1.68 &    7.754e-01 &     1.03  \\
		1/32 &  $1/512$ &1.542e-02  &     1.99   &    7.817e-02 &    1.43      &    3.857e-01 &    1.01         \\  
		\hline
	\end{tabular}
}
\end{table}

\begin{table}[ht]
	\centering
	\small
	\caption{Numerical results of second case \eqref{case2} with $k=3,l=2$.}
	\label{commytable7}
	\resizebox{1\textwidth}{!}{%
	\begin{tabular}{cccccccc}
		\hline
		$h$ & $\Delta t$ &$\|\boldsymbol{u}-\boldsymbol{u}_h \|_{L^2(\Omega)}$  & rate &   $\| \boldsymbol{p}-\boldsymbol{p}_h \|_{L^2(\Omega)}$ & rate & $\| \xi-\xi_h \|_{L^2(\Omega)}$  & rate \\
		\hline
		1/4 &$1/8$ &      6.664e-03 &    &     5.258e-02    &   &      2.162e-02 &\\
		1/8 & $1/32$ &  3.986e-04 &    4.06     &  7.812e-03  &    2.75     &   2.367e-03  &    3.19       \\
		1/16 & $1/128$ &2.340e-05   &   4.09    &   1.025e-03  &    2.93    &    2.601e-04   &     3.18  \\
		1/32 &  $1/512$ &1.425e-06  &     4.04    &    1.308e-04   &   2.97    &    3.021e-05&    3.11         \\  
		\hline
		$h$ & $\Delta t$ &$\| \nabla \left( \boldsymbol{u}-\boldsymbol{u}_h \right) \|_{L^2(\Omega)}$  & rate &   $\| \nabla \left( \boldsymbol{p}-\boldsymbol{p}_h \right) \|_{L^2(\Omega)}$ & rate & $\| \nabla \left( \xi-\xi_h \right) \|_{L^2(\Omega)}$  & rate \\
		\hline
		1/4 &$1/8$ &       1.339e-01&    &    3.194e-01  &   &  7.812e-01   &\\
		1/8 & $1/32$ & 1.635e-02&    3.03     &  4.986e-02&   2.68    &   1.768e-01 &    2.14       \\
		1/16 & $1/128$ & 2.002e-03  &   3.03    &7.835e-03&     2.67 &   3.845e-02 &     2.20  \\
		1/32 &  $1/512$ &2.480e-04 &     3.01   &   1.484e-03  &    2.40      &    8.754e-03&    2.14         \\  
		\hline
	\end{tabular}
}
\end{table}

%Next, to highlight the advantages of our proposed algorithms, we compare the computational time about the parallel algorithm and fully coupled scheme with backward Euler difference in time \cite{lee_mixed_2019}. We take the finite element discrete space $k=2,l=1$, and compare the calculational time at $T=0.5$ with $dt=10^{-4}$ and parameters in \eqref{case1}.

Next, we proceed to evaluate the spatial convergence rates for the \(L^2\) and energy norms of displacement and pressure at time \(T = 0.5\). To verify the results, we consider two distinct parameter sets. For the first case, the parameters are chosen as follows,
\begin{equation}\label{case1}
	E=1,\quad \nu=0.3, \quad c_j=1, \quad \alpha_j=1.0, \quad \kappa_j=1,\quad s_{j\leftarrow i}=0.01,
\end{equation}
where $i,j=1,2.$ For the second case, we consider the more general choice of parameters, i.e
\begin{equation}\label{case2}
	E=1,\quad \nu=0.499999999, \quad c_j=10^{-7},\quad  \kappa =10^{-6} , \quad \alpha_j=1,  \quad s_{j\leftarrow i}=0.01,
\end{equation}
where $i,j=1,2.$ In this case, the Lam\'{e} constant $\lambda$ is $1.6667\times 10^{8}$.

For the above two cases, the results obtained by our parallel scheme are listed in Tables \ref{commytable1}-\ref{commytable7}. In Table \ref{commytable1} and Table \ref{commytable5}, the finite element discrete space is set to \(k = 2\) and \(l = 1\). The results in these tables indicate that the \(H^1\) error rate for displacement \(\boldsymbol{u}\) is 2, while the \(L^2\) error rates for \(\xi\) and pressure \( \boldsymbol{p} \) are 2, and the \(H^1\) error rate for \( \boldsymbol{p} \) is 1.
Meanwhile, in Tables \ref{commytable3} and Table \ref{commytable7}, the discrete finite element space is set to \(k = 3\) and \(l = 2\). The results show that the \(H^1\) error rate for displacement \(\boldsymbol{u}\) is 3, the \(L^2\) error rates for \(\xi\) and pressure \( \boldsymbol{p} \) are also 3, and the \(H^1\) error rate for \( \boldsymbol{p} \) is 2. Both error convergence rates are consistent with our theoretical analysis.
Furthermore, we observed that for the incompressible case with \(\nu = 0.499999999\), the convergence rate of \(\boldsymbol{u}\) in the \(L^2\) norm is optimal. For general values of \(\nu=0.3\), the convergence rate of \(\boldsymbol{u}\) in the \(L^2\)-norm decreases in order \cite{lee_mixed_2019}. However, the \(H^1\) seminorm for \(\boldsymbol{u}\) remains optimal, consistent with our theoretical analysis.

The numerical tests demonstrate that the proposed algorithms are unconditionally stable and locking-free. More importantly, the parallel splitting algorithm allows simultaneous solving of both sub-problems. To highlight its advantages, we compared the CPU times of three methods with parameter setting \eqref{case1}: the fully coupled, sequential, and parallel methods. Both fully coupled and sequential methods use a first-order backward difference scheme in time, with the sequential method solving the Stokes system first and then substituting \(\xi\) into the elliptic system to compute pressure \cite{cai23jsc}. Table \ref{commytabletime} shows that the parallel scheme significantly reduces the computation time while maintaining errors comparable to the fully coupled method for the same discretization size. In contrast, iterative decoupled methods involve more complex computations at each time level, leading to longer computation times than the fully coupled method. The iterative algorithm remains highly efficient with a proper choice of stabilization parameters \cite{Storvik2019}.

\begin{table}[ht]
	\centering
	\small
	\caption{Comparison of CPU time at $T=1$, with $k=2,l=1$.		\label{commytabletime}}
	\resizebox{1\textwidth}{!}{%
	\begin{tabular}{lccccc}
		\hline
		&$h$ & $\Delta t$ &$\| \boldsymbol{u}- \boldsymbol{u}_h \|_{L^2} \&   \| \boldsymbol{p}-\boldsymbol{p}_h \|_{L^2}$  & CPU time(s)& Time Reduction \\
		\hline
		Fully Coupled Scheme&      $1/40$ &$10^{-2}$  &      $4.07 \times 10^{-4}$  \& $8.20\times 10^{-4}$         & 13.9  &  -  \\
		Sequtial Scheme&       $1/40$ &$10^{-2}$&     $4.05\times10^{-4}$ \& $8.51\times 10^{-4}$    &  8.0 & 42.5\%   \\
		Parallel Scheme&          $1/40$ &$10^{-2}$ &      $2.95\times 10^{-4}$    \& $2.33\times 10^{-3}$      &  5.6 & 59.7\%   \\
		\hline
		Fully Coupled Scheme&      1/80 &$10^{-4}$ &      $3.46\times10^{-5}$  \&     $1.15\times10^{-4}$      &   559&    \\
		Sequtial Scheme&      1/80 &$10^{-4}$ &     $3.46\times10^{-5}$  \&     $1.16\times10^{-4}$     &  366 &34.5\%    \\
		Parallel Scheme&      1/80 &$10^{-4}$  &    $3.20\times10^{-5}$  \&     $3.43\times10^{-4}$         & 210  &   62.4\%  \\
		\hline
	\end{tabular}
}
\end{table}

\subsection{Applications in brain simulation using 4-network model}
In this subsection, following \cite{lee_mixed_2019}, we employ the 4-network poroelastic model to simulate the deformation of human brain tissue using physiologically inspired parameters and boundary conditions. Specifically, we solve the MPET equations for \(A = 4\), where the four networks represent \cite{meptori}: (1) extracellular spaces filled with interstitial fluid (or a paravascular network), (2) arteries, (3) veins, and (4) capillaries. The computational domain and grid are taken from \cite{cai21}, as illustrated in \cref{brainmesh}, and the parameter values are listed in Table \ref{pv2}.
	\begin{figure}[htbp]
	\centering
\subfigure{\includegraphics[width=0.5\textwidth]{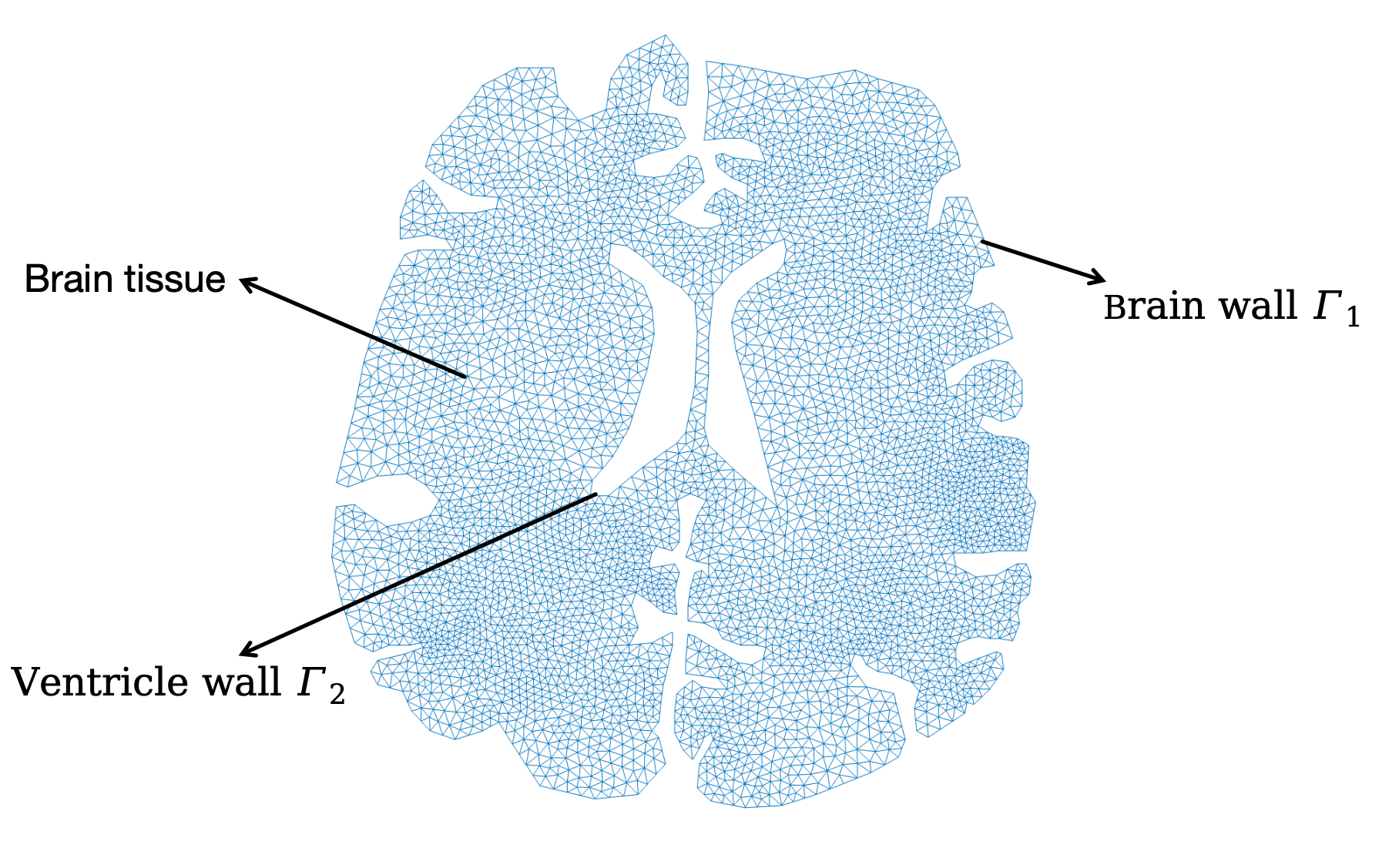}}
%    \captionsetup{skip=-3pt}
	\caption{The mesh for brain.}
	\label{brainmesh}
\end{figure}

\begin{table}[ht]
	\centering
    \small
	\caption{Parameters values, referenced from \cite{lee_mixed_2019}.			\label{pv2}}
    	\resizebox{0.9\textwidth}{!}{%
	\begin{tabular}{llll}
		\hline Parameters & Values &Parameters & Values \\
		\hline$\nu$ & 0.4999 &	$\alpha_2, \alpha_4$ & 0.25 \\
		$E$ & 1500 Pa &	$\alpha_3$ & 0.01  \\
		$c_1$ & $3.9 \times 10^{-4}$  $\mathrm{~Pa}^{-1}$ &	$\kappa_1$ & $1.57 \times 10^{-5}$  $\mathrm{~mm}^2 \mathrm{~Pa}^{-1} \mathrm{~s}^{-1}$   \\
		$c_2, c_4$ & $2.9 \times 10^{-4}$ $\mathrm{~Pa}^{-1}$ & 	$\kappa_2, \kappa_3, \kappa_4$, & $3.75 \times 10^{-2}$  $\mathrm{~mm}^2 \mathrm{~Pa}^{-1} \mathrm{~s}^{-1}$  \\
		$c_3$ & $1.5 \times 10^{-5}$  $\mathrm{~Pa}^{-1}$ & 	$\xi_{2 \leftarrow 4}, \xi_{4 \leftarrow 3}, \xi_{4 \leftarrow 1}, \xi_{1 \leftarrow 3}$ & $1.0 \times 10^{-6}$  $\mathrm{~Pa}^{-1} \mathrm{~s}^{-1}$ \\
		$\alpha_1$ & 0.49 &	$\xi_{1 \leftarrow 2}, \xi_{2 \leftarrow 3}$ & 0 \\
		\hline
	\end{tabular}
    }
\end{table}
\begin{figure}[htbp]
	\centering
	% 第一行
	\subfigure[{Results of $p_1$.}]{\includegraphics[width=0.24\textwidth]{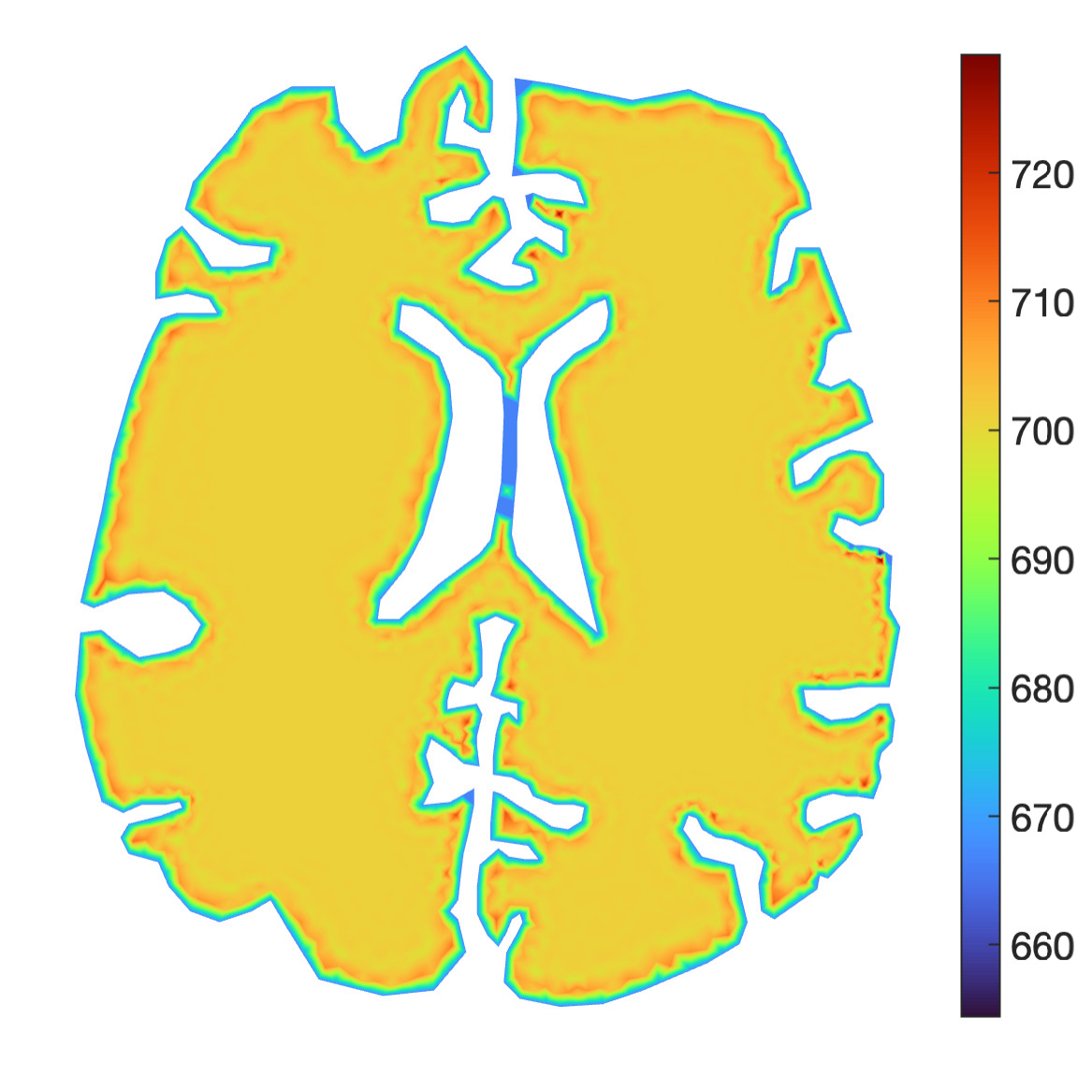}}
	\subfigure[Results of $p_2$.]{\includegraphics[width=0.24\textwidth]{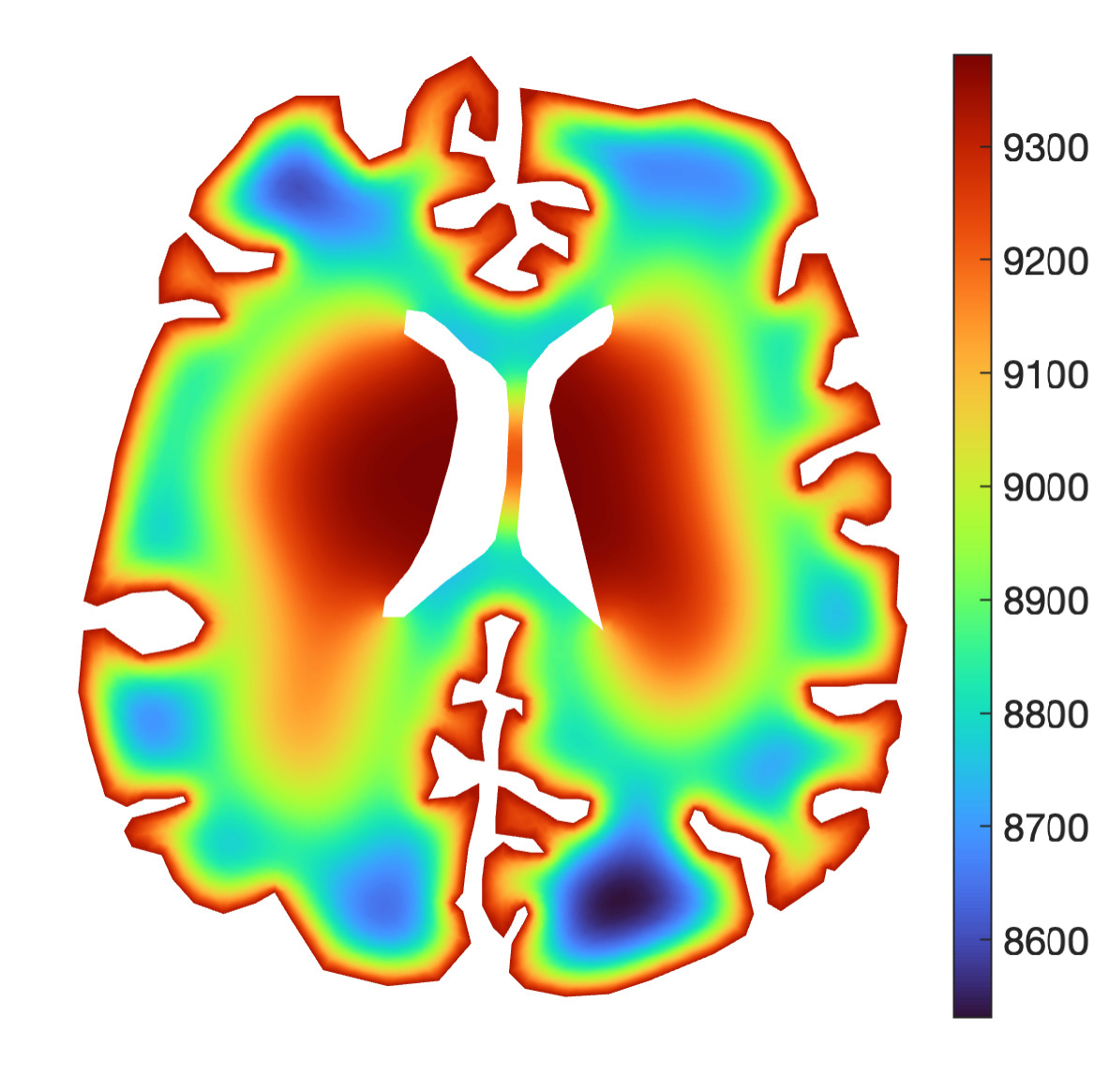}}
	\subfigure[Results of $p_3$.]{\includegraphics[width=0.24\textwidth]{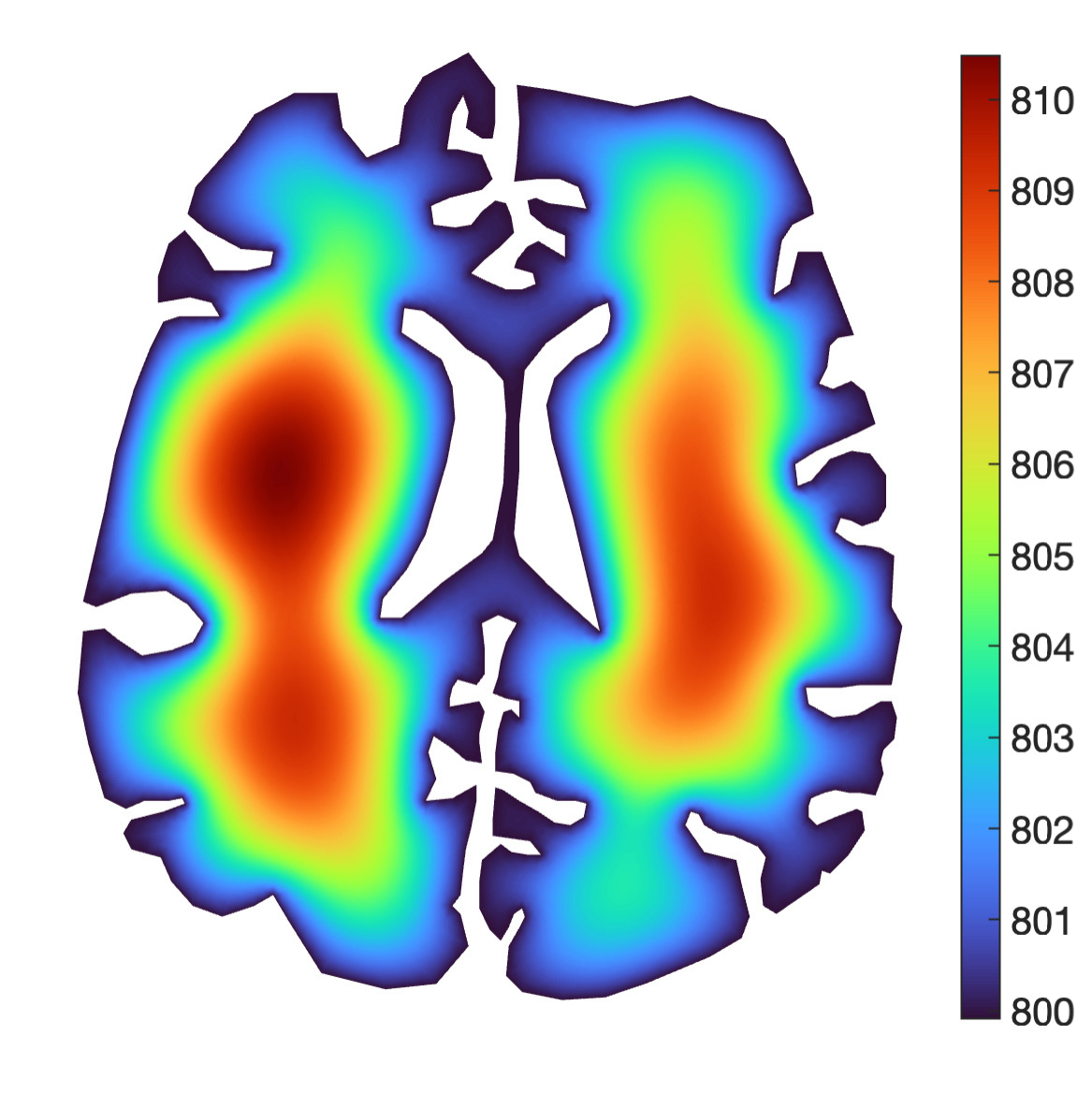}}
	\subfigure[Results of $p_4$.]{\includegraphics[width=0.24\textwidth]{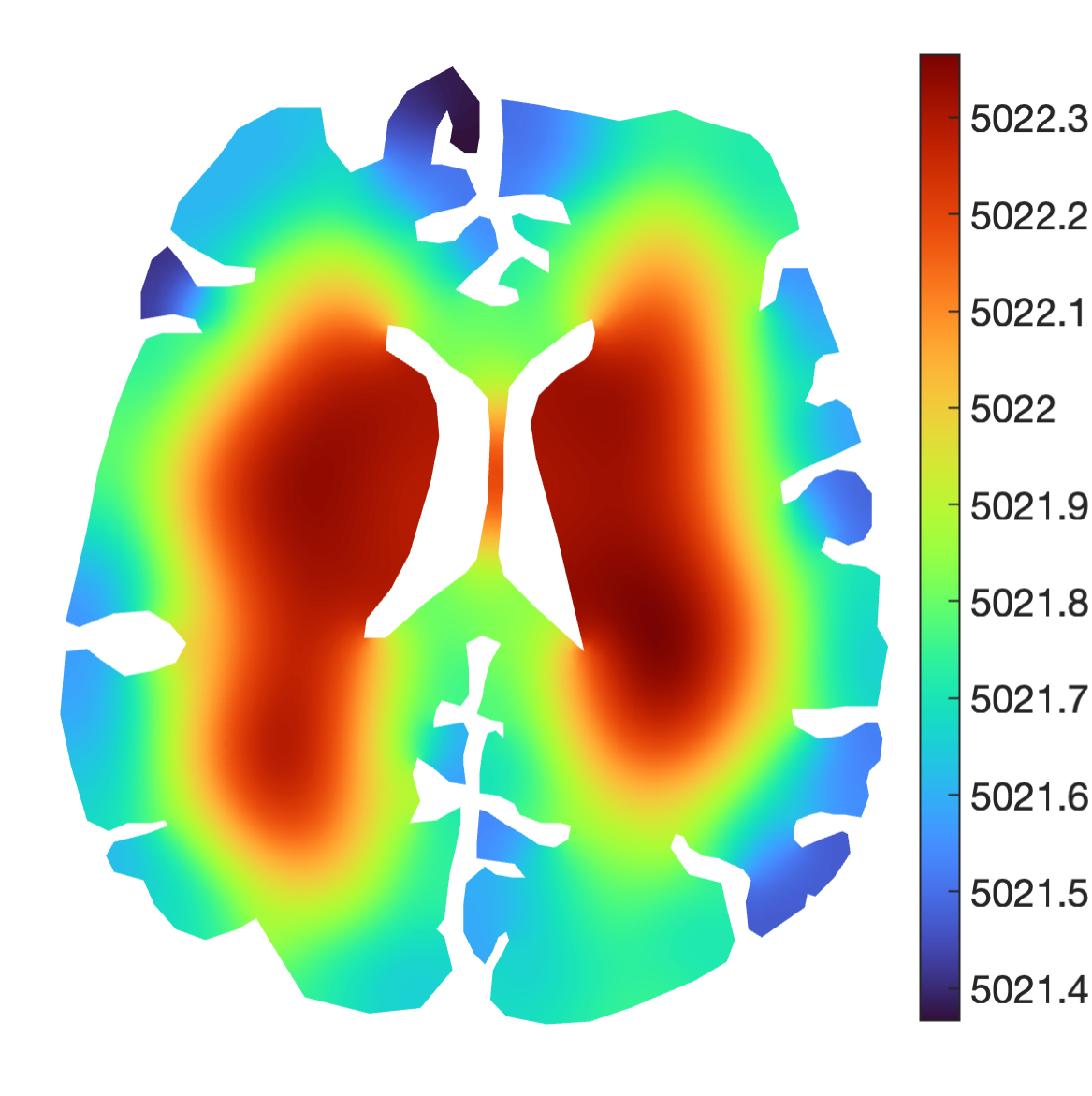}}
	
	\vspace{0.5cm} % 调整两行之间的间距
	\subfigure[Results of $p_1$.]{\includegraphics[width=0.24\textwidth]{couplep1forT=3.eps}}
	\subfigure[Results of $p_2$ .]{\includegraphics[width=0.24\textwidth]{couplep2forT=3.eps}}
	\subfigure[Results of $p_3$.]{\includegraphics[width=0.24\textwidth]{couplep3forT=3.eps}}
	\subfigure[Results of $p_4$.]{\includegraphics[width=0.24\textwidth]{couplep4forT=3.eps}}	
	\caption{Pressure distribution of the numerical solution for the 4-network model (first row: fully coupled scheme; second row: parallel splitting scheme).}
	\label{fig:multiple_images}
\end{figure}

Reference from \cite{lee_mixed_2019}, we consider the boundary conditions for this system for all $t \in [0, T]$. The displacement is fixed on the skull boundary and prescribe total stress on the ventricular wall,
\begin{align*}
	\boldsymbol{u}=\mathbf{0} \quad \text { on } \Gamma_1, \quad 
	(\boldsymbol{\sigma}- \sum_{j=1}^A \alpha_j p_j \boldsymbol{I}) \cdot \boldsymbol{n}=- \sum_{j=1}^A \alpha_j p_j  \boldsymbol{n} \quad \text { on } \Gamma_2 .
\end{align*}
%Assuming that the fluid in network 1 is directly connected to the surrounding cerebrospinal fluid in the nearby ventricles and subarachnoid space and that a cerebrospinal fluid pressure is prescribed. Specifically, we assume that the cerebrospinal fluid pressure pulsates around a baseline pressure of \(666.6 \, \text{Pa}\), with a peak transmantle pressure difference of \(1.6 \, \text{Pa}\),
All boundary pressure values are given in mmHg below; with 1 mmHg         $\approx $ 133.32 Pa. The boundary conditions of fluid in network 1 are
$$
p_1=5+2sin(2\pi t) \quad \text { on } \Gamma_1, \quad 
p_1=5+(2+\delta)sin(2\pi t) \quad \text { on } \Gamma_2 ,
$$
where $\delta =0.012$ is the transmantle pressure difference.
Consider that a pulsating arterial blood pressure is imposed at the outer boundary, while the inner boundary is subjected to zero arterial flow,
$$p_2=70+10 \sin (2 \pi t) \quad \text{on } \Gamma_1, \quad \boldsymbol{K}_2 \nabla p_2 \cdot \boldsymbol{n}=0 \quad \text{on } \Gamma_2.$$
A constant pressure is prescribed at both boundaries of the fluid in network 3,
$$p_3=6 \quad \text{on } \Gamma_1 \text{ and } \Gamma_2.$$
Finally, no flux at both boundaries of the network 4 (capillary compartment),
$$\boldsymbol{K}_4 \nabla p_4 \cdot \boldsymbol{n}=0 \quad \text{on } \Gamma_1 \text{ and } \Gamma_2.$$
The initial conditions at $t=0$ are specified as follows:
$$ \boldsymbol{u}=0,\quad p_1=5,\quad p_2= 70,\quad p_3=6, \quad p_4=38.$$
The numerical results are computed with a time step of \( \Delta t = 0.0125 \) for \( T = 3 \). We compare the results obtained by the parallel splitting scheme with those obtained by the fully coupled scheme, which uses the backward difference method for time discretization.
%This comparison verifies the stability and efficiency of the proposed algorithm compared to an established numerical approach. 
The results in \cref{fig:multiple_images} show the pressure distribution in a four-network model computed using two schemes. The first row depicts the pressure distribution from the fully coupled scheme, while the second row presents results from the parallel splitting scheme. The consistency in pressure profiles demonstrates the effectiveness of the proposed method.
%Furthermore, as shown in Table \ref{tab:time_comparison}, the parallel splitting scheme provides a significant computational advantage by dividing computations among parallel processors, thereby reducing computational time without sacrificing accuracy.

%	\begin{table}[htbp]
%		\centering
%		\caption{Comparison of computational time for different algorithms.}
%		\label{tab:time_comparison}
%		\begin{tabular}{lccc}
%			\hline
%			{Algorithm} & $\Delta t$ & {CPU Time (seconds)} & {Time Reduction} \\
%			\hline
%			Fully Implicit Scheme & 0.0125 & 755 & - \\
%			\hline
%			Sequential Scheme &0.0125&  350 & 53.6\% \\
%			\hline
%			parallel splitting Scheme & 0.0125& 214 & 71.7\% \\
%			\hline
%		\end{tabular}
%	\end{table}

\section{Conclusions}
\label{con}
This paper presents a novel algorithm that splits the MPET system into two independent subsystems, allowing for parallel computation at each time step. The approach significantly reduces computational complexity and improves efficiency. We provide an optimal convergence analysis of the parallel splitting algorithm. The algorithm is unconditionally stable, meaning that it imposes no restrictions on time-step size or model parameters after split decoupling. Numerical examples confirm the effectiveness of the algorithm, demonstrating that it is locking-free. Finally, we apply the parallel algorithm to simulate brain behavior, comparing the results with those of the fully coupled scheme, further validating its effectiveness and efficiency. {One limitation of the present scheme is the lack of local mass conservation, as standard Lagrangian elements are used for pressure approximation. This aspect could be improved by adopting mixed finite element methods or flux-reconstruction techniques in future studies.}

\section*{Acknowledgments}
We would like to express appreciation to the following financial support: National Key Research and Development Project of China (No. 2023YFA1011701).
%This work was supported by King Abdullah University of Science and Technology (KAUST) through the grants BAS/1/1351-01 and
%URF/1/5028-01. 
The work of J. Zhao was supported by the China Postdoctoral Science Foundation under Grant Number GZB20240552 and 2024M762396.
The work of H. Chen was supported by the National Key Research and Development Project of China (Grant No. 2023YFA1011702) and the National Natural Science Foundation of China (Grant No. 12471345, 12122115). 
The work of M. Cai was supported in part by the NIH-RCMI grant through U54MD013376, and the affiliated project award from the Center for Equitable Artificial Intelligence and Machine
Learning Systems (CEAMLS) at Morgan State University (project ID 02232301).

%\bibliographystyle{siamplain}
%\bibliography{references}

\end{document}